\documentclass[runningheads]{llncs}
\usepackage[T1]{fontenc}
\usepackage{graphicx}
\usepackage{comment}
\usepackage{siunitx}
\usepackage{relsize}
\usepackage{ifthen}
\usepackage[colorinlistoftodos]{todonotes}

\usepackage[vlined,ruled,noline]{algorithm2e}
\usepackage{graphics} 
\usepackage{rotating}
\usepackage{color}
\usepackage{enumerate}
\usepackage[T1]{fontenc}
\usepackage{psfrag}
\usepackage{epsfig} 
\usepackage[pagebackref=true,breaklinks=true,letterpaper=true,colorlinks,bookmarks=true]{hyperref}
\usepackage{booktabs}
\usepackage{graphicx,url}
\usepackage{multirow}
\usepackage{array}
\usepackage{latexsym}
\usepackage{amsfonts}
\usepackage{amsmath}
\usepackage{amssymb}
\usepackage{mathtools}
\usepackage{xstring}
\usepackage[noend]{algorithmic}
\usepackage{multirow}
\usepackage{xcolor}
\usepackage{prettyref}
\usepackage{flexisym}
\usepackage{bigdelim}
\usepackage{breqn} 
\usepackage{listings}

\usepackage{enumitem}
\usepackage{xspace}
\usepackage{bm}
\graphicspath{{./figures/}}
\usepackage{tikz}
\usetikzlibrary{matrix,calc}

\usepackage{mdwlist}

\makecompactlist{itemize}{stditemize}

\newcommand{\cf}{\emph{cf.}\xspace}

\newcommand{\bdmath}{\begin{dmath}}
\newcommand{\edmath}{\end{dmath}}
\newcommand{\beq}{\begin{equation}}
\newcommand{\eeq}{\end{equation}}
\newcommand{\bdm}{\begin{displaymath}}
\newcommand{\edm}{\end{displaymath}}
\newcommand{\bea}{\begin{eqnarray}}
\newcommand{\eea}{\end{eqnarray}}
\newcommand{\beal}{\beq \begin{array}{ll}}
\newcommand{\eeal}{\end{array} \eeq}
\newcommand{\beas}{\begin{eqnarray*}}
\newcommand{\eeas}{\end{eqnarray*}}
\newcommand{\ba}{\begin{array}}
\newcommand{\ea}{\end{array}}
\newcommand{\bit}{\begin{itemize}}
\newcommand{\eit}{\end{itemize}}
\newcommand{\ben}{\begin{enumerate}}
\newcommand{\een}{\end{enumerate}}

\newcommand{\calA}{{\cal A}}

\newcommand{\calC}{{\cal C}}

\newcommand{\calG}{{\cal G}}
\newcommand{\calH}{{\cal H}}

\newcommand{\calN}{{\cal N}}

\newcommand{\etal}{\emph{et~al.}\xspace}

\newcommand{\eg}{\emph{e.g.,}\xspace}
\newcommand{\ie}{\emph{i.e.,}\xspace}

\newcommand{\hide}[1]{}

\newcommand{\hiddenText}{{\color{gray} hidden text.}}
\newcommand{\hideWithText}[1]{\hiddenText}

\newcommand{\subject}{\text{ subject to }}

\newcommand{\norm}[1]{\left\| #1 \right\|}

\newcommand{\tran}{^{\mathsf{T}}}

\newcommand{\trace}[1]{\mathrm{tr}\left(#1\right)}

\newcommand{\rank}[1]{\mathrm{rank}\left(#1\right)}

\newcommand{\Real}[1]{ { {\mathbb R}^{#1} } }
\newcommand{\reals}{\Real{}}

\newcommand{\SOtwo}{\ensuremath{\mathrm{SO}(2)}\xspace}

\newcommand{\scenario}[1]{{\smaller \sf#1}\xspace}

\newcommand{\blue}[1]{{\color{blue}#1}}

\newcommand{\red}[1]{{\color{red}#1}}

\newcommand{\linkToPdf}[1]{\href{#1}{\blue{(pdf)}}}
\newcommand{\linkToPpt}[1]{\href{#1}{\blue{(ppt)}}}
\newcommand{\linkToCode}[1]{\href{#1}{\blue{(code)}}}
\newcommand{\linkToWeb}[1]{\href{#1}{\blue{(web)}}}
\newcommand{\linkToVideo}[1]{\href{#1}{\blue{(video)}}}
\newcommand{\linkToMedia}[1]{\href{#1}{\blue{(media)}}}
\newcommand{\award}[1]{\xspace} 

\newcommand{\shucheng}[1]{#1}

\newcommand{\bbN}{\mathbb{N}}

\renewcommand{\norm}[1]{\left\lVert #1 \right\rVert}
\newcommand{\inprod}[2]{\left\langle #1, #2 \right\rangle}

\newcommand{\mymid}{:}

\newcommand{\bmat}{\left[ \begin{array}}
\newcommand{\emat}{\end{array}\right]}

\newcommand{\poly}[1]{\mathbb{R}[#1]}

\newcommand{\ceil}[1]{\left\lceil #1 \right\rceil}

\newcommand{\nameshort}{\scenario{STROM}}

\newcommand{\rc}[1]{r_{c, #1}}
\newcommand{\rs}[1]{r_{s, #1}}
\newcommand{\fc}[1]{f_{c, #1}}
\newcommand{\fs}[1]{f_{s, #1}}
\newcommand{\rx}[1]{r_{x, #1}}
\newcommand{\ry}[1]{r_{y, #1}}
\newcommand{\vx}[1]{v_{x, #1}}
\newcommand{\vy}[1]{v_{y, #1}}
\newcommand{\dt}{\Delta t}
\newcommand{\seqordering}[1]{\left[ #1 \right]}
\newcommand{\dvar}{z}
\newcommand{\tms}{\varphi}
\newcommand{\degf}[1]{d^f_{#1}}
\newcommand{\degg}[1]{d^g_{#1}}
\newcommand{\degh}[1]{d^h_{#1}}
\newcommand{\basis}[2]{\left[ #1 \right]_{#2}}
\newcommand{\symm}[1]{\mathbb{S}^{#1}}
\newcommand{\symmp}[1]{\mathbb{S}^{#1}_{+}}
\newcommand{\cartsymm}{\Omega}
\newcommand{\cartsymmp}{\Omega_+}
\newcommand{\subopt}{\xi}
\newcommand{\xf}{x_f}
\newcommand{\mosek}{\textsc{mosek}\xspace}
\newcommand{\tssos}{\textsc{tssos}\xspace}
\newcommand{\sostools}{\textsc{sostools}\xspace}
\newcommand{\cuadmm}{\scenario{cuADMM}}
\newcommand{\sgsadmm}{\scenario{sGS-ADMM}}
\newcommand{\cdcs}{\textsc{cdcs}\xspace}
\newcommand{\scs}{\textsc{scs}\xspace}
\newcommand{\yalmip}{\textsc{yalmip}\xspace}
\newcommand{\sdpnal}{\textsc{sdpnal+}\xspace}
\newcommand{\fmincon}{\textsc{fmincon}\xspace}
\newcommand{\sedumi}{\textsc{sedumi}\xspace}
\newcommand{\init}{\text{init}}
\newcommand{\myendexample}{\hfill$\blacktriangle$}

\usepackage{multicol}

\setlist[enumerate]{topsep=0pt}
\newcolumntype{E}{>{\displaystyle}r@{}l@{}l}

\begin{document}
\title{Fast and Certifiable Trajectory Optimization \vspace{-4mm}}

%
%
\author{{Shucheng Kang}\inst{1} \and
{Xiaoyang Xu}\inst{2} \and
{Jay Sarva}\inst{3} \and
{Ling Liang}\inst{4} \and
{Heng Yang}\inst{1}
}
\authorrunning{S. Kang, X. Xu, J. Sarva, L. Liang, H. Yang}
%
\institute{Harvard University \and
University of California at Santa Barbara  \and
Brown University \and
University of Maryland at College Park
}

\maketitle              

\begin{center}
  \vspace{-6mm}
 \url{https://computationalrobotics.seas.harvard.edu/project-strom}
\end{center}


\vspace{-8mm}
\begin{abstract}
    We propose semidefinite trajectory optimization (\nameshort), a framework that computes fast and certifiably optimal solutions for nonconvex trajectory optimization problems defined by polynomial objectives and constraints. \nameshort employs sparse second-order Lasserre's hierarchy to generate semidefinite program (SDP) relaxations of trajectory optimization. Different from existing tools (\eg \yalmip~and~\sostools in Matlab), \nameshort generates chain-like multiple-block SDPs with only positive semidefinite (PSD) variables. {Moreover, \nameshort does so {two orders of magnitude faster}}. Underpinning \nameshort is \cuadmm, the first ADMM-based SDP solver implemented in CUDA \shucheng{(with C/C++ extension)} and runs in GPUs. \cuadmm builds upon the symmetric Gauss-Seidel ADMM algorithm and leverages GPU parallelization to speedup solving sparse linear systems and projecting onto PSD cones. In five trajectory optimization problems (inverted pendulum, cart-pole, vehicle landing, flying robot, and
    car back-in), \cuadmm computes optimal trajectories (with certified suboptimality below $1\%$) in minutes (when other solvers take hours
    or run out of memory) and seconds (when others take minutes). Further, when warmstarted by data-driven initialization in the inverted pendulum problem, \cuadmm delivers real-time performance: providing certifiably optimal trajectories in $0.66$ seconds despite the SDP has $49,500$ variables and $47,351$ constraints.


\end{abstract}
\vspace{-8mm}

\section{Introduction}
\label{sec:introduction}

{Trajectory optimization}~\cite{bryson18book-applied} designs dynamical system trajectories by optimizing a performance measure subject to constraints, finding extensive applications in motion planning of robotic~\cite{posa2014ijrr-traopt-directmethod-contact}, aerospace~\cite{malyuta2022ieee-convexopt-trajectorygeneration}, and manufacturing systems~\cite{eaton92ces-model}.

{\bf Problem Statement}. Let $N$ be the number of steps (with $[N]:=\{1,\dots,N\}$), $\{ x_k\}_{k=0}^{N} \subset \Real{d_x}$ be the state trajectory, and $\{ u_k \}_{k=0}^{N-1} \subset \Real{d_u}$ be the control trajectory, we consider the following trajectory optimization problem:
    \begin{subequations}\label{eq:intro:trajopt}
        \begin{eqnarray}
            \label{eq:intro:trajopt-generalform}
            \min_{\{ u_k \}_{k=0}^{N-1}, \{ x_k \}_{k=0}^{N}} & \displaystyle l_N(x_N) + \sum_{k=0}^{N-1} l_k(x_k, u_k) \\
            \subject & x_0 = x_{\text{init}} \\
            & F_k(x_{k-1}, u_{k-1}, x_{k}) = 0, \ \forall k \in [N] \label{eq:intro:trajopt-generalform-dyn} \\
            & (u_{k-1}, x_k) \in \calC_k, \ \forall k \in [N] \label{eq:intro:trajopt-generalform-con}
        \end{eqnarray}
    \end{subequations}
where $l_k,k=0,\dots,N$ are the instantaneous and terminal loss functions; $x_{\text{init}}$ is the initial state; $F_k$ represents the discretized system dynamics in the form of a differential algebraic equation (\eg obtained from the continuous-time dynamics via multiple shooting, \cf Example~\ref{example:strom:popsdp:momentrelax-example} and \S\ref{sec:exp}); and $\calC_k$ imposes constraints on $u_{k-1}$ and $x_k$ (\eg control limits, obstacle avoidance). Trajectory optimization computes open-loop control; when paired with receding horizon control (\ie execute only part of the optimal control sequence and repeatedly solve~\eqref{eq:intro:trajopt})~\cite{mayne88cdc-receding}, leads to closed-loop control with implicit feedback known as \emph{model predictive control} (MPC)~\cite{borrelli17book-mpc}. 
In the case of linear system dynamics, (convex) quadratic losses, and polytopic sets, problem~\eqref{eq:intro:trajopt} reduces to a quadratic program, \ie constrained linear quadratic regulator (LQR)~\cite{borrelli17book-mpc}. 

In this paper, we assume $l_k$ and $F_k$ are polynomial functions and $\calC_k$ are basic semialgebraic sets (\ie described by polynomial constraints), in which case problem~\eqref{eq:intro:trajopt} is an instance of \emph{polynomial optimization} (POP) that is nonconvex and NP-hard in general. We briefly review solution methods for problem \eqref{eq:intro:trajopt}.

\textbf{Local Solver}. 
Most efforts over the past decades have focused on developing solvers that efficiently find local solutions, such as iterative LQR~\cite{li04-iLQR}, differential dynamic programming~\cite{tassa14icra-control}, sequential quadratic programming~\cite{posa2014ijrr-traopt-directmethod-contact} and other nonlinear programming algorithms~\cite{hargraves1987jgcd-trajectoryoptimization-nlp,andersson19mpc-casadi,howell2019iros-altro}. Recent efforts design local solvers on GPUs~\cite{adabag2024arxiv-mpcgpu}, embedded systems~\cite{alavilli23-tinympc,aydinoglu2023icra-realtime-multicontact-mpc-admm}, and make them end-to-end differentiable~\cite{amos18neurips-differentiable}. Despite success in numerous applications, local solvers can get stuck in bad local minima and heavily rely on high-quality initial guesses, which require significant engineering heuristics~\cite{romero2022tro-mpcc-timeoptimal-quadrotorflight} and can be difficult to obtain~\cite{wensing2023tro-optimizationbased-control-legged}. 

\textbf{Global Solver}. 
Under restrictive conditions (\eg global optimality attained in the convex hull of the feasible set), problem~\eqref{eq:intro:trajopt} can be equivalently solved as a convex optimization problem, known as lossless convexification~\cite{blackmore2012scl-lossless-convexification-nonlinearopt,malyuta2022ieee-convexopt-trajectorygeneration}. Further, if the only nonconvexity is combinatorial and can be modelled by integer variables (\eg in graph of convex neighbors/sets~\cite{gentilini13oms-travelling,marcucci24siopt-gcs}), then problem~\eqref{eq:intro:trajopt} can be solved by off-the-shelf mixed-integer programming solvers~\cite{deits2015icra-mixedintegerprogramming-uav,marcucci2020arxiv-warmstart-mixedinteger-mpc,marcucci23science-motion}, albeit the runtime is worst-case exponential. We focus on generic nonconvex trajectory optimization whose nonconvexity is not combinatorial and cannot be losslessly convexified.

\textbf{Certifiable Solver}. 
A generic recipe to design polynomial-time algorithms for solving nonconvex optimization while offering strong performance guarantees is through \emph{convex relaxation}, where a relaxed convex problem is solved to provide a \emph{lower bound} and a feasible solution of the nonconvex problem provides an \emph{upper bound}. The relative error between the lower and upper bounds provides a \emph{certificate of (sub)optimality} (to be made precise in \eqref{eq:suboptimality}). Lasserre's moment and sums-of-squares (SOS) hierarchy~\cite{lasserre2001siopt-global} --relaxing a polynomial optimization as a hierarchy of convex \emph{semidefinite programs} (SDPs) of increasing size-- is arguably the method of choice for designing convex relaxations, as it guarantees a certificate of suboptimality that converges to zero (then the relaxation is called \emph{tight} or \emph{exact}). Khadir \etal~\cite{khadir2021icra-piecewiselinear-motionplanning} applied the moment-SOS hierarchy to compute shortest piece-wise linear paths among obstacles without considering robot dynamics. Teng \etal~\cite{teng2023arxiv-geometricmotionplanning-liegroup} explored the Markovian property of dynamical systems and applied a \emph{sparse} variant~\cite{lasserre2006msc-correlativesparse,wang2022tms-cs-tssos} of the moment-SOS hierarchy to trajectory optimization and observed the \emph{second-order} SDP relaxation is empirically tight, echoing similar findings in perception~\cite{yang2022pami-outlierrobust-geometricperception}. Huang \etal~\cite{huang2024arxiv-sparsehomogenization} proposed sparse SDP relaxations with homogenization to allow unbounded feasible sets and solved simple trajectory optimization problems with \emph{tight} \emph{second-order} relaxation.  

\textbf{Computational Challenges}. Despite the ability to compute \emph{certifiably optimal} trajectories, the second-order (and higher) moment-SOS relaxation is notoriously expensive to solve, rendering its practicality in robotics questionable.\footnote{Tackling the second-order relaxation is necessary, because the first-order relaxation, unfortunately, is very loose in trajectory optimization, see~\cite{teng2023arxiv-geometricmotionplanning-liegroup} and \S\ref{app:sec:firstorder}.} Indeed, in~\cite{khadir2021icra-piecewiselinear-motionplanning,teng2023arxiv-geometricmotionplanning-liegroup,huang2024arxiv-sparsehomogenization}, 
the commercial solver \mosek~\cite{aps2019ugrm-mosek-sdpsolver} is chosen to solve SDPs due to its robustness and high accuracy. However, as a generic-purpose implementation of the interior point method~\cite{helmberg1996siam-interiorpoint-sdp}, \mosek~has three drawbacks. (\emph{i}) \mosek~has poor scalability due to high memory complexity and per-iteration time complexity. Even for small-scale problems like minimum-work block-moving~\cite{huang2024arxiv-sparsehomogenization} and inverted pendulum (\cf \S\ref{sec:exp}), \mosek's runtime is around 10 seconds. As problems get larger (\cf examples in \cite{teng2023arxiv-geometricmotionplanning-liegroup} and \S\ref{sec:exp}), it easily takes runtime in the order of hours and goes out of memory. (\emph{ii}) \mosek~cannot be warmstarted, which means its runtime online cannot be reduced even if many similar problems have been solved offline (\eg in an MPC setup~\cite{aydinoglu2023icra-realtime-multicontact-mpc-admm}). 
(\emph{iii}) It does not fully exploit the special structure in SDPs generated from sparse moment-SOS relaxations. 


\emph{Can the moment-SOS hierarchy ever be practical for trajectory optimization?}

We believe the answer is affirmative if and only if one can design a customized SDP solver that is scalable, exploits problem structures, and can be warmstarted.


\textbf{Contributions}. We present \nameshort (\underline{s}emidefinite \underline{tr}ajectory \underline{o}pti\underline{m}ization), a fast and certifiable trajectory optimization framework that checks the merits. 
\begin{enumerate}[label=(\Roman*)]
    \item \textbf{Faster sparse moment relaxation}. 
    Unlike existing tools in Matlab such as \sostools~\cite{prajna02cdc-sostools} and \yalmip~\cite{lofberg2004cacsd-yalmip} that
    generate sparse second-order SDP relaxations of problem~\eqref{eq:intro:trajopt} from the dual SOS perspective, we develop a C++ tool to generate SDP relaxations from the \emph{primal moment} perspective. Our POP-SDP conversion reveals the special \emph{chain-like sparsity pattern} in the relaxed multiple-cone SDP (\cf Fig.~\ref{fig:strom:popsdp:pop-sdp-conversion}). It is two orders of magnitude faster than \yalmip~and \sostools, and on par with the Julia package \tssos~\cite{wang2021siam-tssos}. 
    \item \textbf{Algorithmic backbone: \sgsadmm}. 
    We choose the \emph{symmetric Gauss-Seidel ADMM} (\sgsadmm)~\cite{chen2017mp-sgsadmm} as the algorithmic backbone to design a scalable SDP solver that can be warmstarted. As a first-order method~\cite{wen2010mp-admmsdp}, \sgsadmm inherits low per-iteration cost from ADMM 
    and offers empirical advantages in solving degenerate SDPs from high-order relaxations~\cite{yang2023mp-stride}. 
    \item \textbf{GPU speedup: \cuadmm}. At each iteration, \sgsadmm alternates between (a) performing projection onto positive semidefinite (PSD) cones, and (b) solving sparse linear systems. Due to the existence of many small-to-medium-scale PSD cones in the SDP relaxation,
    we engineer a highly optimized GPU-based implementation of \sgsadmm in CUDA, named \cuadmm and achieves up to {$10\times$} speedup compared to existing ADMM solvers such as \cdcs~\cite{zheng2017ifac-cdcs-sdpsolver} and \sdpnal~\cite{yang2015mp-sdpnalplus-sdpsolver}. \cuadmm is the first ADMM-based SDP solver that runs in GPUs and we demonstrate its ability to solve five trajectory optimization applications (inverted pendulum, cart-pole, vehicle landing, flying robot, and car back-in with obstacle avoidance): \cuadmm solves them with optimality certificates (suboptimality below $1\%$) in {minutes (when other solvers take hours and run out of memory) and seconds (when others take minutes)}.
    \item \textbf{Data-driven initialization: real-time certifiable optimization}. To show the potential to warmstart \cuadmm, we perform a case study in inverted pendulum: using a vanilla $k$-nearest neighbor search to initialize \cuadmm. The result is the first certifiably optimal swing-up of the pendulum computed in subseconds (\ie $0.66$ seconds with a below $1\%$ optimality certificate).
\end{enumerate}

{\bf Organization}. We present sparse moment relaxation in \S\ref{sec:strom:popsdp}, \sgsadmm in \S\ref{sec:strom:sgsadmm}, and \cuadmm in \S\ref{sec:strom:gpu}. We give numerical experiments in \S\ref{sec:exp} and conclude in \S\ref{sec:conclusion}.

{\bf Notation}. Given $\dvar=(\dvar_1,\dots,\dvar_d)$, a monomial is defined as $\dvar^\alpha:= \dvar_1^{\alpha_1} \cdot \cdots \dvar_d^{\alpha_d}$ for $\alpha \in \bbN^d$. The degree of a monomial is $\sum_{i=1}^d \alpha_i$. A polynomial in $\dvar$ is written as $f(\dvar):=\sum_{\alpha \in \bbN^d} c_\alpha \dvar^\alpha$ with real coefficients $c_\alpha$ and its degree, $\deg(f)$, is the maximum degree of the monomials. Let $\poly{\dvar}$ be the ring of polynomials, and $\poly{\dvar}_{n}$ be those with degree at most $n$. Denoting $[\dvar]_n$ as the vector of monomials of degree up to $n$, then $\poly{\dvar}_n$ can be identified as a vector space of dimension $s(d,n) = \left( \substack{n+d \\ d} \right)$ with a set of basis $[\dvar]_n$. Given an index set $I \subset [d]$, let $\poly{\dvar(I)}_n$ be the set of polynomials in the subset of variables $\dvar(I)$ of degree up to $n$. Denote $\symm{n}$ (resp. $\symmp{n}$) as the set of symmetric (resp. positive semidefinite) $n \times n$ matrices.

\section{Sparse Moment Relaxation}
\label{sec:strom:popsdp}

We first introduce a special type of sparsity in polynomial optimization known as \emph{chain-like} sparsity and show that problem~\eqref{eq:intro:trajopt} satisfies this pattern (\S\ref{sec:relax:sparse-pattern}). We then present the sparse moment-SOS hierarchy (\S\ref{sec:relax:hierarchy}), followed by how to convert it to a standard semidefinite program (SDP) (\S\ref{sec:relax:SDP}). 

To provide a tutorial-style exposition of the mathematical machinery, we use a toy trajectory optimization problem as the running example. 

\begin{example}[Trajectory Optimization of A 1-D Nonlinear System]
    \label{example:strom:popsdp:momentrelax-example}
    Consider the nonlinear dynamical system adapted from~\cite{slotine91book-applied}:
    \begin{align}
        \label{eq:strom:popsdp:momentrelax-example-dyn}
        \dot{x} = -(1 + u) x, \quad u \in [-1, 1].
    \end{align}
    Starting from $x_{\init} = 2$, our goal is to regulate~\eqref{eq:strom:popsdp:momentrelax-example-dyn} towards $x = 0$. We define a trajectory optimization problem with standard direct multiple shooting:
    \begin{subequations}\label{eq:strom:popsdp:momentrelax-example-trajopt}
        \begin{eqnarray}
            \min_{x_0,\dots,x_N;u_0,\dots,u_{N-1}} & \displaystyle P_f \cdot x_N^2 + \sum_{k=0}^{N-1} (u_k^2 + x_k^2) \\
            \subject & x_0 = x_{\init} \\
            & x_k = x_{k-1} - \dt \cdot (1 + u_{k-1}) x_{k-1}, \ k \in \seqordering{N} \\
            & 1 - u_{k-1}^2 \ge 0, \ k \in \seqordering{N}
        \end{eqnarray}
    \end{subequations}
    where $P_f>0$ is the terminal loss coefficient and $\Delta t$ is the time step size.\myendexample

    \subsection{Polynomial Optimization with Chain-like Sparsity}
    \label{sec:relax:sparse-pattern}

    A chain-like sparsity pattern is a special case of the correlative sparsity pattern~\cite{lasserre2006msc-correlativesparse,wang2022tms-cs-tssos}. Correlative sparsity corresponds to the variables of a POP forming a chordal graph~\cite{magron23book-sparse}; chain-like sparsity corresponds to a chain (or line) graph. 
    

    \begin{definition}[POP with Chain-like Sparsity]
        \label{def:pop-chain}
        Let $\dvar \in \Real{d}$ and $I_1,\dots,I_N \subset [d]$ be $N$ index sets such that 
        \begin{equation}\label{eq:strom:popsdp:index-chainrule}
            \bigcup_{k=1}^N I_k = [d] \quad \text{and} \quad 
            \left( \bigcup_{j=1}^{k-1} I_j \right) \bigcap I_k = I_{k-1} \bigcap I_k, \ \forall k \in \left\{ 2, \dots, N \right\},
        \end{equation}
        then the following POP is said to admit a chain-like sparsity pattern
        \begin{subequations}\label{eq:strom:popsdp:chain-sparse-pop}
            \begin{eqnarray}
                p^{\star} = \min_{\dvar \in \Real{d}} & \displaystyle \sum_{k=1}^{N} f_k(\dvar(I_k)) \label{eq:strom:popsdp:chain-sparse-pop-obj} \\
                \subject & g_{k, i}(\dvar(I_k)) \ge 0, \ \forall k \in \seqordering{N}, i \in \calG_k \\
                & h_{k, j}(\dvar(I_k)) = 0, \ \forall k \in \seqordering{N}, j \in \calH_k 
            \end{eqnarray}
        \end{subequations}
        where 
        $\calG_k$ (resp. $\calH_k$) indexes the inequality (resp. equality) constraints for variables $\dvar(I_k)$, and $f_k,g_{k,i},h_{h,j} \in \poly{\dvar(I_k)}$ are polynomials in the variables $\dvar(I_k)$.
    \end{definition}

In words, every polynomial constraint involves only a subset of the variables, known as a \emph{clique}, the objective can be decomposed as a sum of polynomials each involving a single clique, and the cliques satisfy the sparsity pattern \eqref{eq:strom:popsdp:index-chainrule}.

Observe that the toy problem~\eqref{eq:strom:popsdp:momentrelax-example-trajopt} in Example~\ref{example:strom:popsdp:momentrelax-example} satisfies the chain-like sparsity pattern. Indeed, set $z:=[x_0;u_0;\dots;u_{N-1};x_N]$, $I_k:=\{ 2k-1,2k,2k+1\}$ such that $\dvar(I_k)=(x_{k-1},u_{k-1},x_k)$, then~\eqref{eq:strom:popsdp:momentrelax-example-trajopt} is an instance of \eqref{eq:strom:popsdp:chain-sparse-pop} with 
    \begin{subequations}\label{eq:toy-as-sparse-pop}
        \begin{align}
            \hspace{-8mm} f_k(\dvar(I_k)) & = u_{k-1}^2 + x_{k-1}^2, k \in \seqordering{N-1}, 
            f_N(\dvar(I_N)) = u_{N-1}^2 + x_{N-1}^2 + P_f \cdot x_N^2 \\
            \hspace{-8mm} g_{k, 1}(\dvar(I_k)) & = 1 - u_{k-1}^2, \quad \calG_k = \left\{ 1 \right\}, \ k \in \seqordering{N} \\
            \hspace{-8mm} h_{k, 1}(\dvar(I_k)) & = x_k + (\dt - 1) \cdot x_{k-1} + \dt \cdot u_{k-1} x_{k-1}, \quad k \in \seqordering{N} \\
            \hspace{-8mm} h_{1, 2}(\dvar(I_1)) & = x_0 - 2, \quad \calH_1 = \left\{ 1, 2 \right\}, \calH_k = \left\{ 1 \right\}, \ k \in \left\{ 2,\dots,N \right\}. 
        \end{align}
    \end{subequations}
\end{example}

\vspace{-9mm}
\begin{figure}[htbp]
    \begin{center}
        \includegraphics[width=\textwidth]{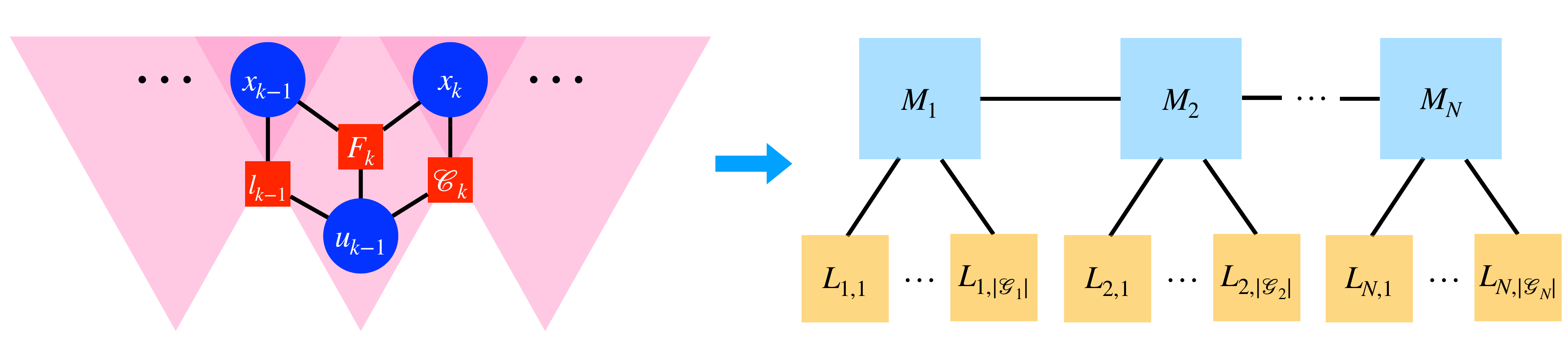}
    \end{center}
    \vspace{-6mm}
    \caption{Illustration of sparse moment relaxation. Left: POP with a chain-like sparsity pattern. Right: standard multi-block SDP.
    \label{fig:strom:popsdp:pop-sdp-conversion}}
\end{figure}
\vspace{-6mm}

\shucheng{Similarly, for the generic trajectory optimization problem~\eqref{eq:intro:trajopt}, Fig.~\ref{fig:strom:popsdp:pop-sdp-conversion} shows its factor graph represetation~\cite{dellaert2021arcras-factor-graph-robotics}. Each clique contains variables $\dvar(I_k) = \{x_{k-1},u_{k-1},x_k\} \in \Real{2d_x + d_u}$. These cliques form a chain-like graph. Each dynamics constraint $F_k$~\eqref{eq:intro:trajopt-generalform-dyn} and set constraint $\calC_k$~\eqref{eq:intro:trajopt-generalform-con} only involves $k$-th clique's variables, and the objective function~\eqref{eq:intro:trajopt-generalform} readily admits the decomposition in the form of~\eqref{eq:strom:popsdp:chain-sparse-pop-obj}.}


\subsection{The Sparse Moment-SOS Hierarchy}
\label{sec:relax:hierarchy}

Associated with every vector space (of points) is its dual vector space of linear functionals. Viewing $\poly{\dvar}_n$ as a vector space (which is isomorphic to $\Real{s(d,n)}$ after fixing the basis $[\dvar]_n$), its dual vector space $\poly{\dvar}_n^* \cong \Real{s(d,n)}$ contains linear functionals of polynomials. Particularly, given a sequence of numbers $\tms = (\tms_\alpha) \in \Real{s(d,n)}$ indexed by the (exponents of) monomials in $[\dvar]_n$, 
we define a one-to-one linear map from $\poly{\dvar}_{n}$ to $\reals$, known as the \emph{Riesz linear functional}
\begin{equation}
    \label{eq:strom:popsdp:linearfunctional}
    \ell_{\tms}: \ f(\dvar) = \sum_{\alpha} c_\alpha \dvar^\alpha \ \mapsto \ \sum_{\alpha} c_\alpha \tms_\alpha, \ \forall f(\dvar) \in \reals[\dvar]_{n}.
\end{equation}
We can easily extend the notation of $\ell_\tms$ from polynomials to polynomial vectors and matrices, as shown in the following example.
\begin{example}[Riesz Functional]
    Let $d = 3, I_1 = \left\{ 2,3 \right\}$. Then, $\basis{\dvar}{1} = [1; z_1; z_2; z_3]$ and $\basis{\dvar(I_1)}{2} = [1; z_2; z_3; z_2^2; z_2 z_3; z_3^2]$. Setting $n=3$, we have
    \begin{subequations}
         \begin{align}
            \hspace{-4mm} \ell_\tms\left( (z_2 - z_3) \cdot \basis{\dvar(I_1)}{2} \right) & = \ell_\tms\left(\begin{bmatrix}
            \dvar_2 - \dvar_3 \\
            \dvar_2^2 - \dvar_2 \dvar_3 \\
            \dvar_2\dvar_3 - \dvar_3^2 \\
            \dvar_2^3 - \dvar_2^2 \dvar_3 \\
            \dvar_2^2 \dvar_3 - \dvar_2 \dvar_3^2 \\
            \dvar_2\dvar_3^2 - \dvar_3^3
            \end{bmatrix}\right)
                =
                \begin{bmatrix}
                \tms_{0,1,0} - \tms_{0,0,1} \\
                \tms_{0,2,0} - \tms_{0,1,1} \\
                \tms_{0,1,1} - \tms_{0,0,2} \\
                \tms_{0,3,0} - \tms_{0,2,1} \\
                \tms_{0,2,1} - \tms_{0,1,2} \\
                \tms_{0,1,2} - \tms_{0,0,3}
            \end{bmatrix} \nonumber\\
            \hspace{-4mm} \ell_\tms\left( \basis{\dvar}{1} \basis{\dvar}{1}\tran \right)
            & = \ell_\tms\left( 
                \begin{bmatrix}
                    1 & z_1 & z_2 & z_3 \\
                    z_1 & z_1^2 & z_1 z_2 & z_1 z_3 \\
                    z_2 & z_2 z_1 & z_2^2 & z_2 z_3 \\
                    z_3 & z_3 z_1 & z_3 z_2 & z_3^2
                \end{bmatrix}
            \right) =
            \begin{bmatrix}
                \tms_{0,0,0} & \tms_{1,0,0} & \tms_{0,1,0} & \tms_{0,0,1} \\
                \tms_{1,0,0} & \tms_{2,0,0} & \tms_{1,1,0} & \tms_{1,0,1} \\
                \tms_{0,1,0} & \tms_{1,1,0} & \tms_{0,2,0} & \tms_{0,1,1} \\
                \tms_{0,0,1} & \tms_{1,0,1} & \tms_{0,1,1} & \tms_{0,0,2}
            \end{bmatrix} \nonumber
        \end{align}
    \end{subequations}
    where the application of $\ell_\tms$ is element-wise. \myendexample
\end{example}

We are ready to present the sparse moment-SOS hierarchy for the POP~\eqref{eq:strom:popsdp:chain-sparse-pop}.

\begin{proposition}[Sparse Moment Relaxation]
    In~\eqref{eq:strom:popsdp:chain-sparse-pop}, denote $\degf{k}:=\deg(f_k)$, $\degg{k,i}:=\ceil{\deg(g_{k,i})/2}$, and $\degh{k,j}:=\deg(h_{k,j})$. Given a positive integer $\kappa$ such that
    \begin{align}
        \label{eq:strom:popsdp:kappa-lowerbound}
        2\kappa \ge 2\kappa_0 = \max\{ \{ \degf{k} \}_{k \in \seqordering{N}}, \ \{ 2\degg{k,i} \}_{k \in \seqordering{N}, i \in \calG_k}, \ \{ \degh{k,j} \}_{k \in \seqordering{N}, j \in \calH_k} \},
    \end{align} 
    the $\kappa$-th order sparse moment relaxation for problem~\eqref{eq:strom:popsdp:chain-sparse-pop} reads 
        \begin{subequations}
            \label{eq:strom:popsdp:sparse-momentrelax}
            \begin{eqnarray}
                p_\kappa^\star = \min_{\tms \in \Real{s(d, 2\kappa)}} & \displaystyle \sum_{k=1}^{N} \ell_\tms \left( f_k(\dvar(I_k)) \right) \label{eq:strom:popsdp:sparse-momentrelax-objective} \\
                \subject & \tms_0 = 1, \quad \text{and} \quad \forall k \in \seqordering{N}, i \in \calG_k, j \in \calH_k: \label{eq:strom:popsdp:sparse-momentrelax-normalize} \\
                & M_k := \ell_\tms\left( \basis{\dvar(I_k)}{\kappa} \basis{\dvar(I_k)}{\kappa}\tran \right) \succeq 0, \label{eq:strom:popsdp:sparse-momentrelax-moment} \\
                & L_{k,i} := \ell_\tms \left( g_{k,i}(\dvar(I_k)) \cdot \basis{\dvar(I_k)}{\kappa - \degg{k,i}} \basis{\dvar(I_k)}{\kappa - \degg{k,i}}\tran \right) \succeq 0, \label{eq:strom:popsdp:sparse-momentrelax-inequality} \\
                & \ell_\varphi \left( h_{k,j}(\dvar(I_k)) \cdot \basis{\dvar(I_k)}{2\kappa - \degh{k,j}} \right) = 0. \label{eq:strom:popsdp:sparse-momentrelax-equality}
            \end{eqnarray}
        \end{subequations}
\end{proposition}

Problem~\eqref{eq:strom:popsdp:sparse-momentrelax} is a convex optimization whose variable is the sequence $\tms$. The sequence $\tms$ is called the \emph{truncated moment sequence} (TMS) because it can be interpreted as the moments 
of a probability measure supported on the feasible set of the POP~\eqref{eq:strom:popsdp:chain-sparse-pop}. Indeed, the nonconvex POP is equivalent to an infinite-dimensional linear optimization over the space of probability measures~\cite{lasserre2001siopt-global}, and the TMS $\tms$ is a finite-dimensional relaxation of the measure. With this perspective, the constraints of~\eqref{eq:strom:popsdp:sparse-momentrelax} are \emph{necessary} conditions for $\tms$ to admit a \emph{representing measure}, \ie $\tms_\alpha = \int \dvar^\alpha d\mu$ for some measure $\mu$. (\emph{i}) Its zero-order moment equals to $1$ (\cf~\eqref{eq:strom:popsdp:sparse-momentrelax-normalize}); (\emph{ii}) the \emph{moment matrix} $M_k \in \symm{s(|I_k|, \kappa)}$ and the \emph{localizing matrix} $L_{k,i} \in \symm{s(|I_k|, \kappa - \degg{k,i})}$ must be positive semidefinite (\cf~\eqref{eq:strom:popsdp:sparse-momentrelax-moment} and~\eqref{eq:strom:popsdp:sparse-momentrelax-inequality}, due to the fact the polynomial matrices inside $\ell_{\tms}(\cdot)$ are PSD for any $z$ feasible for the POP); and (\emph{iii}) the \emph{localizing vector} must vanish (\cf~\eqref{eq:strom:popsdp:sparse-momentrelax-equality}, due to the polynomial vector inside $\ell_{\tms}(\cdot)$ vanishes for any $z$ feasible for the POP). 

The power of the sparse moment relaxation~\eqref{eq:strom:popsdp:sparse-momentrelax} is twofold. First, every relaxation generates a lower bound $p^\star_\kappa \leq p^\star$ and the lower bound increasingly converges to $p^\star$. Second, the convergence can be detected and certified, precisely, due to \emph{sufficient} conditions on $\tms$ admitting a representing measure.

\vspace{-1mm}
\begin{theorem}[Convergence of Sparse Moment Relaxations~\cite{lasserre2006msc-correlativesparse}]
    \label{thm:strom:popsdp:convergence-sparsemoment}
    Suppose $\forall k \in \seqordering{N}$, $\exists R_k > 0$ such that $\norm{\dvar(I_k)}_\infty \le R_k$. Then:
    \begin{enumerate}[label=(\roman*)]
        \item Problem~\eqref{eq:strom:popsdp:sparse-momentrelax} is \emph{solvable} for any $\kappa \geq \kappa_0$. Let $M_{k}^\star$ be an optimal solution. 
        \item $p_\kappa^\star \le p^\star, \forall \kappa \ge \kappa_0$. Moreover, $p^\star_\kappa \rightarrow p^\star$, as $\kappa \rightarrow \infty$.
        \item \label{item:rank} If under some relaxation order $\kappa$, it holds that $\rank{M_{k}^\star} = 1, \forall k \in \seqordering{N}$. Then $p^\star_\kappa = p^\star$, \ie the relaxation is tight.
    \end{enumerate}
\end{theorem}
\vspace{-1mm}
Theorem~\ref{thm:strom:popsdp:convergence-sparsemoment} provides a way to check the tightness of the relaxation and extract optimal solutions $\dvar^\star$ of the nonconvex POP~\eqref{eq:strom:popsdp:chain-sparse-pop}. In fact, condition~\ref{item:rank} guarantees the optimal TMS $\tms^\star$ admits a representing Dirac measure supported on the unique optimal solution $\dvar^\star$. Therefore, one can extract $\dvar^\star$ from the entries of $\tms^\star$ corresponding to degree-one monomials. 
Note that condition~\ref{item:rank} is a special case of the \emph{flat extension condition}~\cite{magron23book-sparse,nie2023siopt-moment-momentpolynomialopt,curto2005truncated}, which states that even when the rank of $M_k^\star$ is larger than one, say $\rank{M_k^\star}=r$, as long as $r$ is consistent with the rank of a submatrix of $M_k^\star$ and the restriction of moment matrices to each intersection of cliques is rank-one, then global optimality can be certified and exactly $r$ minimizers of the POP can be extracted. We do not detail this setup because in our applications $\rank{M_k^\star}$ is always one (up to numerical precision).

{\bf Extracting Near-Optimal Solutions}. When condition~\ref{item:rank} does not hold, we can use a template three-step heuristic to extract a near-optimal solution~\cite{yang2023mp-stride}. In step one, we perform spectral decomposition of $M_k^\star$ and let $v_k$ be the eigenvector of the largest eigenvalue. In step two, we normalize $v_k$ by $v_k \leftarrow \frac{v_k}{v_k(1)}$ (so its leading entry is $1$, just like a TMS) and take the entries corresponding to the degree-one monomials of $v_k$, denoted as $\bar{z}$. In step three, we perform local optimization of the POP~\eqref{eq:strom:popsdp:chain-sparse-pop} starting from the (potentially infeasible) initialization $\bar{z}$ (\eg using Matlab \fmincon) and obtain a local minimizer $\hat{z}$.\footnote{We note that finding a feasible solution of a nonconvex POP is in general NP-hard. However, in all our experiments, we found that when the convex optimization~\eqref{eq:strom:popsdp:sparse-momentrelax} is solved to sufficient accuracy, finding a local minimizer is always successful.} Denote as $\hat{p}$ the POP objective~\eqref{eq:strom:popsdp:chain-sparse-pop-obj} evaluated at $\hat{z}$, we compute the \emph{certificate of suboptimality} 
\bea\label{eq:suboptimality}
\subopt_\kappa = \frac{\hat{p} - p^\star_\kappa}{1 + |\hat{p}| + |p^\star_\kappa|} \geq 0.
\eea 
An $\subopt_\kappa$ that is close to zero certifies the global optimality of $\hat{z}$. For example, if $\subopt_\kappa=1\%$, then it certifies $\hat{z}$ is a feasible POP solution that attains a cost $\hat{p}$ \emph{at most} $1\%$ larger than the unknown global minimum $p^\star$. 

We remark that associated with each moment relaxation~\eqref{eq:strom:popsdp:sparse-momentrelax} is its \emph{dual sums-of-squres} (SOS) relaxation (and strong duality holds), hence the name moment-SOS hierarchy. We omit the dual SOS relaxation because the moment relaxation enables extraction of near-optimal solutions and certificate of tightness.



\subsection{Faster Conversion to Standard SDP}
\label{sec:relax:SDP}

Solving the moment relaxation~\eqref{eq:strom:popsdp:sparse-momentrelax} is typically done by converting~\eqref{eq:strom:popsdp:sparse-momentrelax} into a standard linear conic optimization problem in the format of \sedumi~\cite{sturm1999oms-sedumi-sdpsolver} or \mosek~\cite{aps2019ugrm-mosek-sdpsolver}. The standard way to convert~\eqref{eq:strom:popsdp:sparse-momentrelax} into a conic program follows~\cite[\S5.7.1]{Yang2024book-sdp} (\eg as in \tssos, \sostools, \yalmip), which will generate a conic program with not only PSD variables but also free (unconstrained) variables. 

We advocate for a different conversion standard~\cite[\S5.7.2]{Yang2024book-sdp} that generates \emph{only} PSD variables and the resulting semidefinite program admits the same chain-like structure as the POP (\cf Fig.~\ref{fig:strom:popsdp:pop-sdp-conversion}). 


Focusing on the sparse moment relaxation~\eqref{eq:strom:popsdp:sparse-momentrelax}, 
denote 
\bea\label{eq:def-X}
X:=( \{ M_k \}_{k \in \seqordering{N}}, \{ L_{k,i} \}_{k \in \seqordering{N}, i \in \calG_k} )
\eea
as the tuple of all moment and localizing matrices,
which lives in a vector space $\cartsymm$ that is the Cartesian product of $\{ \symm{s(|I_k|,\kappa)} \}_{k \in \seqordering{N}}$ and $\{ \symm{s(|I_k|,\kappa-\degg{k,i})} \}_{k \in \seqordering{N}, i \in \calG_k}$. Let $\cartsymmp \subset \cartsymm$ be the cone containing tuples $X$ whose elements are PSD, we claim the moment relaxation~\eqref{eq:strom:popsdp:sparse-momentrelax} can be written as a standard multi-block SDP~\cite{tutuncu2003mp-sdpt3-sdpsolver}
\begin{equation}
    \label{eq:strom:popsdp:standardsdp}
    p^\star_\kappa = \min_X \left\{ 
        \inprod{C}{X} \mymid \calA(X) = b, \ X \in \cartsymmp
        \right\} 
\end{equation}
for some $b \in \Real{m}$, $C \in \cartsymm$ and linear map $\calA(X):=( \inprod{A_i}{X} )_{i \in \seqordering{m}}$ with $A_i \in \cartsymm$. The inner product in the space $\cartsymm$ is defined element-wise. 

Without detailing the conversion from~\eqref{eq:strom:popsdp:sparse-momentrelax} to \eqref{eq:strom:popsdp:standardsdp} in full generality (which we implement in C++), we present the high-level idea using Example~\eqref{eq:strom:popsdp:momentrelax-example-trajopt}.


{\bf Linear Objective $\inprod{C}{X}$}. Recall the $k$-th clique contains variables $\dvar(I_k)=(x_{k-1},u_{k-1},x_k)$. Set $k=2$, with relaxation order $\kappa=2$, the moment matrix (recall it is symmetric and hence we only write the upper triangular part)
\begingroup
\fontsize{6pt}{0pt}\selectfont
\bea \label{eq:toy-M2}
M_2 = \ell_\tms\left( 
            \begin{bmatrix}
                \red{1} & x_1 & u_1 & \red{x_2} & x_1^2 & x_1 u_1 & x_1 x_2 & u_1^2 & u_1 x_2 & \red{x_2^2} \\
                * & x_1^2 & x_1 u_1 & x_1 x_2 & x_1^3 & x_1^2 u_1 & x_1^2 x_2 & x_1 u_1^2 & \blue{x_1 u_1 x_2} & x_1 x_2^2 \\
                * & * & u_1^2 & u_1 x_2 & x_1^2 u_1 & x_1 u_1^2 & \blue{x_1 u_1 x_2} & u_1^3 & u_1^2 x_2 & u_1 x_2^2 \\
                * & * & * & \red{x_2^2} & x_1^2 x_2 & \blue{x_1 u_1 x_2} & x_1 x_2^2 & u_1^2 x_2 & u_1 x_2^2 & \red{x_2^3} \\
                * & * & * & * & x_1^4 & x_1^3 u_1 & x_1^3 x_2 & x_1^2 u_1^2 & x_1^2 u_1 x_2 & x_1^2 x_2^2 \\
                * & * & * & * & * & x_1^2 u_1^2 & x_1^2 u_1 x_2 & x_1 u_1^3 & x_1 u_1^2 x_2 & x_1 u_1 x_2^2 \\
                * & * & * & * & * & * & x_1^2 x_2^2 & x_1 u_1^2 x_2 & x_1 u_1^2 x_2 & x_1 x_2^3 \\
                * & * & * & * & * & * & * & u_1^4 & u_1^3 x_2 & u_1^2 x_2^2 \\
                * & * & * & * & * & * & * & * & u_1^2 x_2^2 & u_1 x_2^3 \\
                * & * & * & * & * & * & * & * & * & \red{x_2^4}
            \end{bmatrix}
\right) 
\eea
\endgroup
has rows (columns) indexed by $[\dvar(I_2)]_2$, \ie the vector of monomials in $\dvar(I_2)$ of degree up to $2$ (the first row in~\eqref{eq:toy-M2}). 
Clearly, $M_2$ contains the monomials in $\dvar(I_2)$ of degree up to $4$, and hence the objective $f_2(\dvar(I_2))$ in~\eqref{eq:toy-as-sparse-pop}, a polynomial of degree $2$, can be written as $\inprod{C_2}{M_2}$ for some matrix $C_2$. This procedure can be done for every $k \in [N]$, thus, the objective of \eqref{eq:strom:popsdp:momentrelax-example-trajopt} can be written as $\sum_{k=1}^N \inprod{C_k}{M_k}$, or compactly $\inprod{C}{X}$ (recall $X$ contains all moment matrices~\eqref{eq:def-X}).

{\bf Linear Constraints $\calA(X)=b$}. We decompose $\calA(X)=b$ into four types.
\begin{enumerate}
    \item Moment constraints $\calA_{\text{mom}}$. Due to the definition of the moment matrix $M_k$~\eqref{eq:strom:popsdp:sparse-momentrelax-moment}, a single monomial can appear multiple times. For example in~\eqref{eq:toy-M2}, we have $M_2(2, 9) = M_2(3, 7) = M_2(4, 6)$ (highlighted in blue).
    \item Inequality constraints $\calA_{\text{ineq}}$. Each element in $L_{k,i}$ can be expressed as a linear combination of elements in $M_k$. For example~\eqref{eq:strom:popsdp:momentrelax-example-trajopt}, we can write down the localizing matrix generated by the inequality constraint $1-u_1^2\geq 0$: 
            \begin{align}
                L_{2,1} & = \ell_\tms \left( 
                    (1 - u_1^2) \cdot \basis{\dvar(I_2)}{1} \basis{\dvar(I_2)}{1}\tran
                \right) \nonumber\\
                & = \ell_\tms \left( 
                    \begin{bmatrix}
                        1 - u_1^2 & x_1 - x_1 u_1^2 & u_1 - u_1^3 & x_2 - u_1^2 x_2 \\
                        * & x_1^2 - x_1^2 u_1^2 & x_1 u_1 - x_1 u_1^3 & x_1 x_2 - x_1 u_1^2 x_2 \\
                        * & * & u_1^2 - u_1^4 & u_1 x_2 - u_1^3 x_2 \\
                        * & * & * & x_2^2 - u_1^2 x_2^2 
                    \end{bmatrix}
                 \right)
            \end{align}
    and observe the linear constraints
    \begin{equation}
        L_{2, 1}(1, 1) = M_2(1, 1) - M_2(3, 3), \quad  L_{2, 2}(1, 2) = M_2(1, 2) - M_2(3, 6),
    \end{equation}
    which can be repeated for every entry of $L_{2,1}$.
    \item Equality constraints $\calA_{\text{eq}}$. Each linear constraint in~\eqref{eq:strom:popsdp:sparse-momentrelax-equality} can be expressed by a linear combination of elements in $M_k$. For example, the linear constraints generated from $h_{2,1}(\dvar(I_2)) = 0$ are:
    \begin{align}
        \ell_\tms \left( 
            \left( x_2 + (\dt - 1) \cdot x_1 + \dt \cdot x_1 u_1 \right) \cdot \basis{\dvar(I_2)}{2} 
         \right) = 0
    \end{align} 
    which leads to $10$ linear constraints in $\calA_{\text{eq}}$. The first one of them is
    \begin{align}
        \ell_\tms \left( 
                x_2 + (\dt - 1) \cdot x_1 + \dt \cdot x_1 u_1
         \right) = 0
    \end{align} 
    which leads to the constraint
    \begin{equation}
            M_2(1, 4) + (\dt - 1) \cdot M_2(1, 2) + \dt \cdot M_2(2, 3) = 0.
    \end{equation}
    \item Consensus constraints $\calA_{\text{sen}}$. Since $I_k$ and $I_{k+1}$ have overlapping elements, two adjacent moment matrices $M_k$ and $M_{k+1}$ also have overlapping elements. For example~\eqref{eq:strom:popsdp:momentrelax-example-trajopt}, we can write down the third moment matrix
    \begingroup
    \fontsize{6pt}{0pt}\selectfont
    \begin{equation}
        M_3 = \ell_\tms\left( 
            \begin{bmatrix}
                \red{1} & \red{x_2} & u_2 & x_3 & \red{x_2^2} & x_2 u_2 & x_2 x_3 & u_2^2 & u_2 x_3 & x_3^2 \\
                * & \red{x_2^2} & x_2 u_2 & x_2 x_3 & \red{x_2^3} & x_2^2 u_2 & x_2^2 x_3 & x_2 u_2^2 & x_2 u_2 x_3 & x_2 x_3^2 \\
                * & * & u_2^2 & u_2 x_3 & x_2^2 u_2 & x_2 u_2^2 & x_2 u_2 x_3 & u_2^3 & u_2^2 x_3 & u_2 x_3^2 \\
                * & * & * & x_3^2 & x_2^2 x_3 & x_2 u_2 x_3 & x_2 x_3^2 & u_2^2 x_3 & u_2 x_3^2 & x_3^3 \\
                * & * & * & * & \red{x_2^4} & x_2^3 u_2 & x_2^3 x_3 & x_2^2 u_2^2 & x_2^2 u_2 x_3 & \blue{x_2^2 x_3^2} \\
                * & * & * & * & * & x_2^2 u_2^2 & x_2^2 u_2 x_3 & x_2 u_2^3 & x_2 u_2^2 x_3 & x_2 u_2 x_3^2 \\
                * & * & * & * & * & * & \blue{x_2^2 x_3^2} & x_2 u_2^2 x_3 & x_2 u_2^2 x_3 & x_2 x_3^3 \\
                * & * & * & * & * & * & * & u_2^4 & u_2^3 x_3 & u_2^2 x_3^2 \\
                * & * & * & * & * & * & * & * & u_2^2 x_3^2 & u_2 x_3^3 \\
                * & * & * & * & * & * & * & * & * & x_3^4
            \end{bmatrix}
        \right)
\end{equation}
\endgroup
and observe that $M_2$ and $M_3$ share the same monomials highlighted in red:
\begin{subequations}
    \begin{align}
        & M_2(1, 1) = M_3(1, 1), M_2(1, 4) = M_3(1, 2), M_2(4, 4) = M_3(2, 2) \\
        & M_2(4, 10) = M_3(2, 5), M_2(10, 10) = M_3(5, 5).
    \end{align}
\end{subequations}
\end{enumerate} 

In summary, our conversion generates an SDP whose structure is illustrated on the right-hand side of Fig.~\ref{fig:strom:popsdp:pop-sdp-conversion}, where the solid lines connect two PSD variables if and only if there exist linear constraints between them.
{In \S\ref{app:sec:conversion}, we show our conversion package is the only one capable of real-time conversion in Matlab, compared to \sostools and \yalmip.}

\section{Algorithmic Backbone: \sgsadmm}
\label{sec:strom:sgsadmm}

Solving the SDP~\eqref{eq:strom:popsdp:standardsdp} using interior point solvers does not scale to practical trajectory optimization problems. We introduce the scalable \sgsadmm algorithm.

Consider~\eqref{eq:strom:popsdp:standardsdp} as the primal SDP. Let $y \in \Real{m}$ and denote $\calA^*$ as the adjoint of $\calA$ defined as $\calA^* y := \sum_{l \in \seqordering{m}} y_l A_l$. The Lagrangian dual of~\eqref{eq:strom:popsdp:standardsdp} reads:
\begin{align}
    \label{eq:strom:sgsadmm:standardsdp-dual}
    \max_{y \in \Real{m}, S \in \Omega} \left\{ 
        \inprod{b}{y} \mymid \calA^* y + S = C, \ S \in \cartsymmp
     \right\}.
\end{align}
\sgsadmm can be seen as a semi-proximal ADMM method applied to the Augmented Lagrangian of \eqref{eq:strom:sgsadmm:standardsdp-dual}. Without going deep into the theory, and denoting $\Pi_{\cartsymmp}(\cdot)$ as the projection onto $\cartsymmp$, we present \sgsadmm in Algorithm~\ref{alg:strom:sgsadmm:sgsadmm}. Global convergence of \sgsadmm for the SDP pair~\eqref{eq:strom:popsdp:standardsdp}-\eqref{eq:strom:sgsadmm:standardsdp-dual} is established in~\cite{li2018mp-sgs-ccqp,chen2017mp-sgsadmm}. From Algorithm~\ref{alg:strom:sgsadmm:sgsadmm}, we see two main computational tasks are solving linear systems with a fixed coefficient matrix $\calA\calA^*$ and computing the projection of a sequence of symmetric matrices onto the PSD cone. As we shall see in \S\ref{sec:strom:gpu}, these two tasks can be implemented efficiently in GPUs, resulting in significant speedups.



\setlength{\textfloatsep}{0pt}
\begin{algorithm}[t]
    \label{alg:strom:sgsadmm:sgsadmm}
    \caption{\sgsadmm for solving~\eqref{eq:strom:sgsadmm:standardsdp-dual}.}
    \KwIn{Initial points $X^0 \in \cartsymm$ and $S^0 \in \cartsymm$, $\tau \in (0, 2)$, $\sigma > 0$.}
    
    \textbf{For} $k = 0, 1, 2, \dots$:
    \textbf{\quad Step 1.} Compute
    \begin{align}
        \label{eq:strom:sgsadmm:solve-y1}
        r_s^{k+\frac{1}{2}} := \frac{1}{\sigma} b - \calA \left( \frac{1}{\sigma} X^k + S^k - C \right), \ 
        y^{k + \frac{1}{2}} = \left( \calA \calA^* \right)^{-1} r_s^{k+\frac{1}{2}}
    \end{align}

    \textbf{\quad Step 2.} Compute 
    \begin{align}
        \label{eq:strom:sgsadmm:solve-S}
        X_b^{k+1} := X^k + \sigma (\calA^* y^{k+\frac{1}{2}} - C), \
        S^{k+1} = \frac{1}{\sigma}\left( 
            \Pi_{\cartsymmp}\left( 
                X_b^{k+1}
                \right) - X_b^{k+1}
            \right)
    \end{align}

    \textbf{\quad Step 3.} Compute 
    \begin{align}
        \label{eq:strom:sgsadmm:solve-y2}
        r_s^{k+1} := \frac{1}{\sigma} b - \calA \left( \frac{1}{\sigma} X^k + S^{k+1} - C \right), \
        y^{k + 1} = \left( \calA \calA^* \right)^{-1} r_s^{k+1}
    \end{align}

    \textbf{\quad Step 4.} Compute 
    \begin{align}
        \label{eq:strom:sgsadmm:solve-X}
        X^{k+1} = X^k + \tau \sigma \left( 
            S^{k+1} + \calA^* y^{k+1} - C
            \right)
    \end{align}
    
    \textbf{Until} terminal conditions hold.
    
    \KwOut{$\left( X^{k+1}, y^{k+1}, S^{k+1} \right)$.} 
\end{algorithm}


\textbf{Terminal Conditions}. We terminate Algorithm~\ref{alg:strom:sgsadmm:sgsadmm}  if either of the following two terminal conditions are met: (1) The maximum iteration number \texttt{maxiter} is reached. (2) The standard max KKT residual $\eta := \max\left\{ \eta_p, \eta_d, \eta_g \right\}$ is below a certain threshold \texttt{tol}, where $\eta_p, \eta_d, \eta_g$ are defined as:
\begin{align}
    \label{eq:strom:sgsadmm:kkt-residual}
    \eta_p = \frac{
        \norm{\calA(X) - b}_2
    }{1 + \norm{b}_2}, \ \eta_d = \frac{
        \norm{\calA^* y + S - C}_2
    }{1 + \norm{C}_2}, \ \eta_g = \frac{
        | \inprod{C}{X} - \inprod{b}{y} |
    }{
        1 + | \inprod{C}{X} | + | \inprod{b}{y} | 
    }.
\end{align}

\textbf{Refined Certificate of Suboptimality}. Unlike \mosek, which can solve the KKT residual $\eta$ to machine precision and compute $p^\star_\kappa$ exactly for estimating the certificate of suboptimality~\eqref{eq:suboptimality}, first-order methods can only achieve moderate accuracy ($\eta$ around $10^{-3}$ to $10^{-5}$) in our setting. Therefore, given an output triplet $(X, y, S)$ from Algorithm~\ref{alg:strom:sgsadmm:sgsadmm}, we need to generate a valid lower bound of $p^\star$. 
Given any $\dvar \in \Real{d}$, define the rank-1 lifting operator at relaxation order $\kappa$:
\begin{subequations}
    \label{eq:strom:sgsadmm:rank1-lifting}
    \begin{align}
        M_k(\dvar) & = \basis{\dvar(I_k)}{\kappa} \basis{\dvar(I_k)}{\kappa}\tran, \ \forall k \in \seqordering{N} \\
        L_{k,i}(\dvar) & = g_{k,i}(\dvar(I_k)) \cdot \basis{\dvar(I_k)}{\kappa - \degg{k,i}} \basis{\dvar(I_k)}{\kappa - \degg{k,i}}\tran, \ \forall k \in \seqordering{N}, i \in \calG_k
    \end{align}
\end{subequations}
Let $X(\dvar)$ be the aggregation of $M_k(\dvar)$ and $L_{k,i}(\dvar)$.
Denote $\hat{\dvar}$ as any feasible solution for the POP~\eqref{eq:strom:popsdp:chain-sparse-pop} such that $\calA(X(\hat{\dvar})) = b$ naturally holds. Let $X_\beta$ be the $\beta$-th symmetric matrix in $X$. The following inequality holds for any feasible $\hat{z}$
\begin{subequations}
    \begin{align}
        & \sum_{k=1}^N f_k(\hat{\dvar}(I_k)) = \inprod{C}{X(\hat{\dvar})} = \inprod{C - \calA^* y}{X(\hat{\dvar})} + \inprod{\calA^* y}{X(\hat{\dvar})} \\ 
        = & \inprod{C - \calA^* y}{X(\hat{\dvar})} + \inprod{y}{\calA(X(\hat{\dvar}))} = \inprod{C - \calA^* y}{X(\hat{\dvar})} + \inprod{y}{b} \\
        \ge & \inprod{b}{y} + \sum_{\beta} R_\beta \cdot \min \left\{ 
            0, \lambda_{\min}\left( (C - \calA^* y)_\beta \right)
         \right\} \label{eq:inexact-lower-bound}
    \end{align}
\end{subequations}
if $R_\beta \ge \trace{X(\hat{\dvar})_\beta}$. Since $\hat{\dvar}$ can be arbitrarily picked, we have
\begin{align}
    \label{eq:strom:sgsadmm:valid-lowerbound}
    \inprod{C}{X(\hat{\dvar})} \ge p^\star \ge \inprod{b}{y} + \sum_{\beta} R_\beta \cdot \min \left\{ 
        0, \lambda_{\min}\left( (C - \calA^* y)_\beta \right)
     \right\}.
\end{align}
The suboptimality gap $\subopt$ is refined as:
\begin{align}
    \label{eq:strom:sgsadmm:suboptimality-gap}
    \subopt := \frac{
        \inprod{C}{X(\hat{\dvar})} - \inprod{b}{y} - \sum_{\beta} R_\beta \cdot \min \left\{ 
            0, \lambda_{\min}\left( (C - \calA^* y)_\beta \right)
         \right\}
    }{
        1 + | \inprod{C}{X(\hat{\dvar})} | + | \inprod{b}{y} + \sum_{\beta} R_\beta \cdot \min \left\{ 
            0, \lambda_{\min}\left( (C - \calA^* y)_\beta \right)
         \right\} |
    },
\end{align}
where the only difference from~\eqref{eq:suboptimality} is the lower bound in~\eqref{eq:inexact-lower-bound} is used instead of $p^\star_\kappa$. The feasible $\hat{z}$ is obtained with the extraction method presented in \S\ref{sec:relax:hierarchy}.

The certificate~\eqref{eq:strom:sgsadmm:suboptimality-gap} holds if an upper bound $R_\beta$ exists for the trace of each $X(\hat{z})_\beta$. We prove the existence of $R_\beta$ in \S\ref{app:sec:existenceRbeta}.




\section{GPU Speedup: \cuadmm}
\label{sec:strom:gpu}

\shucheng{We now implement Algorithm~\ref{alg:strom:sgsadmm:sgsadmm} in C++ and CUDA}.
We will assume (\emph{i}) relaxation order $\kappa = 2$; (\emph{ii}) each clique has the same size, \ie $|I| = |I_1| = \cdots =|I_N|$; (\emph{iii}) each polynomial inequality constraint is of degree $1$ or $2$, \ie $\degg{k,i} = 1, \forall k \in \seqordering{N}, i \in \calG_k$. These assumptions hold for all applications considered in this paper.



\sgsadmm in Algorithm~\ref{alg:strom:sgsadmm:sgsadmm} involves three primary operations: (a) performing sparse matrix and dense vector multiplications (\cf~\eqref{eq:strom:sgsadmm:solve-X}), (b) solving sparse linear systems (\cf~\eqref{eq:strom:sgsadmm:solve-y1} and~\eqref{eq:strom:sgsadmm:solve-y2}), and (c) projecting onto the Cartesian product of multiple PSD cones (\cf~\eqref{eq:strom:sgsadmm:solve-S}). For (a), we sort matrix variables and linear constraints by clique indices. 
$X$ is stored by concatenating the symmetric vectorizations (\texttt{svec})~\cite{tutuncu2003mp-sdpt3-sdpsolver} of $M_k$ and $L_{k,i}$. These arrangements ensure $\calA$ and $\calA^*$ exhibit low bandwidth and good locality during sparse matrix-dense vector multiplications. On GPUs, we use \texttt{cuSPARSE}~\cite{naumov2010gputc-cusparse} for these operations.

\vspace{-6mm}

\begin{figure}[htbp]
    \begin{minipage}{\textwidth}
        \centering
        \includegraphics[width=\textwidth]{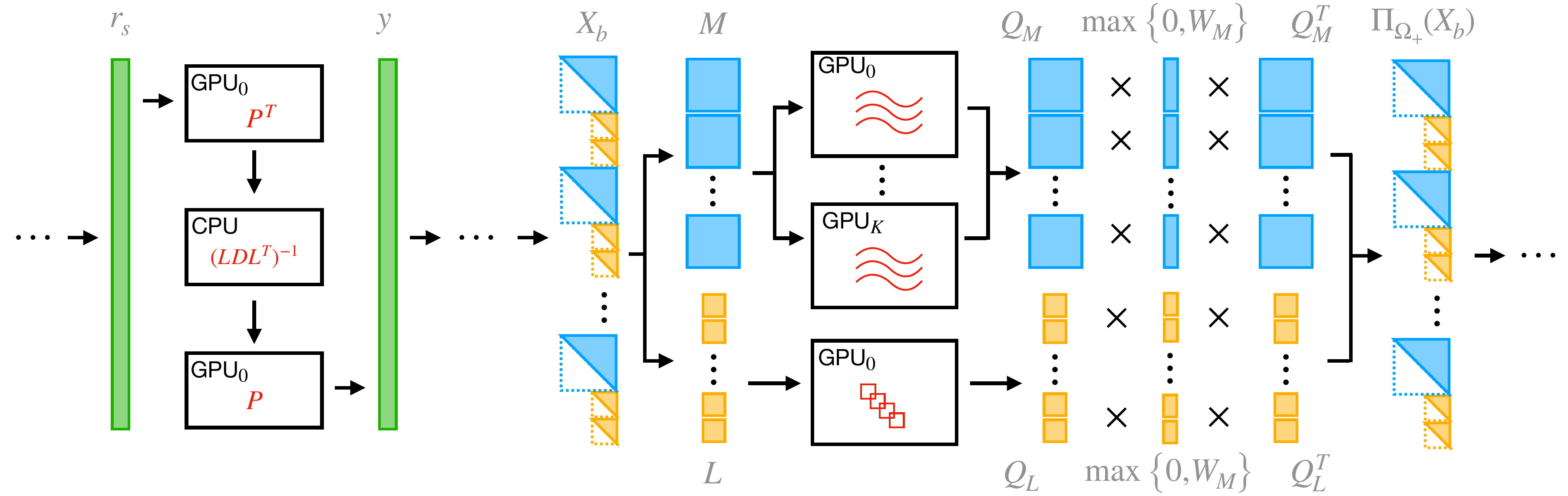}
    \end{minipage}
    \caption{Illustration of two operations in \cuadmm. Left: solving the sparse linear system. Right: parallel PSD cone projection.
    \label{fig:strom:sdp:solver-pipeline}}
    \vspace{2mm}
\end{figure}
\vspace{-8mm}

\textbf{Hybrid Sparse Linear System Solver}. 
For (b), we run sparse Cholesky decomposition with pre-permutation~\cite{chen2008toms-cholmod} on $\calA \calA^*$ at the beginning and only once. Since $\calA$ may be ill-conditioned, we add $\calA \calA^*$ by a scaled identity matrix $\epsilon I$ (with $\epsilon$ a small positive number) to enforce positive definiteness:
\begin{align}
    \label{eq:strom:gpu:cholesky}
    \epsilon I + \calA \calA^* = P LDL\tran P\tran.
\end{align}
Then, we can efficiently solve sparse triangular linear systems in~\eqref{eq:strom:sgsadmm:solve-y1} and~\eqref{eq:strom:sgsadmm:solve-y2}:
\begin{subequations}
    \begin{align}
        & y = \left( PLDL\tran P\tran \right)^{-1} r_s = P (LDL\tran)^{-1} P\tran r_s \\
        \Longrightarrow & w \leftarrow P\tran r_s, \ w \leftarrow (LDL\tran)^{-1} w, \ y \leftarrow P w .
    \end{align}
\end{subequations}
$P$ is a permutation matrix whose multiplication with dense vectors can be done in parallel. Thus, we implement $w \leftarrow P\tran r_s$ and $y \leftarrow P w$ on GPUs as 1-D reorderings. $w \leftarrow \left( LDL\tran \right)^{-1} w$ is not parallelizable, so we call \texttt{CHOLMOD}'s $LDL\tran$ solver~\cite{chen2008toms-cholmod} in CPU. This hybrid approach is approximately $4$ times faster than solving two sparse triangular systems on a GPU with \texttt{cuSPARSE}~\cite{naumov2010gputc-cusparse}, and $10\%$ faster than performing two permutations on a CPU, even with the GPU-CPU transfer overhead. The hybrid linear system solver is depicted in Fig.~\ref{fig:strom:sdp:solver-pipeline} left side.

\textbf{Parallel PSD Cone Projection.} 
For (c), we perform the projection onto multiple PSD cones in parallel. Let $X_{\beta}$ be the $\beta$-th matrix in $X$, projecting
$X_\beta$ to be PSD involves two steps~\cite{higham88-psdprojection}: (1) eigenvalue decomposition (\texttt{eig}): $X_\beta = Q_\beta W_\beta Q_\beta\tran$; (2) removal of negative eigenvalues: $\Pi_{\mathbb{S}_+}(X_\beta) := Q_\beta \max\{0,  W_\beta \} Q_\beta\tran$.
For the moment relaxation~\eqref{eq:strom:popsdp:sparse-momentrelax}, the number of moment matrices is $N$ and localizing matrices is $\sum_{k \in \seqordering{N}} |\calG_k|$. 
The size of the moment matrices is fixed at $s(|I|, 2)$, and the size of the localizing matrices is fixed at $s(|I|, 1)$. Based on two observations: (1) $N \ll \sum_{k \in \seqordering{N}} |\calG_k|$ and (2) $s(|I|, 2) \gg s(|I|, 1)$, we choose different eigenvalue decomposition routines for moment and localizing matrices. For localizing matrices, we use the batched-matrix \texttt{eig} interface with the Jacobi method~\cite{sleijpen2000siopt-eig-jacobi} in \texttt{cuSOLVER}, which is faster when the number of matrices is large and the matrix size is small. For moment matrices, we employ the single-matrix \texttt{eig} in \texttt{cuSOLVER} with the direct QR method~\cite{watkins1982siopt-eig-qr} and parallelize operations using multiple CUDA streams. Empirically in a single GPU, this approach is $10\%$ to $20\%$ faster than wrapping all moment matrices into the batched \texttt{eig}. 
If multiple GPUs are available, we distribute moment matrices for \texttt{eig} since this step dominates the runtime in Algorithm~\ref{alg:strom:sgsadmm:sgsadmm}.
After eigenvalue decomposition, we employ a mixture of custom kernel functions and batched matrix multiplication operators in \texttt{cuBLAS}~\cite{fatica2008hcs-cuda-toolkit} to perform batched projection. Moreover, 
we frequently split $X$ in \texttt{svec} form to two batched matrix sequences $\{ M_k \}_{k \in \seqordering{N}}$ and $\{ L_{k,i} \}_{k \in \seqordering{N}, i\in \calG_k}$, and vice versa.
Fast parallel mappings are implemented to minimize the cost. The parallel PSD cone projection is depicted in the right part of Fig.~\ref{fig:strom:sdp:solver-pipeline}.



\section{Experiments}
\label{sec:exp}


{\bf General Setup}. We consider five trajectory optimization problems: inverted pendulum, cart-pole, vehicle landing, flying robot, and car back-in, as illustrated in Fig.~\ref{fig:exp:gen:sys-illustration}. We rescale all polynomial variables and coefficients
to $[-1, 1]$ before moment relaxation for numerical stability. We unify the loss function design as the LQR-style loss, \ie denote the final state as $\xf$, the loss function~\eqref{eq:intro:trajopt-generalform} is
\begin{equation}
    \label{eq:exp:gen:lqr-loss}
    P_f \cdot (x_N - \xf)\tran Q_x (x_N - \xf) + \sum_{k=0}^{N-1} \left\{ 
        (x_k - \xf)\tran Q_x (x_k - \xf) + u_k\tran Q_u u_k 
     \right\},
\end{equation}
where $Q_x$ and $Q_u$ are designed to be identity matrices after rescaling. 


\begin{figure}[t]
    \begin{minipage}{\textwidth}
        \centering
        \begin{tabular}{ccc}
            \begin{minipage}{0.14\textwidth}
                \centering
                \includegraphics[width=\columnwidth]{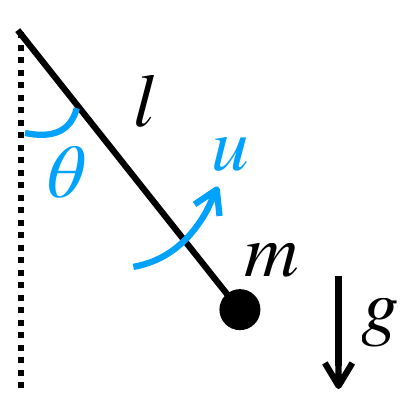}
                {\smaller (a) Pendulum}
            \end{minipage}

            \begin{minipage}{0.28\textwidth}
                \centering
                \includegraphics[width=\columnwidth]{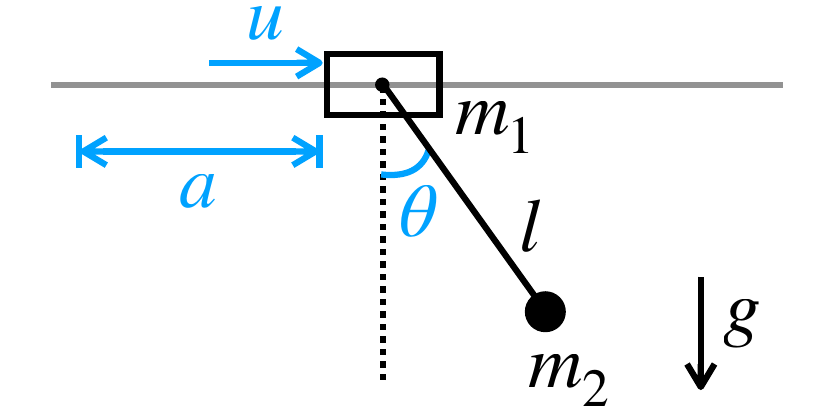}
                {\smaller (b) Cart-pole}
            \end{minipage}

            \begin{minipage}{0.14\textwidth}
                \centering
                \includegraphics[width=\columnwidth]{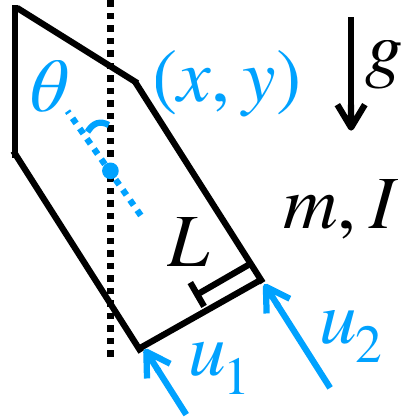}
                {\smaller (c) Landing}
            \end{minipage}

            \begin{minipage}{0.18\textwidth}
                \centering
                \includegraphics[width=\columnwidth]{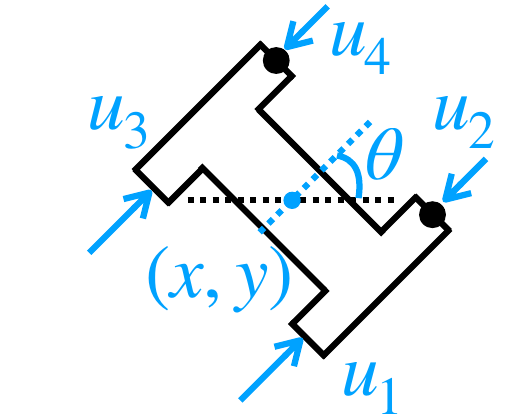}
                {\smaller (d) Flying robot}
            \end{minipage}

            \begin{minipage}{0.28\textwidth}
                \centering
                \includegraphics[width=\columnwidth]{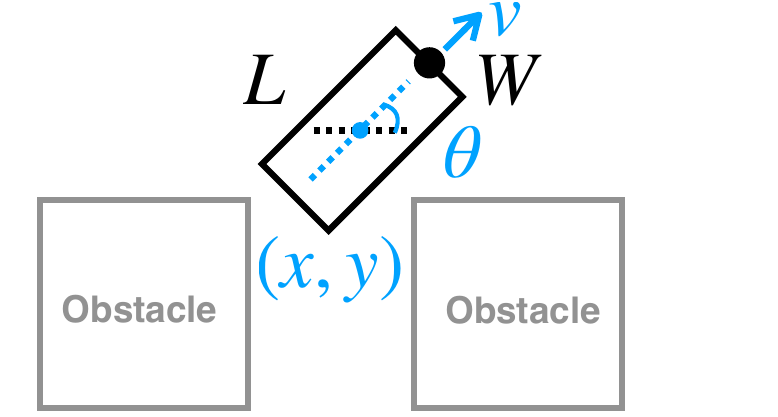}
                {\smaller (e) Car back-in}
            \end{minipage}
        \end{tabular}
    \end{minipage}

    \caption{Five trajectory optimization problems of dynamical systems. \label{fig:exp:gen:sys-illustration}
    }
    \vspace{2mm}
\end{figure}


\textbf{Baselines.} For interior point methods, we choose the commercial solver \mosek~\cite{aps2019ugrm-mosek-sdpsolver} as our baseline. For first-order methods, we choose the state-of-the-art general SDP solvers \cdcs~\cite{zheng2017ifac-cdcs-sdpsolver} and \sdpnal~\cite{yang2015mp-sdpnalplus-sdpsolver}.\footnote{We have tested \scs~\cite{odonoghue2023-scs-sdpsolver} but found its Matlab version to always generate wrong results for multiple-block SDP problems.} We only use the ADMM solver in \sdpnal because the semismooth Newton solver is not scalable. For first-order methods, \texttt{tol} is set as $10^{-4}$ in Algorithm~\ref{alg:strom:sgsadmm:sgsadmm}. 

Experiments were conducted on a high-performance workstation equipped with a 2.7 GHz AMD 64-Core sWRX8 Processor, two 2.5 GHz NVIDIA RTX 6000 Ada Generation GPUs, and 1 TB of RAM, to enable \mosek~to solve large-scale problems (with 64 CPU threads). 
In \cuadmm, we set up 15 CUDA streams per GPU for eigenvalue decomposition of moment matrices. The suboptimality gap $\subopt$ is defined in~\eqref{eq:strom:sgsadmm:suboptimality-gap}, where $\hat{\dvar}$ is obtained from the extraction method presented in \S\ref{sec:relax:hierarchy} with Matlab \fmincon~as the local solver. We only report runtime of SDP solvers because (1) the runtime of local solvers is negligible compared to the convex SDP solver, and (2) better CPU~\cite{howell2019iros-altro} and GPU~\cite{malyuta2022ieee-convexopt-trajectorygeneration} local solvers are available, whose design is beyond the scope of this paper.  

\textbf{Summary Results}. 
For the pendulum problem, we tested the performance of four SDP solvers on 100 initial states generated by a dense grid.
For the other four experiments, we assessed \cuadmm's performance on 10 randomly generated initial states. Due to the significantly longer solving times of the other three solvers, we evaluated their performance on only one random initial state that \cuadmm could solve, with a maximum running time set at 5 hours. Table~\ref{tab:exp:gen:onlyone} presents the problem scale, suboptimality gaps, and total running times. It shows that (1) the mean and median suboptimality gaps from \cuadmm are either well below $10^{-2}$ (in all cases except for the car back-in) or near $10^{-2}$ (for the car back-in); (2) \mosek can solve the pendulum and cart-pole problems very well, but cannot scale to other problems; (3) Compared with other ADMM-based solvers, \cuadmm achieves up to $10\times$ speedup in large-scale SDP problems (with $m > 5 \times 10^5$); In fact, the other ADMM solvers cannot produce $<1\%$ suboptimality certificates with the time presented in Table~\ref{tab:exp:gen:onlyone}; (4) With warmstart techniques, \cuadmm can solve the pendulum problem to global optimality with an average cost of 0.66 seconds. We only list \cuadmm and \sdpnal's warm start results since \cdcs does not support explicit initial guesses.

\begin{table}[t]
    \centering
    \resizebox{\columnwidth}{!} {
    \begin{tabular}{|c|c|c|c|c|c|c|c|c|c|c|c|}
        \hline
        \multirow{2}{*}{} & \multirow{2}{*}{$d_x$} &\multirow{2}{*}{$d_u$} & \multirow{2}{*}{$\text{size}(M)$} & \multirow{2}{*}{$\#M$} & \multirow{2}{*}{$\text{size}(L)$} & \multirow{2}{*}{$m$} & \multirow{2}{*}{$\log_{10}\xi$} & \multicolumn{4}{c|}{Total Time} \\
        \cline{9-12}
        & & & & & & & & \mosek & \cdcs & \sdpnal & \cuadmm \\
        \hline 
        \multirow{2}{*}{Pendulum} & \multirow{2}{*}{$4$} & \multirow{2}{*}{$1$} & \multirow{2}{*}{$55$} & \multirow{2}{*}{$30$} & \multirow{2}{*}{$10$} & \multirow{2}{*}{$47351$} & $\bm{-2.58}$ ($\bm{-2.93}$) &  \multirow{2}{*}{$9.5$s} & \multirow{2}{*}{$110.3$s} & cold: $264.4$s & cold: $31.2$s ($14.6$s) \\
        \cline{8-8} \cline{11-12}
        & & & & & & & $\bm{-2.71}$ ($\bm{-2.97}$) & & & warm: $2.95$s & warm: \textbf{$\bm{0.66}$s} (\textbf{$\bm{0.62}$s}) \\
        \hline 
        Cart-Pole & $5$ & $1$ & $105$ & $30$ & $14$ & $168961$ & $\bm{-2.89}$ ($\bm{-2.98}$) & \textbf{$\bm{112.0}$s} & $1471$s & $3671$s & $183.0$s ($173.1$s) \\
        \hline 
        Car Back-in & $7$ & $4$ & $190$ & $30$ & $19$ & $509141$ & $\bm{-1.87}$ ($\bm{-1.99}$) & $>5$h & $6123$s & $>5$h & \textbf{$\bm{431.5}$s} (\textbf{$\bm{536.8}$s})  \\
        \hline 
        Vehicle Landing & $8$ & $2$ & $190$ & $50$ & $19$ & $946326$ & $\bm{-2.47}$ ($\bm{-2.42}$)  & $>5$h & $2.9$h & $>5$h &  \textbf{$\bm{919.2}$s} (\textbf{$\bm{1165}$s}) \\
        \hline 
        Flying Robot & $8$ & $4$ & $231$ & $60$ & $21$ & $1595001$ & $\bm{-2.57}$ ($\bm{-2.49}$) & $**$ & $>5$h & $>5$h & \textbf{$\bm{1648}$s} (\textbf{$\bm{2027}$s}) \\
        \hline 
    \end{tabular}
    }
    \vspace{0.5mm}
    \caption{Summary of numerical experiments. $(\text{size}(M), \#M)$: (dimension, number) of the moment matrices; $\text{size}(L)$: dimension of the localizing matrix. $m$: number of equality constraints in SDP.  Means (medians) are reported for $\log_{10}$ of suboptimality gaps $\xi$, and total running time of \cuadmm. 
    ``$**$'' indicates \mosek runs out of 1TB memory.
    \label{tab:exp:gen:onlyone}}
    \vspace{-2mm}
\end{table}

Detailed experimental results and analyses are presented in \S\ref{app:sec:exp}. \emph{We encourage the reader to check out videos of the certifiably optimal trajectories computed by \cuadmm on our project website}. 
Due to limited space, in the next we briefly analyze three examples out of five: pendulum, car back-in, and flying robot. 



{\bf Inverted Pendulum}. Denote pendulum's state as $(\theta, \dot{\theta})$, we first use the variational integrator to discretize the continuous-time dynamics on $\SOtwo$. See \S\ref{app:subsec:exp:p} for a detailed derivation of discretized dynamics and constraints. 

\emph{Warmstarts with kNN}. To enable data-driven methods for warmstarts, we collect $60 \times 120$ \mosek solutions in the state space $(\theta, \dot{\theta}) \in [0, \pi] \times [-5, 5]$ off-line. Given a new initial state $(\theta_0, \dot{\theta}_0)$, we use Delaunay Triangulation~\cite{lee1980springer-delaunay} to search for three nearest neighbors among the \mosek's solutions. A convex combination of these three solutions is treated as the initial solution for \cuadmm and \sdpnal. 

\emph{Comparison with baselines}. 
In the state space $(\theta, \dot{\theta}) \in [0, \pi] \times [-5, 5]$, we sample $10 \times 10$ initial states with a dense grid. Fig.~\ref{fig:exp:p} bottom shows the performance of different solvers. In almost all experiments, \mosek achieved suboptimality gaps as low as $10^{-7}$. Among the first-order solvers, \cdcs struggled to get low dual infeasibility $\eta_d$, resulting in significantly higher suboptimality gaps. While \sdpnal shows reliable performance, its per-iteration cost is $7$ times higher than \cuadmm. With warmstarts, \cuadmm solves the pendulum problem in less than $0.7$s on average, achieving $10\times$ speedup compared to \mosek. 
$10\%$ initial states are hard to solve for all solvers, which form a mysterious spiral line as discussed in~\cite{han2023arxiv-nonsmooth}. See \S\ref{app:subsec:exp:p} for a detailed discussion.

\emph{Model predictive control}. 
Fig.~\ref{fig:exp:p} top shows a simulated MPC case starting from $\theta_0 = 0.1, \dot{\theta} = 0.0$. We set the control frequency as $10$Hz. In almost all time steps, the suboptimality gap is below $10^{-2}$, and the average solving time is $0.72$s. 
\vspace{-10mm}

\begin{figure}[htbp]
    \centering
    \begin{minipage}{\textwidth}
        \centering
        \begin{tabular}{ccc}
            \begin{minipage}{0.33\textwidth}
                \centering
                \includegraphics[width=\columnwidth]{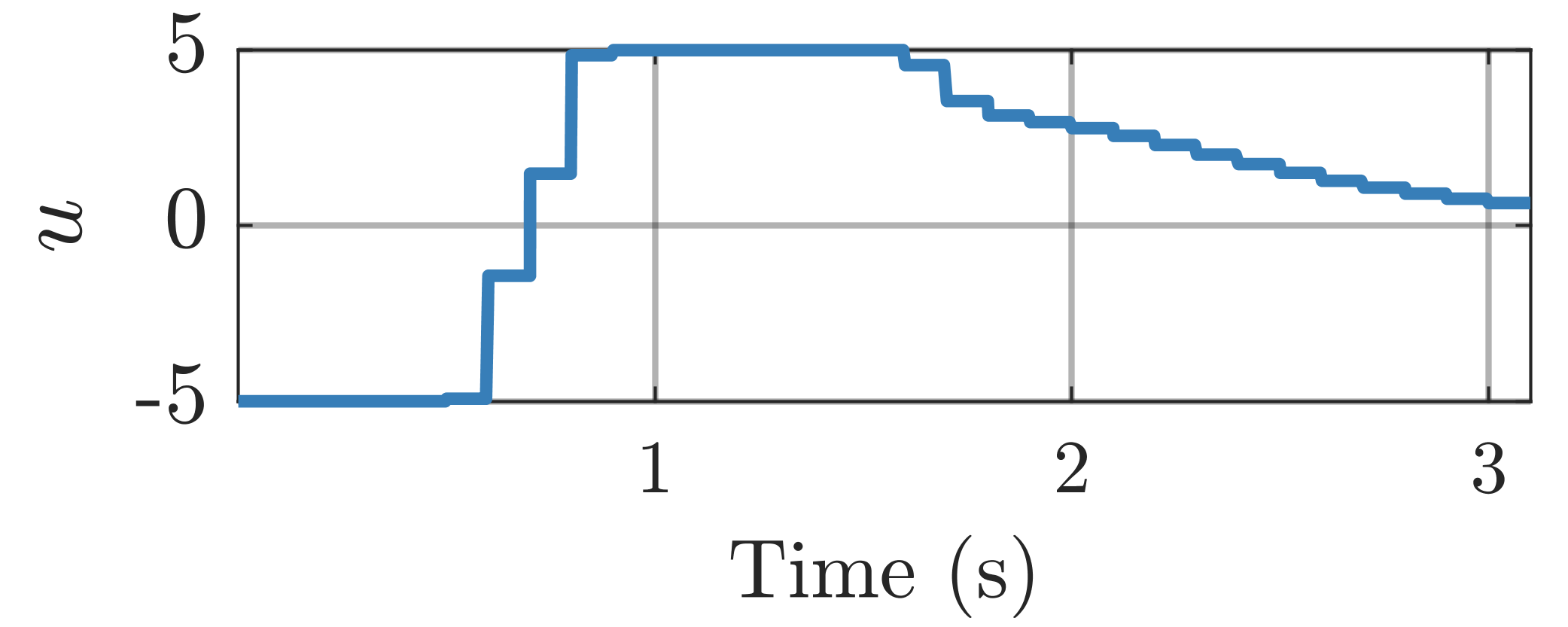}
            \end{minipage}

            \begin{minipage}{0.33\textwidth}
                \centering
                \includegraphics[width=\columnwidth]{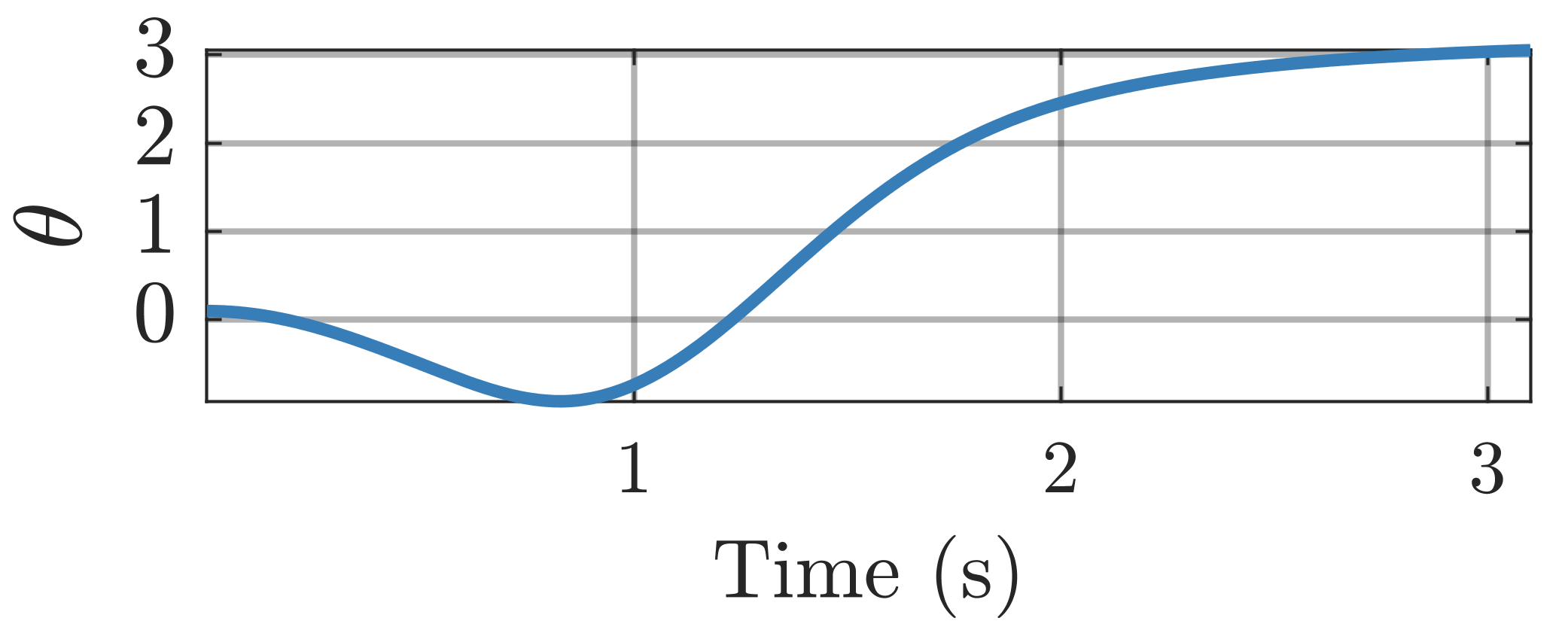}
            \end{minipage}

            \begin{minipage}{0.33\textwidth}
                \centering
                \includegraphics[width=\columnwidth]{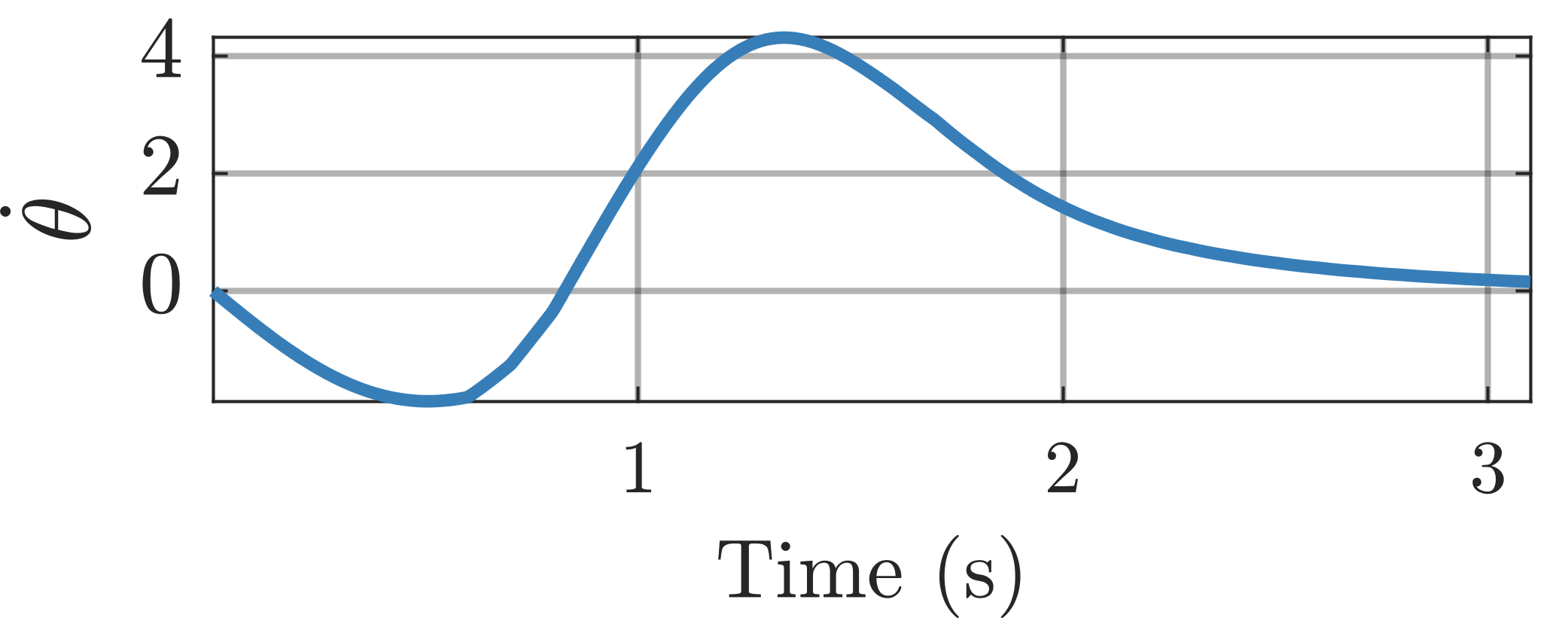}
            \end{minipage}
        \end{tabular}
    \end{minipage}

    \begin{minipage}{\textwidth}
        \centering
        \begin{tabular}{ccc}
            \begin{minipage}{0.35\textwidth}
                \centering
                \includegraphics[width=\columnwidth]{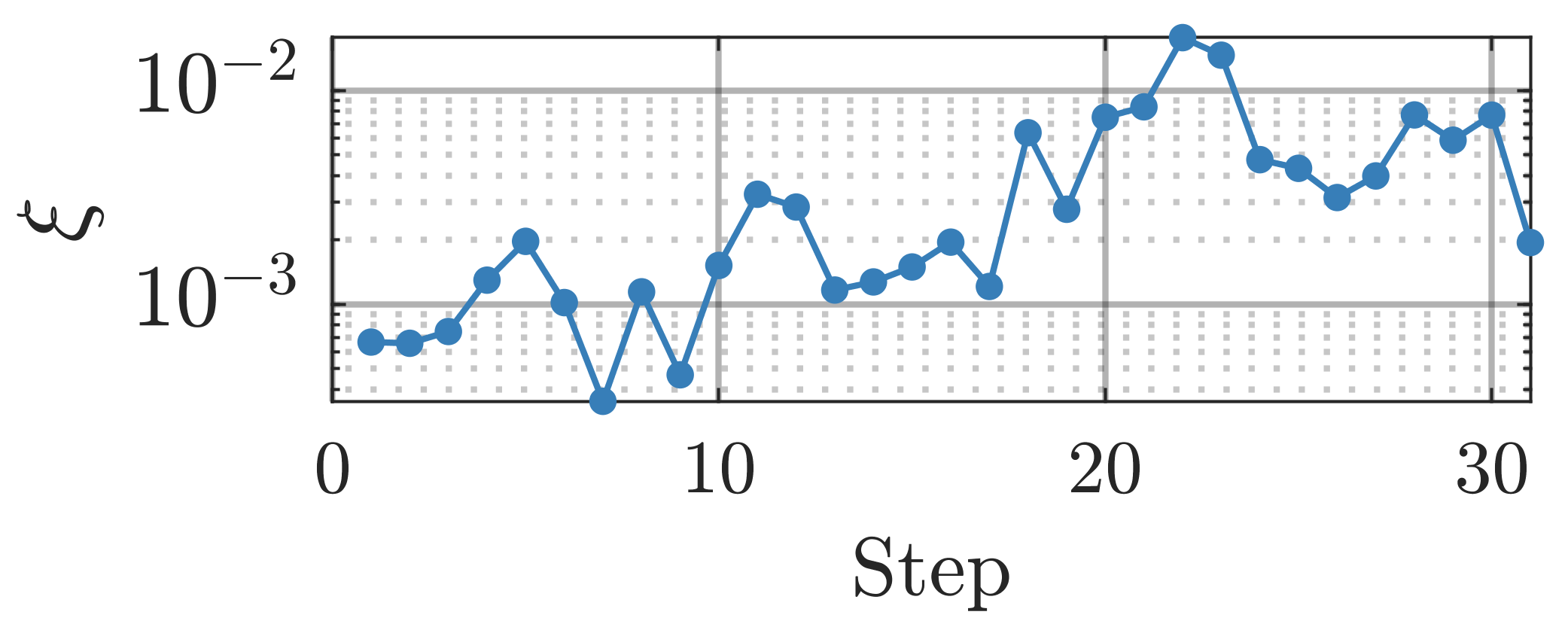}
            \end{minipage}

            \begin{minipage}{0.32\textwidth}
                \centering
                \includegraphics[width=\columnwidth]{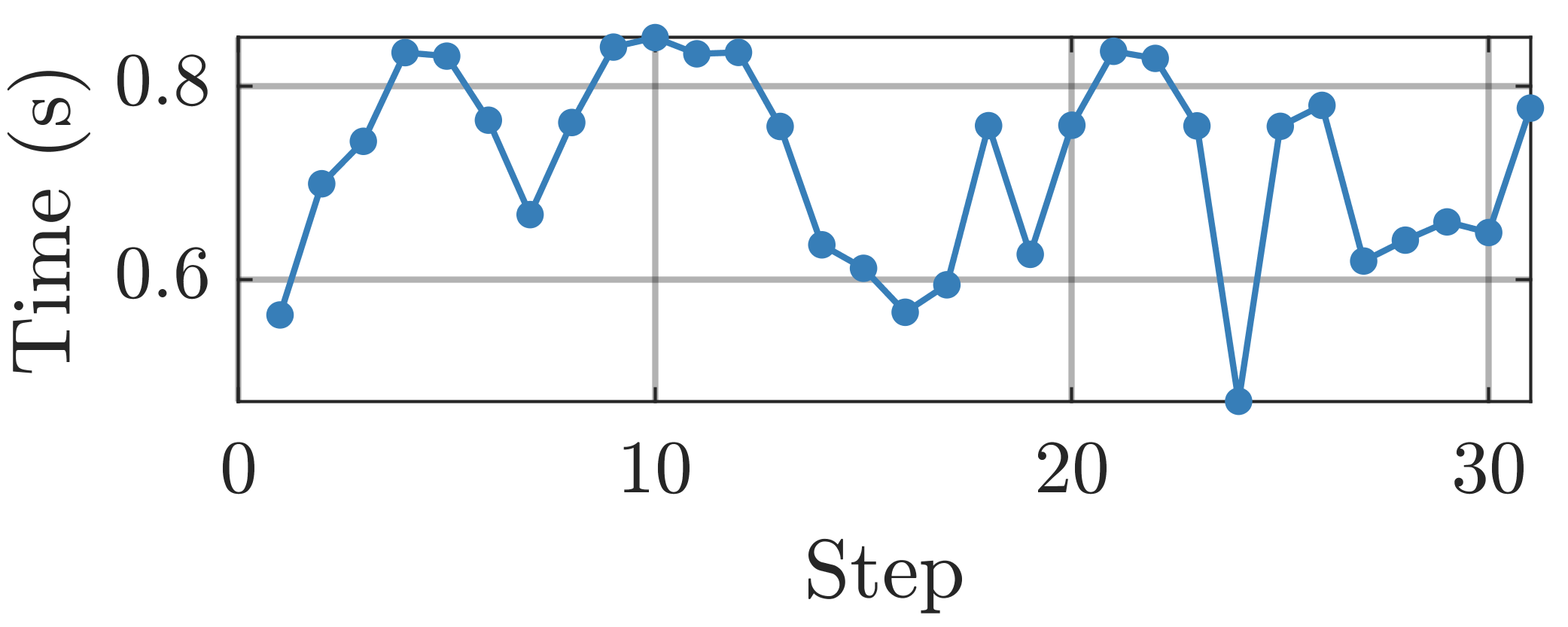}
            \end{minipage}

            \begin{minipage}{0.32\textwidth}
                \centering
                \includegraphics[width=\columnwidth]{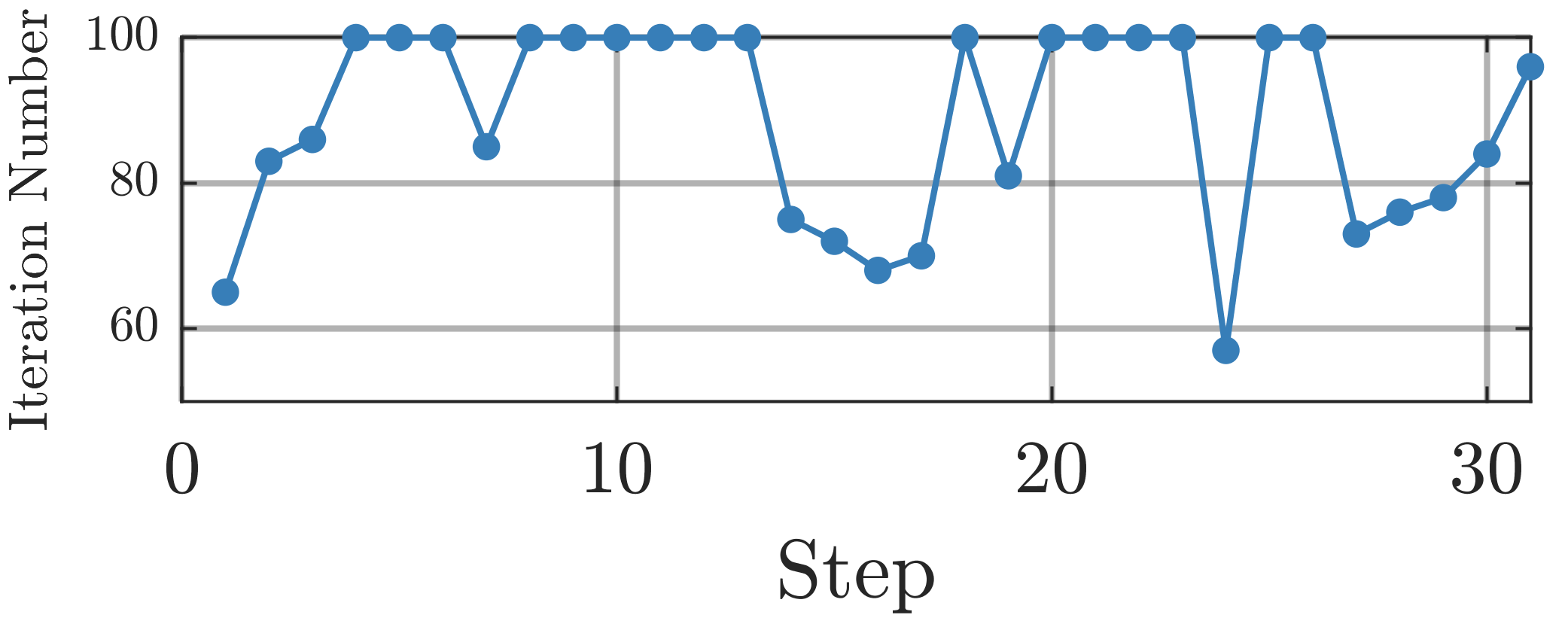}
            \end{minipage}
        \end{tabular}
    \end{minipage}

    \begin{minipage}{\textwidth}
        \centering
        \includegraphics[width=0.9\columnwidth]{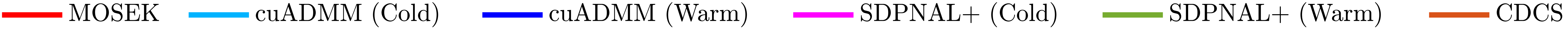}
    \end{minipage}

    \begin{minipage}{\textwidth}
        \centering
        \begin{tabular}{cccc}
            \begin{minipage}{0.25\textwidth}
                \centering
                \includegraphics[width=\columnwidth]{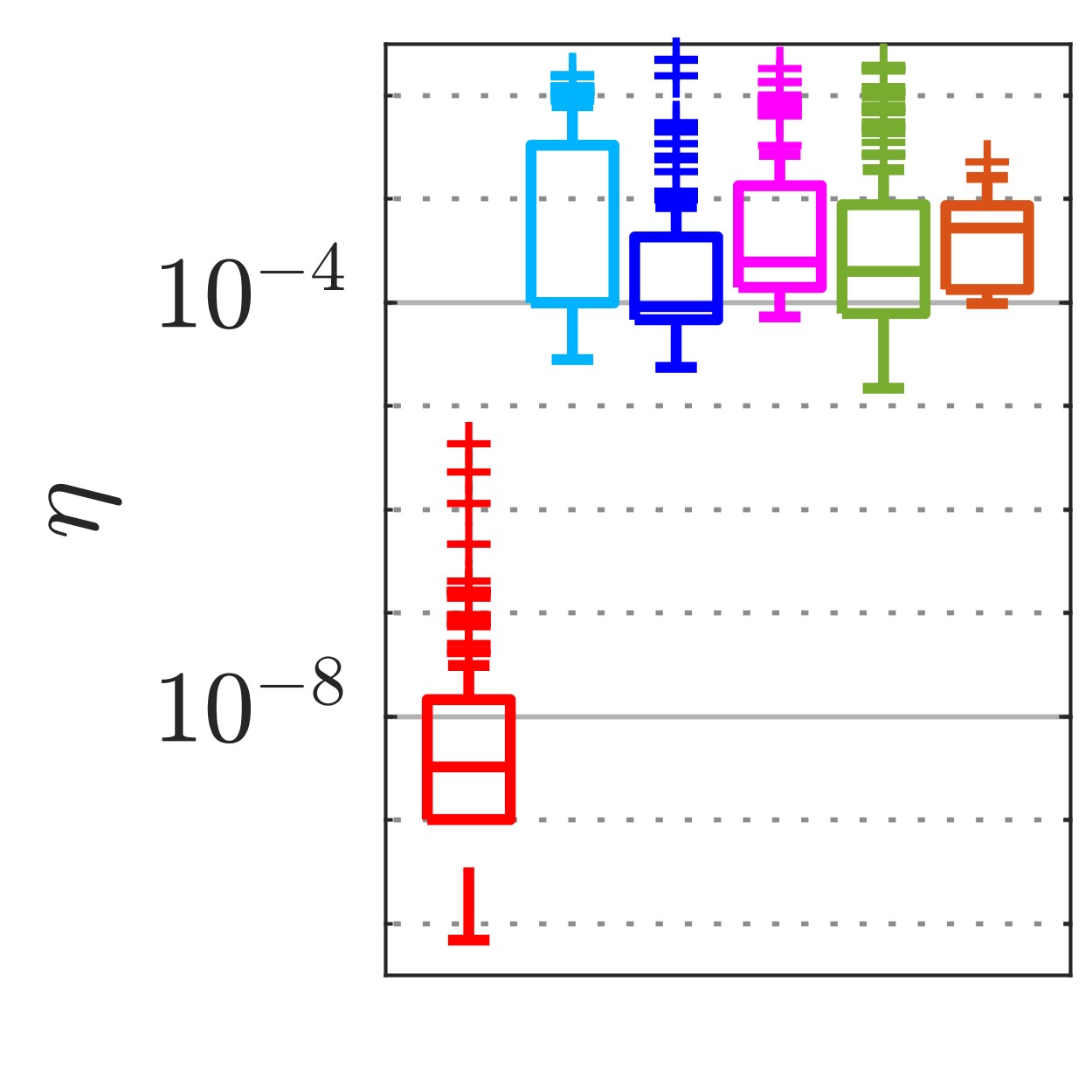}
            \end{minipage}

            \begin{minipage}{0.25\textwidth}
                \centering
                \includegraphics[width=\columnwidth]{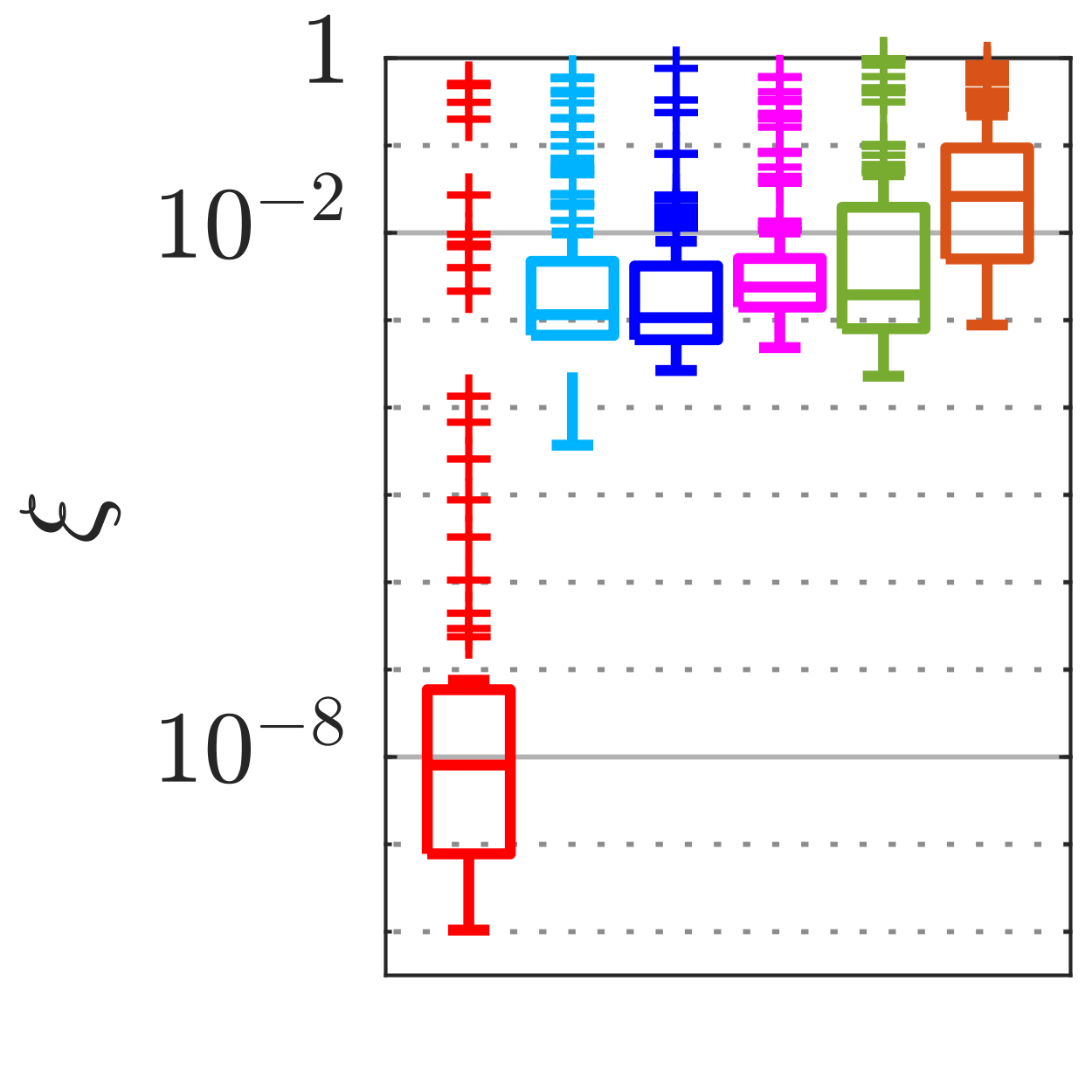}
            \end{minipage}

            \begin{minipage}{0.25\textwidth}
                \centering
                \includegraphics[width=\columnwidth]{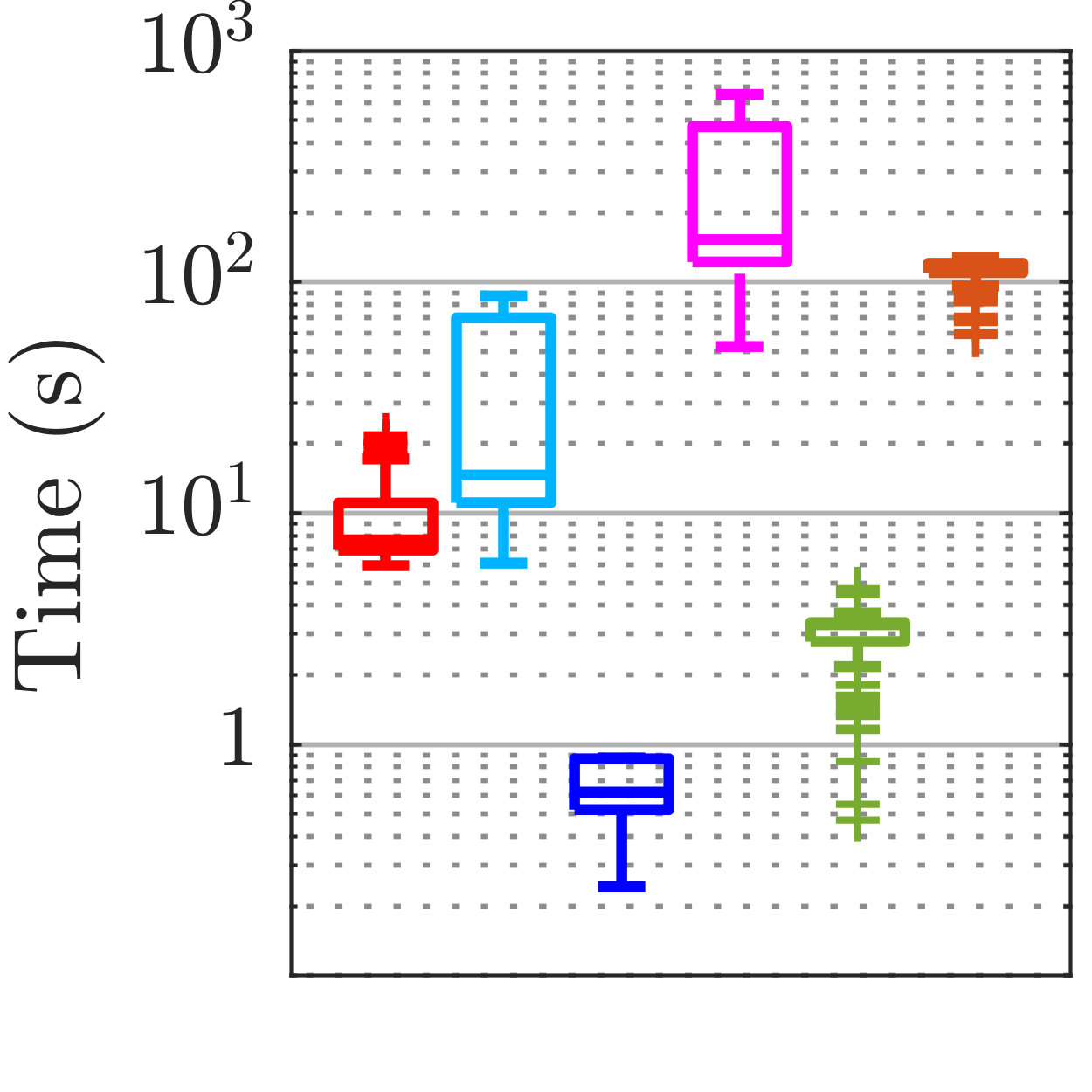}
            \end{minipage}

            \begin{minipage}{0.22\textwidth}
                \centering
                \includegraphics[width=\columnwidth]{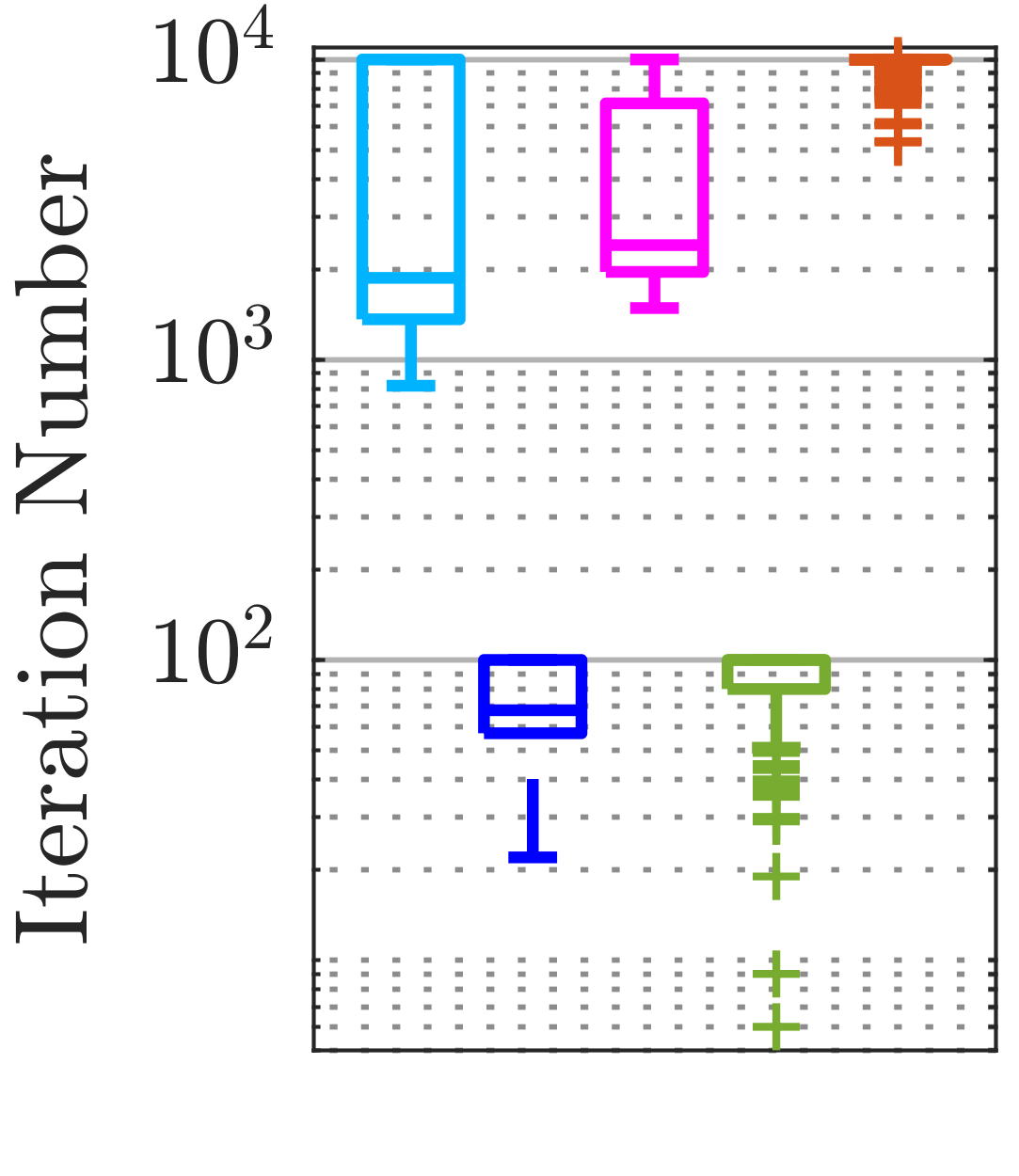}
            \end{minipage}
        \end{tabular}
    \end{minipage}
    \vspace{-6mm}
    \caption{Top: MPC of inverted pendulum with kNN warmstarts. Bottom: performance of six SDP solvers in inverted pendulum.
    \label{fig:exp:p}}
\end{figure}
\vspace{-8mm}


{\bf Car Back-in}. In car back-in, we try to back-up a car (represented as a rectangle) between two square obstacles. To enforce obstacle avoidance using polynomial constraints, we adopt the separating hyperplane method in~\cite{amice2022wafr-collisionfree-robot-manipulation}.
Since the separating hyperplane is in general not unique, the moment matrices will not be rank-one. However, heuristic methods still exist to extract globally optimal trajectories, see all the details in \S\ref{app:subsec:exp:cr}. We randomly select $10$ initial states for four solvers to test their performance. We set the maximal running time to $1$ hour. Results are recorded in the top panel of Fig.~\ref{fig:exp:cr}. \cuadmm is the only solver to finish all experiments in one hour and generate certifiable results ($\xi$ near $10^{-2}$). The bottom panel of Fig.~\ref{fig:exp:cr} shows three globally optimal trajectories. 


\begin{figure}[t]
    \vspace{-3mm}
    \begin{minipage}{\textwidth}
        \centering
        \includegraphics[width=0.7\columnwidth]{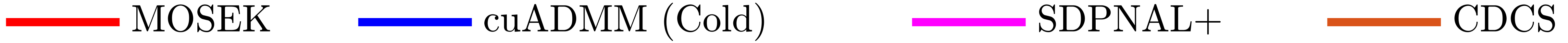}
    \end{minipage}

    \begin{minipage}{\textwidth}
        \centering
        \begin{tabular}{cccc}
            \begin{minipage}{0.25\textwidth}
                \centering
                \includegraphics[width=\columnwidth]{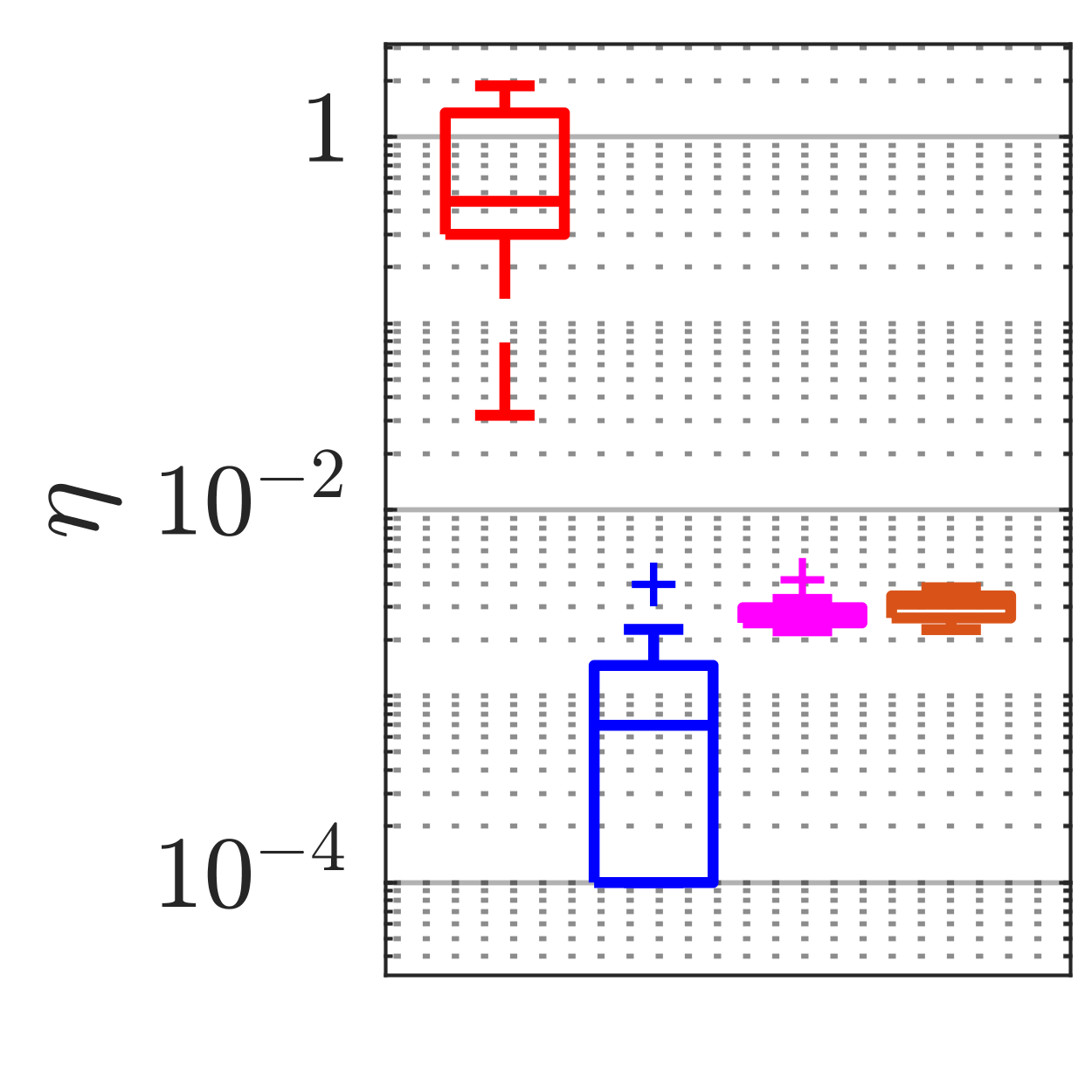}
            \end{minipage}

            \begin{minipage}{0.25\textwidth}
                \centering
                \includegraphics[width=\columnwidth]{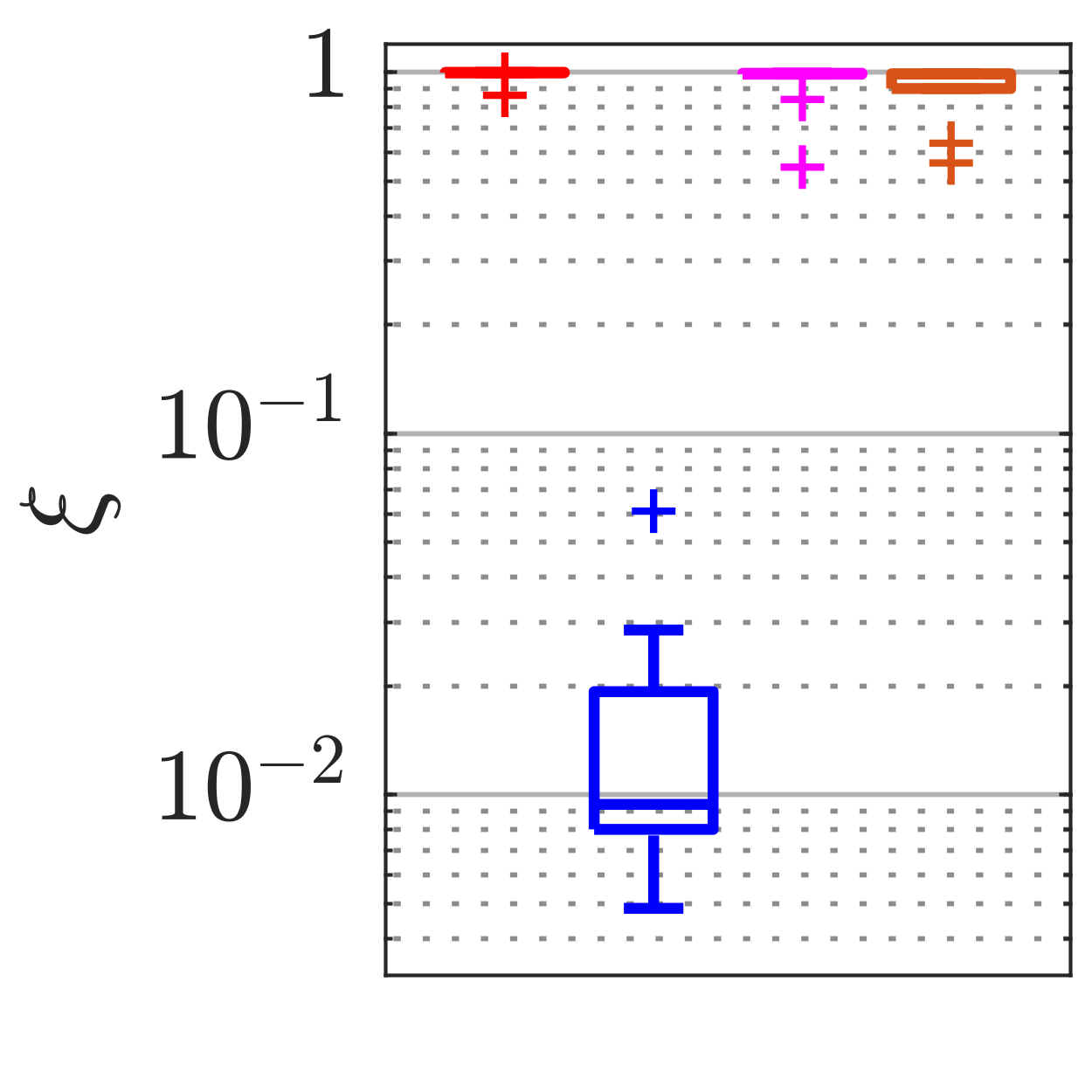}
            \end{minipage}

            \begin{minipage}{0.25\textwidth}
                \centering
                \includegraphics[width=\columnwidth]{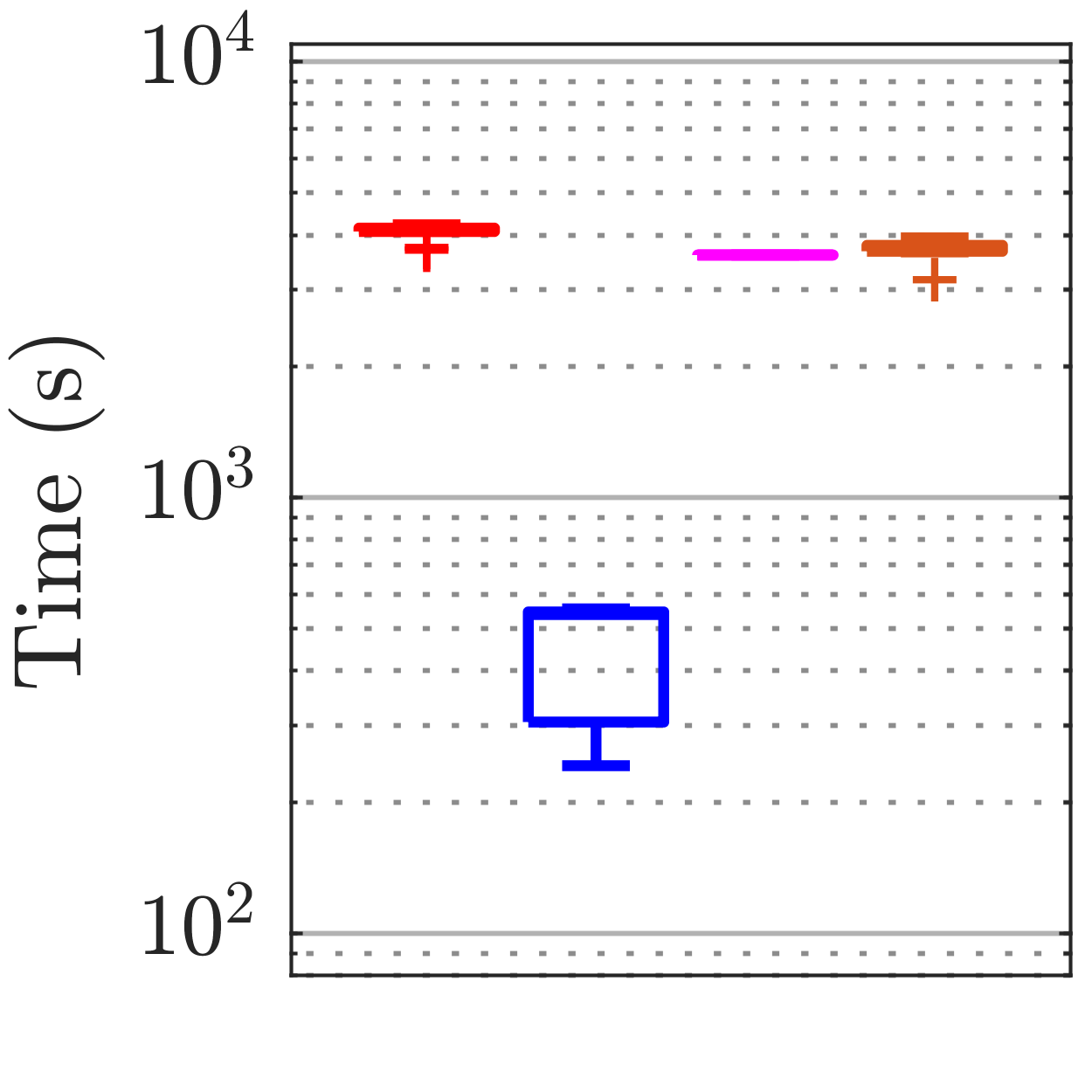}
            \end{minipage}

            \begin{minipage}{0.22\textwidth}
                \centering
                \includegraphics[width=\columnwidth]{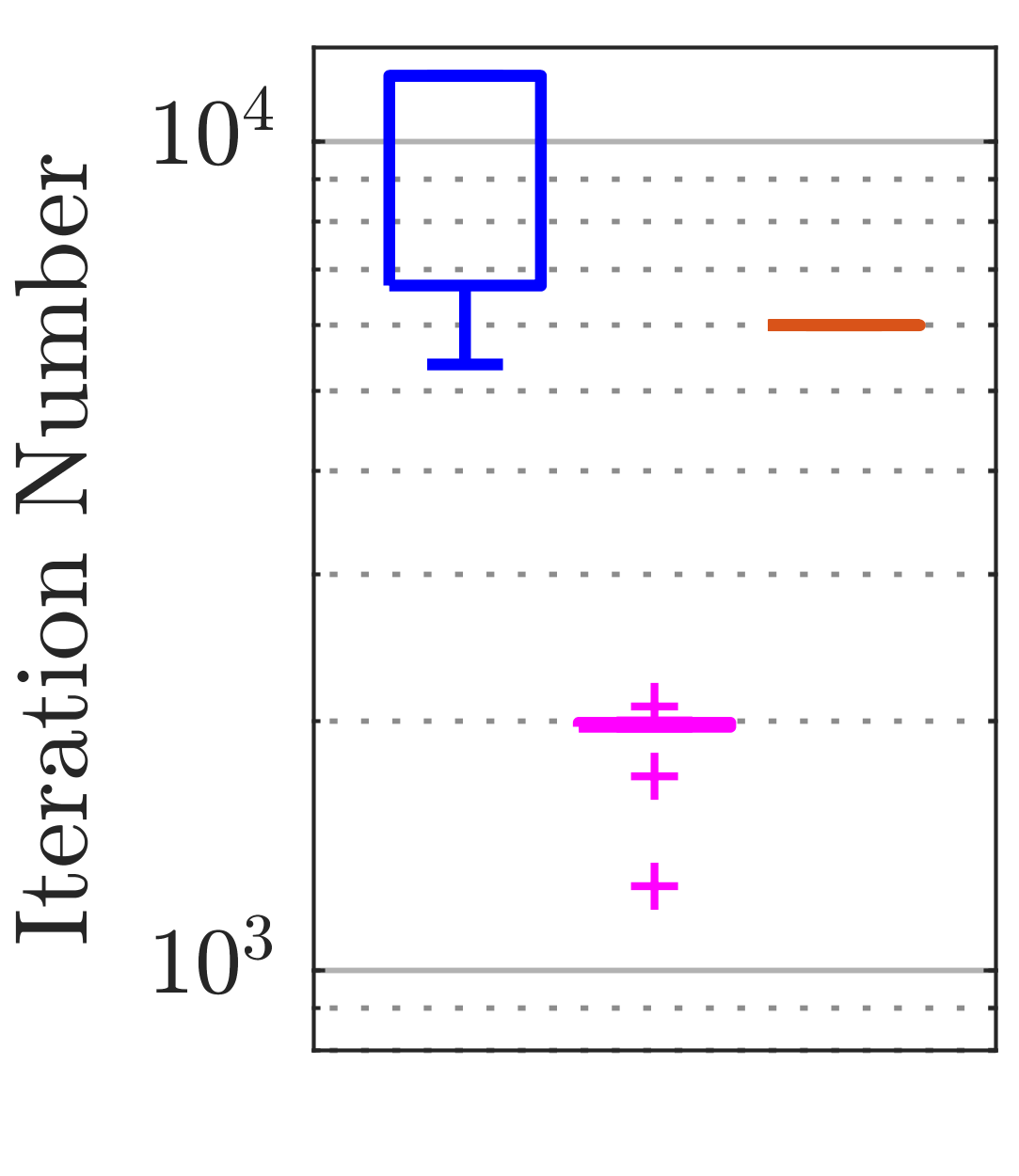}
            \end{minipage}
        \end{tabular}
        \vspace{-3mm}
    \end{minipage}
    \begin{minipage}{\textwidth}
        \centering
        \begin{tabular}{ccc}
            \begin{minipage}{0.33\textwidth}
                \centering
                \vspace{0.5mm}
                \includegraphics[width=\columnwidth]{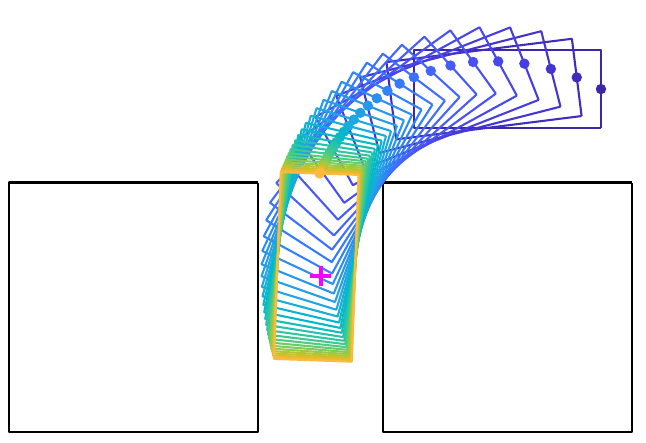}
            \end{minipage}

            \begin{minipage}{0.33\textwidth}
                \centering
                \includegraphics[width=\columnwidth]{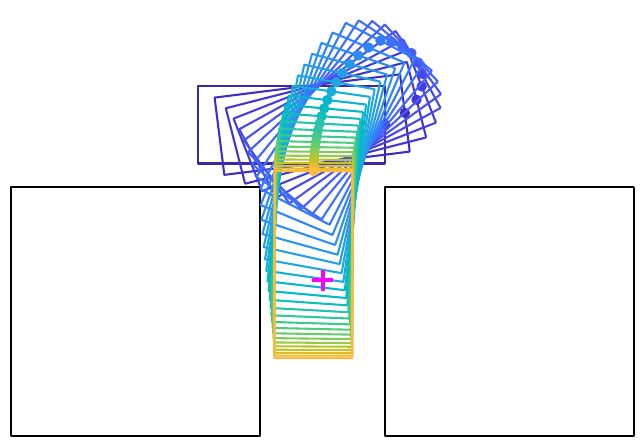}
            \end{minipage}

            \begin{minipage}{0.33\textwidth}
                \centering
                \vspace{0.5mm}
                \includegraphics[width=\columnwidth]{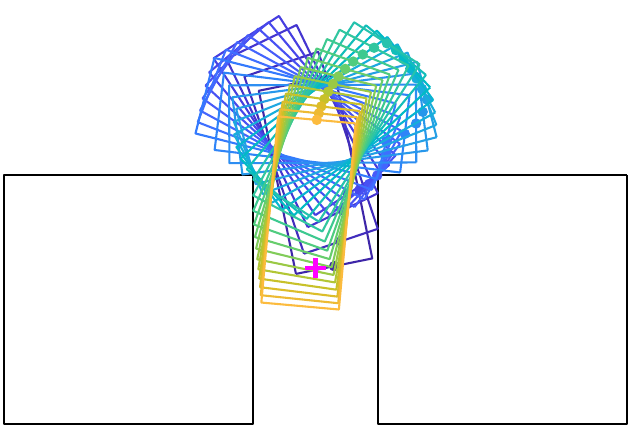}
            \end{minipage}
        \end{tabular}
    \end{minipage}
    \vspace{-5mm}
    \caption{Top: Performance of four SDP solvers in the car back-in problem. Bottom: Three globally optimal trajectories, where dot indicates front of vehicle, and from blue to green indicates from start to end of trajectory. 
    \label{fig:exp:cr}}
    \vspace{1mm}
\end{figure}


{\bf Flying Robot}.
In this experiment, we want to send a flying robot with four boosters to the origin. The vehicle has six continuous-time states: 2-D position $(x, y)$, 2-D velocity $(\dot{x}, \dot{y})$, angular position $\theta$ and velocity $\dot{\theta}$. The discretization procedure is similar to the pendulum case. Fig.~\ref{fig:exp:fr:demos-plots} shows two globally optimal trajectories. Note that \cuadmm is the only solver capable of computing optimal trajectories in this problem (\cf Table~\ref{tab:exp:gen:onlyone}).

\vspace{-6mm}
\begin{figure}[h]
    \label{fig:exp:fr:demos-plots}
    \hspace{-10mm}
    \begin{minipage}{\textwidth}
        \centering
        \begin{tabular}{cccc}
            \begin{minipage}{0.25\textwidth}
                \centering
                \includegraphics[width=\columnwidth]{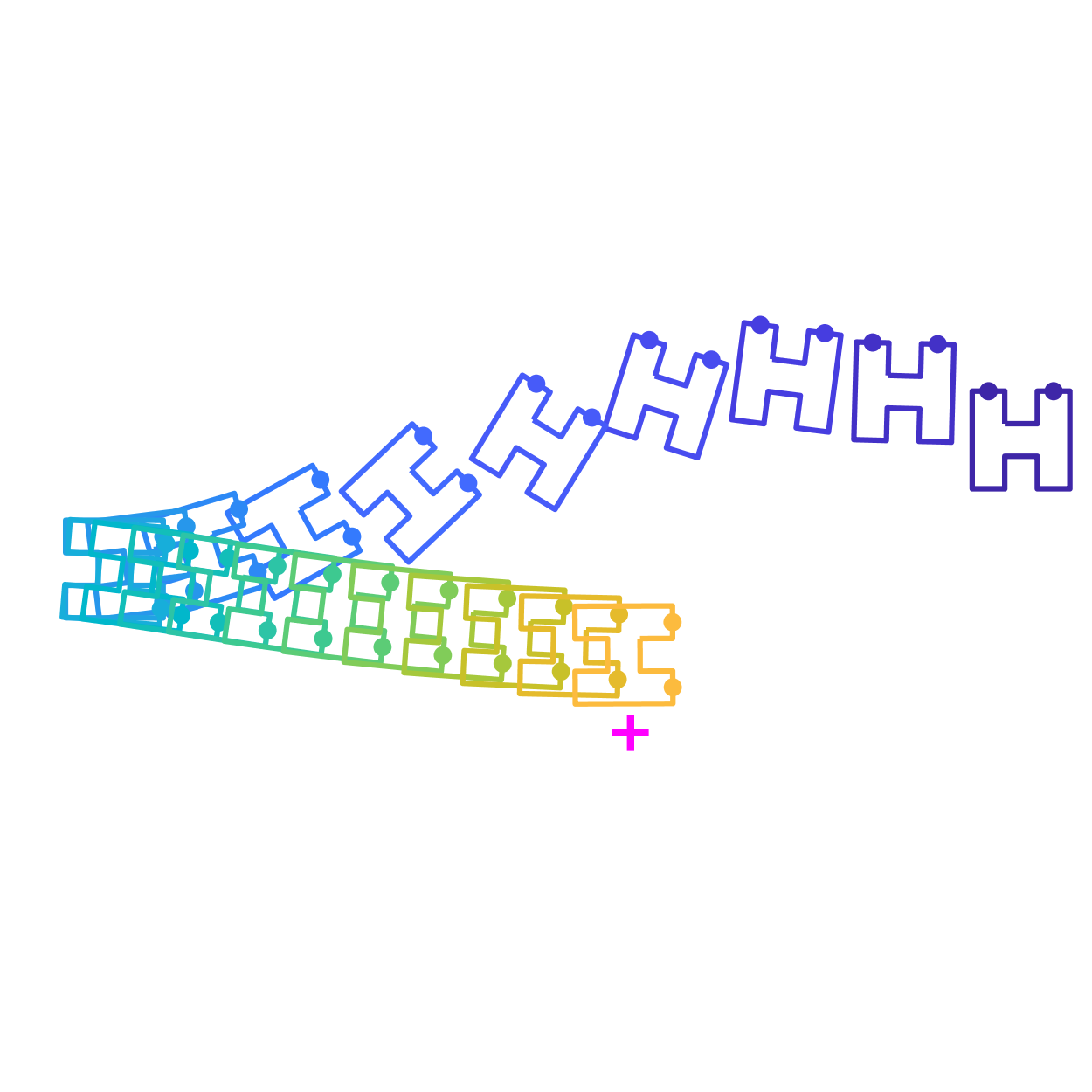}
            \end{minipage}

            \begin{minipage}{0.3\textwidth}
                \includegraphics[width=\columnwidth]{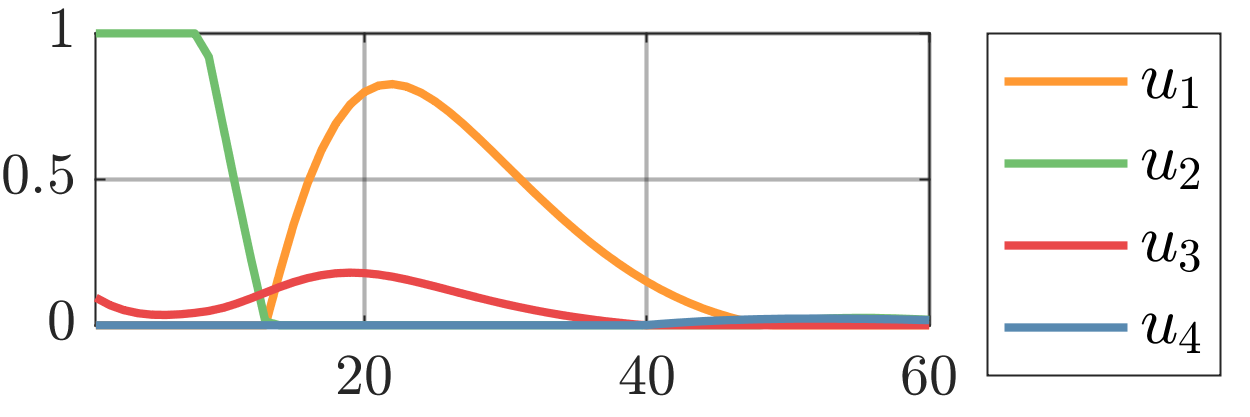}
                \includegraphics[width=\columnwidth]{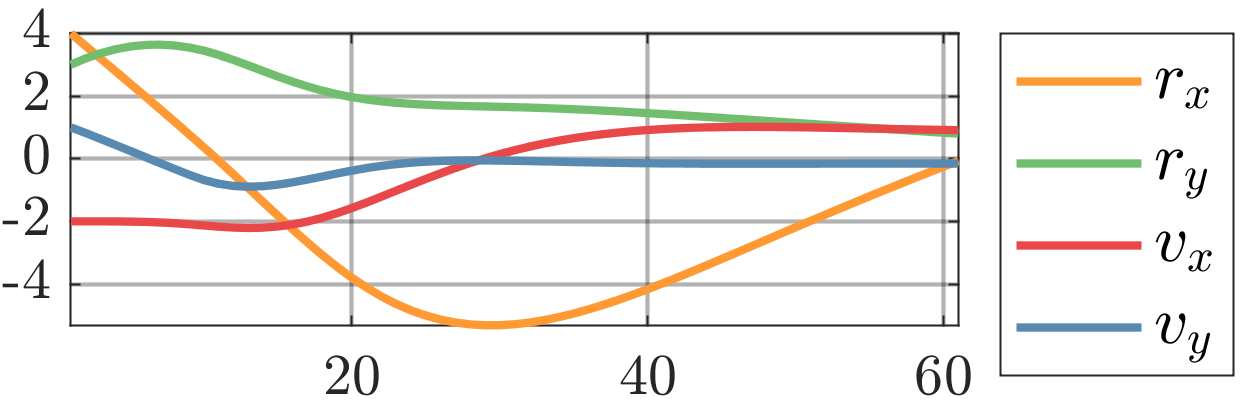}
                \includegraphics[width=\columnwidth]{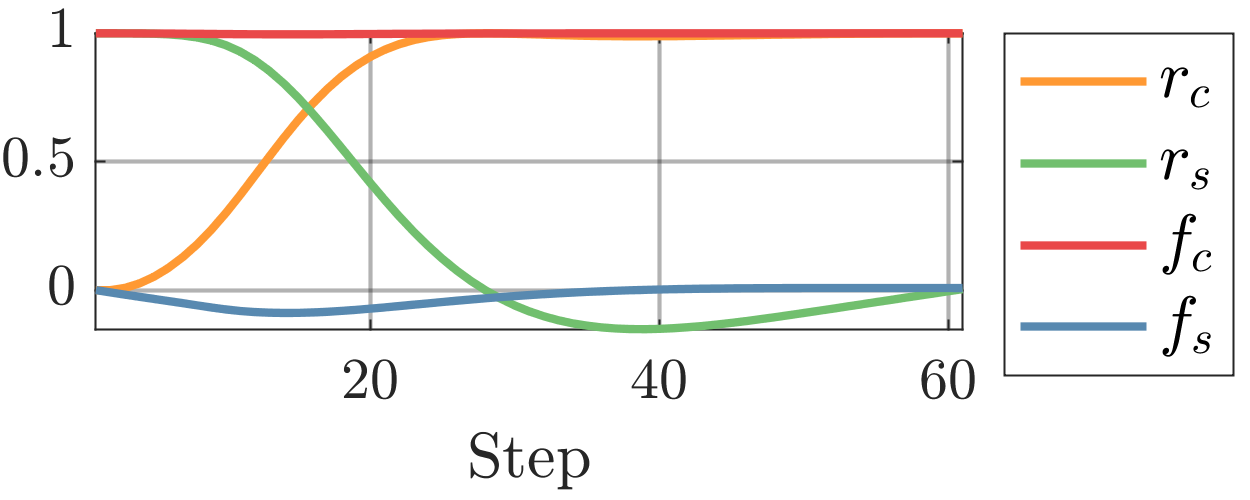}
            \end{minipage}

            \begin{minipage}{0.25\textwidth}
                \centering
                \includegraphics[width=\columnwidth]{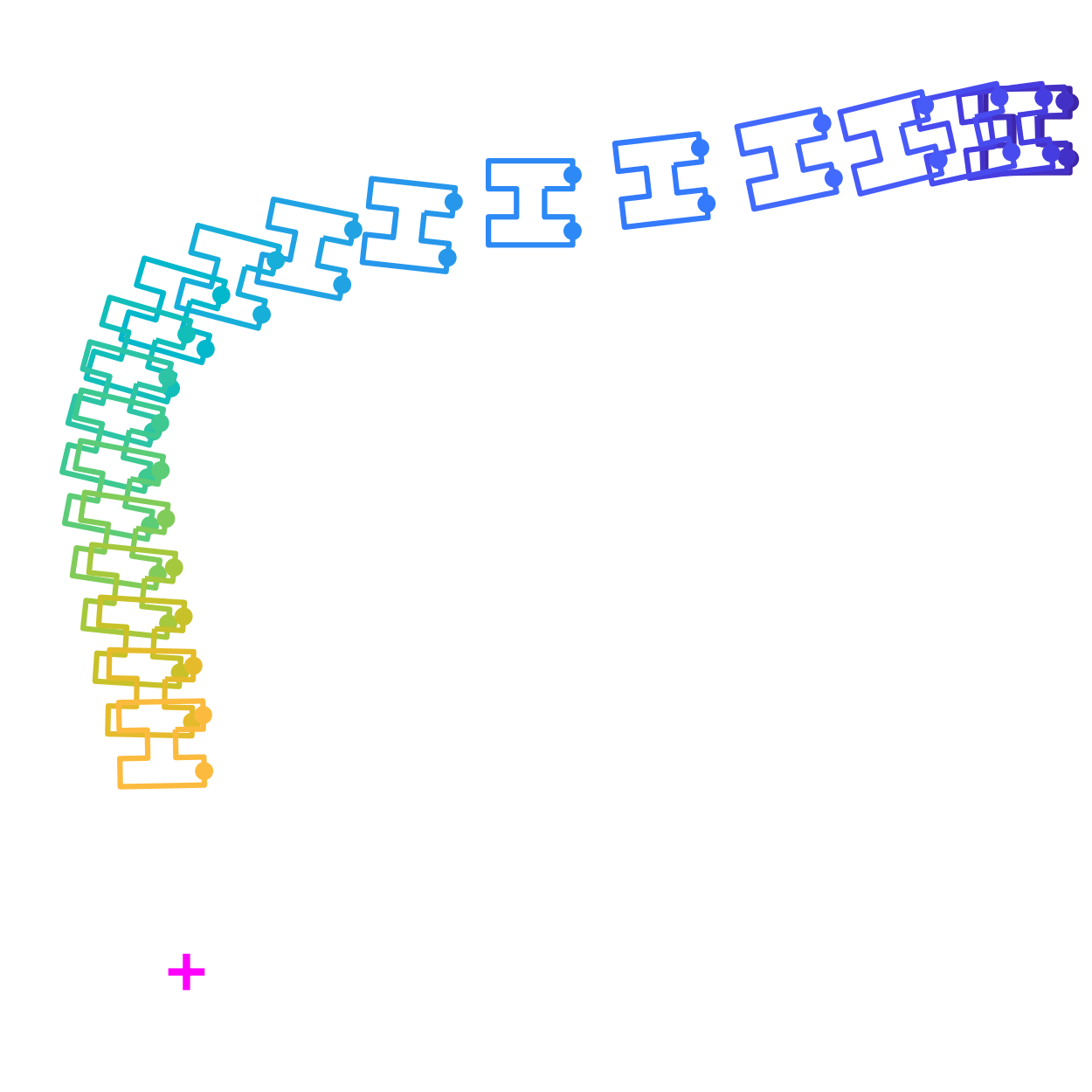}
            \end{minipage}

            \begin{minipage}{0.3\textwidth}
                \includegraphics[width=\columnwidth]{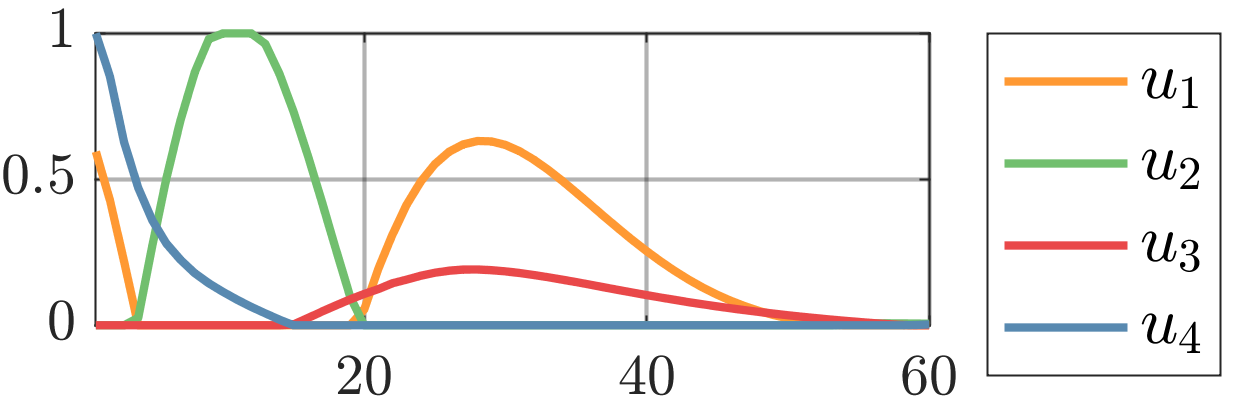}
                \includegraphics[width=\columnwidth]{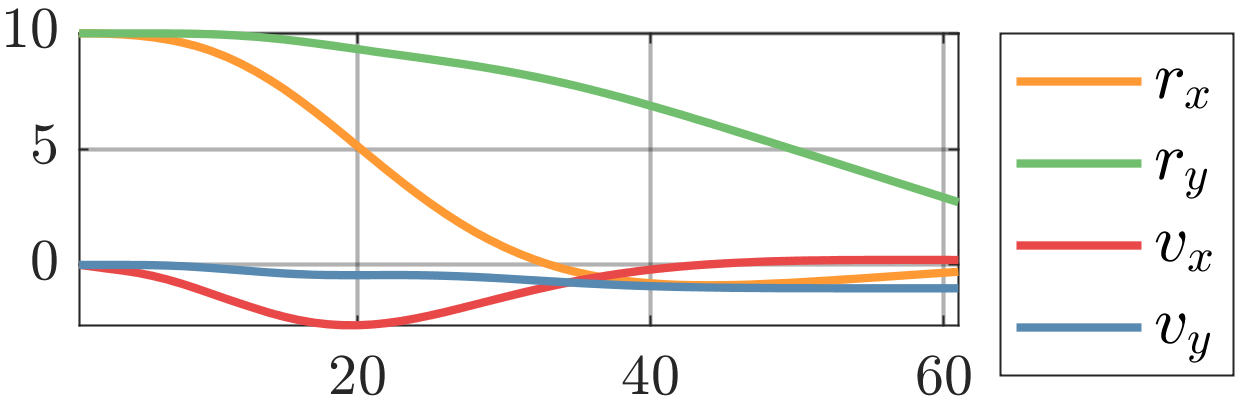}
                \includegraphics[width=\columnwidth]{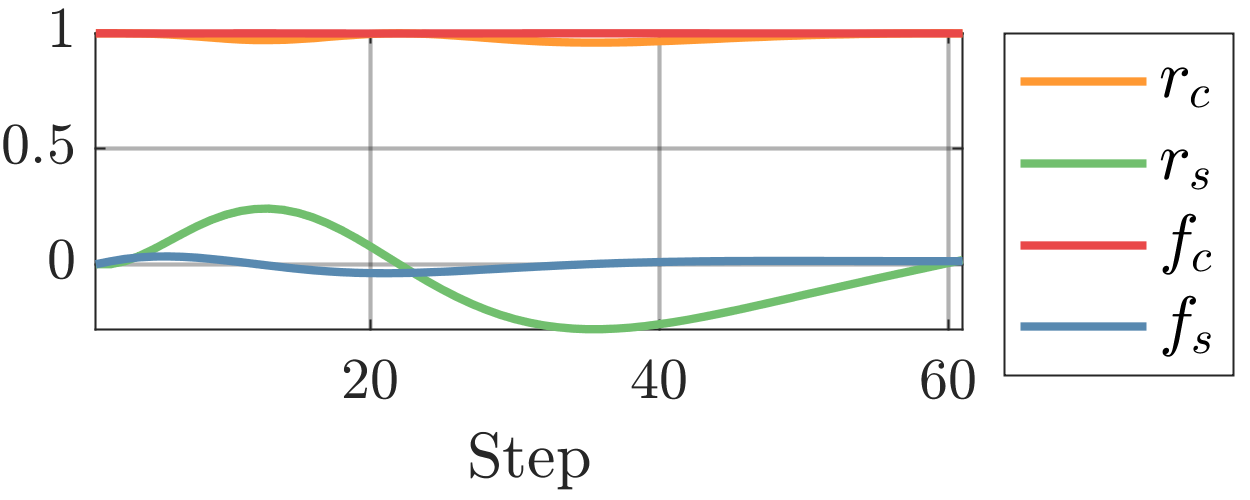}
            \end{minipage}
        \end{tabular}
    \end{minipage}

    \vspace{-3mm}
    \caption{Two globally optimal trajectories for the flying robot problem.}
    \vspace{-8mm}
\end{figure}


\section{Conclusion}
\label{sec:conclusion}
We presented \nameshort, a new framework for fast and certifiable trajectory optimization. \nameshort contains two modules: a C++ package that generates sparse moment relaxations, and a first-order ADMM-based SDP solver \cuadmm directly implemented in CUDA. Our C++ package is two orders of magnitude faster than existing Matlab packages, and our \cuadmm solves large-scale SDPs far beyond the reach of existing solvers. Moreover, we demonstrated the potential of real-time certifiable trajectory optimization in inverted pendulum using data-driven warmstarts. Several directions are worth exploring in future works, such as moment-cone-specific conic solvers and differentiable ADMM for integration with deep learning. We believe the time is now to make Lasserre's moment-SOS hierarchy a practical computational tool for robotics.


\clearpage
{\bf Acknowledgments}. We thank Tianyun Tang, Kim-Chuan Toh, Jie Wang and Jean B. Lasserre for discussion about trajectory optimization and moment-SOS hierarchy. We thank Xin Jiang for discussion about ADMM. We thank Brian Plancher and Emre Adabag for offering advice for CUDA programming.

%
%
%
\bibliographystyle{splncs04}
\bibliography{refs.bib}
%

\appendix

\section{Related Works}
\label{sec:relatedworks}

\subsection{Trajectory optimization: global solvers}
To achieve global optimality in nonconvex trajectory optimization, semidefinite relaxation techniques are widely adopted~\cite{khadir2021icra-piecewiselinear-motionplanning,teng2023arxiv-geometricmotionplanning-liegroup,huang2024arxiv-sparsehomogenization,graesdal2024arxiv-tightconvexrelax-contactrich}. \cite{khadir2021icra-piecewiselinear-motionplanning} uses a piecewise linear kinematics model and Moment-SOS Hierarchy to search for a collision-free path asymptotically, but finite convergence is not observed. \cite{teng2023arxiv-geometricmotionplanning-liegroup} first reports finite convergence when applying sparse second-order Moment-SOS Hierarchy to geometric motion planning. \cite{huang2024arxiv-sparsehomogenization} generalizes the results from \cite{teng2023arxiv-geometricmotionplanning-liegroup} to unbounded domains using sparse homogenization techniques. \cite{graesdal2024arxiv-tightconvexrelax-contactrich} reformulates the contact-rich manipulation task as a nonconvex quadratic constrained quadratic programming (QCQP) problem, then applies Shor's relaxation with redundant constraints, though the suboptimality gap remains around $10\%$. This line of works mainly focus on problem reformulation and using existing solvers for computing solutions, leaving the design of problem-specific SDP solvers to facilitate computation as an untapped area. Our work precisely aims at improving the computational solvers.
Another line of research focuses on mixed integer programming~\cite{marcucci2020arxiv-warmstart-mixedinteger-mpc,deits2015icra-mixedintegerprogramming-uav,ding2020iros-motionplanning-multilegged-mixedinteger}. Due to the combinatorial nature of the problem formulation, these methods do not scale well for long time horizon problems. Notably, some sampling-based planning methods~\cite{karaman2011ijrr-samplingbased-optimalmotionplanning,lavalle2006book-planningalgorithms}, such as RRT$^\star$, can also be viewed as global solvers, being probabilistically complete.

\subsection{Trajectory optimization: other methods}
In trajectory optimization, a common approach is to use or develop gradient-based local solvers to find local solutions~\cite{posa2014ijrr-traopt-directmethod-contact,kelly2017siam-intro-directcollocation,howell2019iros-altro,howell2022ral-trajopt-optimizationbased-dynamics,adabag2024arxiv-mpcgpu}\shucheng{\cite{yunt2007combined,yunt2006-opttraj-planning-structure-variant}}. Various nonlinear programming algorithms are empolyed, such as sequential quadratic programming (SQP)~\cite{posa2014ijrr-traopt-directmethod-contact}, interior point method~\cite{howell2022ral-trajopt-optimizationbased-dynamics}, augmented Lagrangian method (ALM)~\cite{howell2019iros-altro}, and preconditioned conjugate gradient (CG)~\cite{adabag2024arxiv-mpcgpu}. Zeroth-order methods, like evolutionary algorithms~\cite{hyatt2020ral-realtime-nmpc-gpu}, stochastic optimization methods~\cite{kalakrishnan2011irca-stomp}, and rule-based planning~\cite{heinrich2015iros-traopt-motionuncertainty-gpu}, also complement gradient-based methods. Despite the rise of global solvers, local solvers remain essential for certifying global optimality by providing upper bounds of the globally optimal objective values. Indeed, our \nameshort framework does not aim to replace local solvers, but instead uses local solvers to help compute a suboptimality certificate.

\subsection{Polynomial optimization and Moment-SOS hierarchy}
As a general modeling and optimization tool, polynomial optimization problems (POPs) are prevalent in various fields such as control~\cite{majumdar2013icra-controldedsign-sos}, perception~\cite{yang2022pami-outlierrobust-geometricperception}, and optimal transport~\cite{mula2022arxiv-momentsos-optimal-transport}. 
The celebrated Moment-SOS Hierarchy~\cite{lasserre2001siopt-global} offers a powerful tool to solve POP to global optimality asymptotically. By leveraging the duality between moment relaxation and sum-of-squares (SOS) relaxation, POP is solved by a series of monotonically growing SDPs. Under certain constraint qualification conditions~\cite{nie2023siopt-moment-momentpolynomialopt}, finite convergence can be observed~\cite{yang2022pami-outlierrobust-geometricperception,teng2023arxiv-geometricmotionplanning-liegroup}. To enhance scalability, researchers exploit various sparsity patterns, including correlative sparsity~\cite{lasserre2006msc-correlativesparse}, term sparsity~\cite{wang2021siam-tssos}, and ideal sparsity~\cite{korda2023mp-ideal-sparsity}, which significantly improves computation. \shucheng{Notably,~\cite{fantuzzi2024siam-global-pop-integral-functionals} utilizes correlative sparsity in POP to efficiently solve PDEs.}

\subsection{Semidefinite programming (SDP) solvers}
For small to medium-scale general SDPs, the interior point method~\cite{helmberg1996siam-interiorpoint-sdp} is the preferred choice due to its robustness and high accuracy, with several academic and commercial solvers available~\cite{tutuncu2003mp-sdpt3-sdpsolver,sturm1999oms-sedumi-sdpsolver,aps2019ugrm-mosek-sdpsolver}. As SDP size increases, first-order methods, particularly ADMM-style algorithms~\cite{wen2010mp-admmsdp,chen2017mp-sgsadmm,li2018mp-sgs-ccqp,chen2021mp-alm-admm-equivalence}, become more popular due to their affordable per-iteration time and memory cost. Among these, sGS-ADMM~\cite{chen2017mp-sgsadmm,li2018mp-sgs-ccqp} demonstrates superior empirical performance. Several academic first-order SDP solvers are available~\cite{odonoghue2023-scs-sdpsolver,zheng2017ifac-cdcs-sdpsolver,yang2015mp-sdpnalplus-sdpsolver}. In real-world applications~\cite{yang2022pami-outlierrobust-geometricperception,teng2023arxiv-geometricmotionplanning-liegroup}, many SDPs exhibit a low-rank property, prompting the development of specialized algorithms~\cite{burer2003springer-bm,tang2023arxiv-feasible-lowranksdp} and solvers~\cite{yang2023mp-stride,wang2023arxiv-manisdp}. However, these algorithms often underperform for large-scale multiple-cone SDPs. While GPU-based interior point SDP solvers have been deployed on supercomputers~\cite{fujisawa2012sc-extremely-sdp-interiorpoint}, all existing first-order SDP solvers are currently implemented on CPUs.


\section{Insufficiency of First-Order Relaxation}
\label{app:sec:firstorder}

We randomly selected $10$ initial states in the pendulum problem to compare the tightness of sparse first-order and second-order relaxations. The mean and median of the SDP objective values for sparse first-order relaxation were $13.7$ and $13.8$, respectively. Using first-order relaxation, none of the optimal SDP solutions (\ie the moment matrices) led to successful rounding of a feasible solution in \fmincon. In stark contrast, for sparse second-order relaxation, the mean and median of the optimal SDP objective value were $61.3$ and $59.2$, respectively, much higher than the first-order relaxation. Moreover, all of the second-order relaxation solutions come with a suboptimality gap of less than $10^{-2}$.  

We also note that first-order relaxation cannot handle polynomial optimization problems beyond quadratically constrained quadratic programming (QCQP), \eg Car Back-in problem in \S\ref{sec:exp} and Van der Pol problem in~\cite{huang2024arxiv-sparsehomogenization}.

\section{Existence of $R_\beta$ for Suboptimality~\eqref{eq:strom:sgsadmm:suboptimality-gap}}
\label{app:sec:existenceRbeta}

The certificate~\eqref{eq:strom:sgsadmm:suboptimality-gap} holds if an upper bound $R_\beta$ exists for the trace of each $X(\hat{z})_\beta$. The next theorem guarantees the existence of $R_\beta$. 

\begin{theorem}[Existence of $R_\beta$]
    \label{thm:strom:sgsadmm:existence-rbeta}
    If $\norm{\dvar(I_k)}_\infty \le R_k, \forall k \in \seqordering{N}$. Then, for $X_\beta$ corresponding to the moment matrices $M_k$, one can pick 
    \begin{align}
        R_\beta = s(|I_k|, \kappa) \cdot R_k^2,
    \end{align} 
    and for $X_\beta$ corresponding to the localizing matrices $L_{k,i}$, one can pick 
    \begin{align}
        R_\beta = \max_{\norm{\dvar(I_k)}_\infty \le R_k} \left\{ g_{k,i}(\dvar(I_k)) \right\}  \cdot s(|I_k|, \kappa - \degg{k,i}) \cdot R_k^2
    \end{align}
\end{theorem}

\begin{proof}
    Given arbitrary $\dvar$ satisfies the assumption in~\ref{thm:strom:sgsadmm:existence-rbeta}. For moment matrices:
    \begin{align}
        \trace{M_k(\dvar)} = \trace{
            \basis{\dvar(I_k)}{\kappa} \basis{\dvar(I_k)}{\kappa}\tran
        } = \norm{\basis{\dvar(I_k)}{\kappa}}_2^2 \le s(|I_k|, \kappa) \cdot R_k 
    \end{align}
    For localizing matrices:
    \begin{subequations}
        \begin{align}
            \trace{L_{k,i}(\dvar)} & = \trace{
                g_{k,i}(\dvar(I_k)) \cdot \basis{\dvar(I_k)}{\kappa - \degg{k,i}} \basis{\dvar(I_k)}{\kappa - \degg{k,i}}\tran
            } \\
            & = g_{k,i}(\dvar(I_k)) \cdot \norm{\basis{\dvar(I_k)}{\kappa - \degg{k,i}}}_2^2 \\
            & \le \max_{\norm{\dvar(I_k)}_\infty \le R_k} \left\{ g_{k,i}(\dvar(I_k)) \right\}  \cdot s(|I_k|, \kappa - \degg{k,i}) \cdot R_k^2
        \end{align}
    \end{subequations}
\end{proof}

Essentially, as long as the variables in the original trajectory optimization~\eqref{eq:intro:trajopt} are bounded, we can compute a certificate of suboptimality from the output of \sgsadmm according to~\eqref{eq:strom:sgsadmm:suboptimality-gap}.

\section{POP-SDP Conversion Package Comparison}
\label{app:sec:conversion}
We compare our sparse moment relaxation package (Matlab interface) with three state-of-the-art SOS relaxation packages. \sostools and \yalmip~\cite{lofberg2004cacsd-yalmip} are implemented in Matlab, while \tssos is implemented in Julia. We evaluate the performance of these four conversion packages on the pendulum problem, as shown in Table~\ref{tab:exp:gen:conversion-speed}. Our conversion package is the only one capable of converting the pendulum problem to an SDP in real-time in MATLAB.


\begin{table}[t]
    \centering
    \begin{tabular}{|c|c|c|c|c|}
        \hline
        & \yalmip (Matlab) & \sostools (Matlab) & \tssos (Julia) & Ours (Matlab) \\
        \hline
        $N = 4$ & $8.18$s & $3.39$s & $2.40$s & $\bm{0.09}$\textbf{s} \\
        \hline 
        $N = 30$ & $68.89$s & $>30$min & $2.51$s & $\bm{0.32}$\textbf{s}  \\
        \hline
    \end{tabular}
    \vspace{1mm}
    \caption{POP-SDP conversion time of $4$ conversion packages in the inverted pendulum problem. When $N = 30$, \sostools takes more than $30$ miniutes and takes more than $300$GB memory.
    \label{tab:exp:gen:conversion-speed}}
\end{table}

\section{Experimental Details}
\label{app:sec:exp}

\subsection{Inverted Pendulum}
\label{app:subsec:exp:p}

\textbf{Dynamics and constraints.} For a pendulum with continuous dynamics
\begin{align}
    \label{eq:exp:p:con-dyn}
    ml^2 \ddot{\theta} = u - mgl \sin \theta - b \dot{\theta}
\end{align}
From~\cite{lee2008thesis-computationalgeometricmechanics}, the discretized nonlinear dynamics $F_k$ in~\eqref{eq:intro:trajopt-generalform-dyn} and constraints $\calC_k$ in~\eqref{eq:intro:trajopt-generalform-con} at time step $k$ are:
\begin{subequations}
    \begin{align}
        \label{eq:exp:p:dis-dyn-constraints}
        & m l^2 \cdot \frac{1}{\dt^2}(\fs{k} - \fs{k-1}) = u_{k-1} - m g l \rs{k-1} - b \cdot \frac{1}{\dt} \fs{k-1} \\
        & \rc{k} = \rc{k-1} \fc{k-1} - \rs{k-1} \fs{k-1} \label{eq:exp:p:dis-dyn-constraints-rcupdate} \\
        & \rs{k} = \rs{k-1} \fc{k-1} + \rc{k-1} \fs{k-1} \label{eq:exp:p:dis-dyn-constraints-rsupdate} \\
        & \rc{k}^2 + \rs{k}^2 = 1 \label{eq:exp:p:dis-dyn-constraints-so2-r} \\
        & \fc{k}^2 + \fs{k}^2 = 1 \label{eq:exp:p:dis-dyn-constraints-so2-f} \\
        & \fc{k} \ge f_{c, \min} \label{eq:exp:p:dis-dyn-constraints-fcmin} \\
        & u_{\max}^2 - u_k^2 \ge 0 
    \end{align}
\end{subequations}
\eqref{eq:exp:p:dis-dyn-constraints-fcmin} is to prevent excessively fast angular velocities. Starting from $\theta = \theta_0; \dot{\theta} = \dot{\theta}_0$, we want the pendulum to achieve $\theta = \pi, \dot{\theta} = 0$ after $N$ time steps. Figure~\ref{fig:exp:gen:sys-illustration} (a) illustrates the inverted pendulum.

\textbf{Hyepr-parameters.} Denote $x_k$ as $\left[ \rc{k} \; \rs{k} \; \fs{k} \; \fs{k} \right] \in \Real{4}$.
In all experiments, $m = 1, l = 1, b = 0.1, g = 9.8$ and $N=30, \dt = 0.1, f_{c, \min} = 0.5, u_{\max} = 5$. $P_f$ in~\eqref{eq:exp:gen:lqr-loss} is set to $1$. The number of localizing matrices is $60$. For first-order methods, \texttt{maxiter} is set to $10000$ without kNN warm starts and $100$ with kNN warm starts.

\textbf{Why do we favor global optimality?} 
We conducted two experiments demonstrating the advantages of certifiable solvers over local solvers. In the first experiment, we solved the same trajectory optimization problem using \fmincon with 10 feasible initial guesses, obtained by rolling out random control input sequences. Three out of the ten attempts converged to infeasible points. Among the remaining seven feasible trajectories, we display the three with the lowest suboptimality gaps in Figure~\ref{fig:exp:p:random}. None of these trajectories succeeded in swinging up the pendulum. In the second experiment, starting from the globally optimal trajectory, we added scaled Gaussian noise to each element: \( x(i) \leftarrow x(i) + \epsilon \calN(0, 1) \). The disturbed trajectory was then fed to \fmincon. We selected ten different \(\epsilon\) values ranging from \(10^{-1}\) to \(10\). Seven out of ten attempts converged to infeasible states, with the remaining three shown in Figure~\ref{fig:exp:p:add-noise}. Both experiments illustrate the sensitivity of local solvers to initial guesses.


\begin{figure}[t]
    \centering
    \begin{minipage}{0.7\textwidth}
        \centering
        \includegraphics[width=\columnwidth]{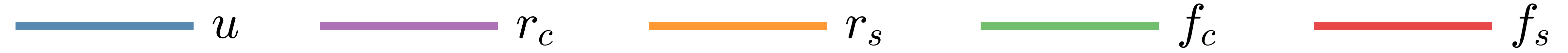}
    \end{minipage}

    \begin{minipage}{\textwidth}
        \centering
        \begin{tabular}{cccc}
            \begin{minipage}{0.24\textwidth}
                \centering
                \includegraphics[width=\columnwidth]{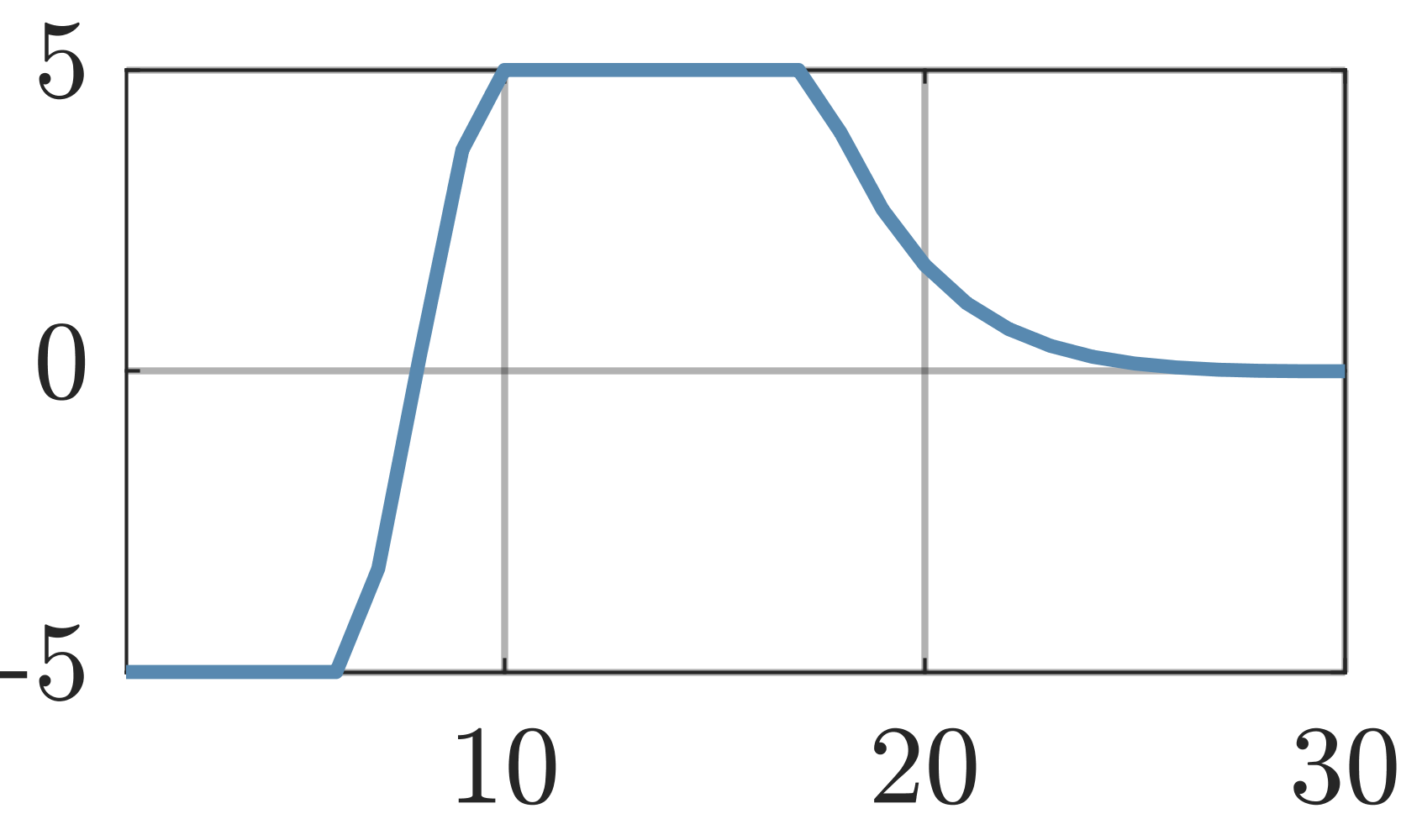}
            \end{minipage}

            \begin{minipage}{0.24\textwidth}
                \centering
                \includegraphics[width=\columnwidth]{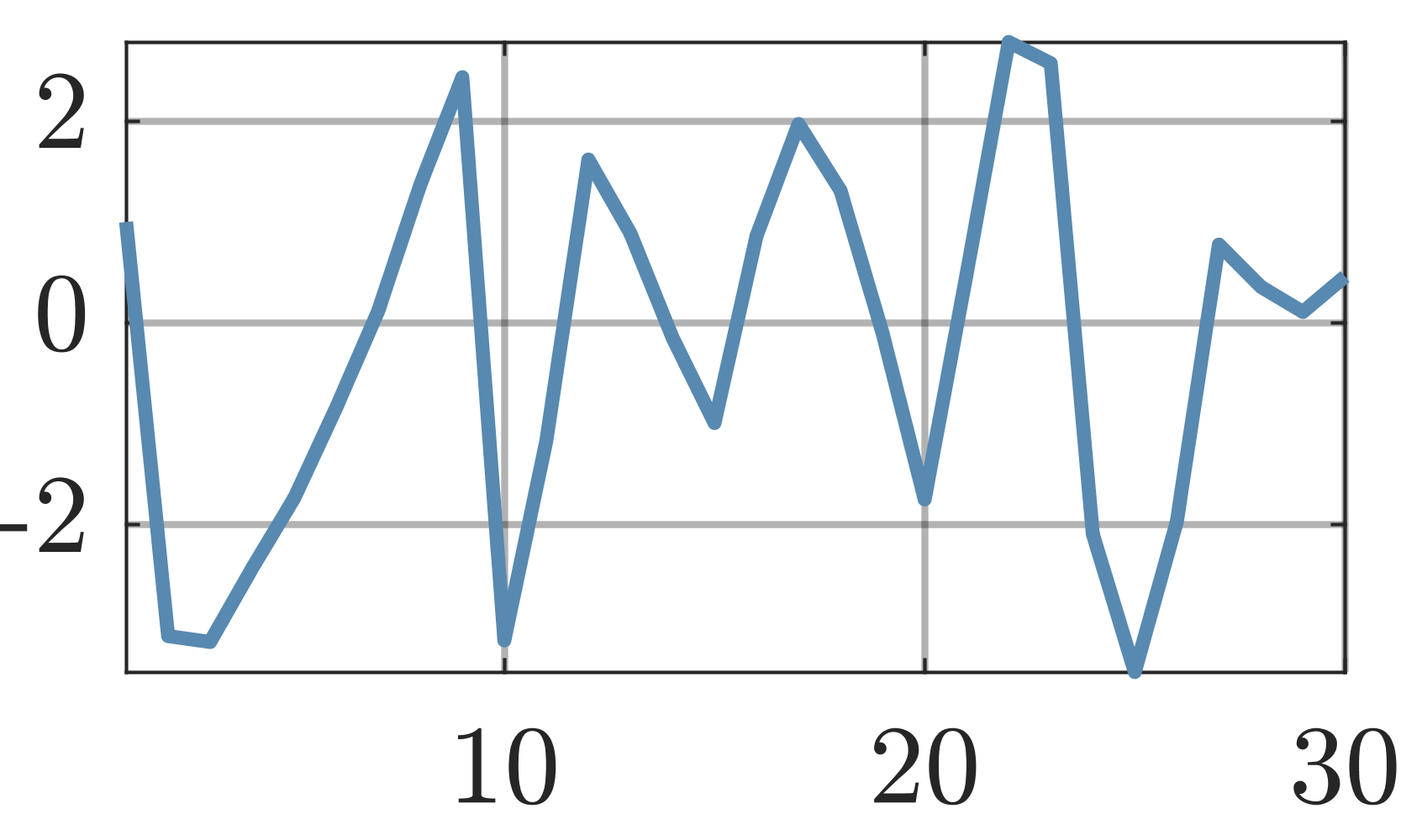}
            \end{minipage}

            \begin{minipage}{0.24\textwidth}
                \centering
                \includegraphics[width=\columnwidth]{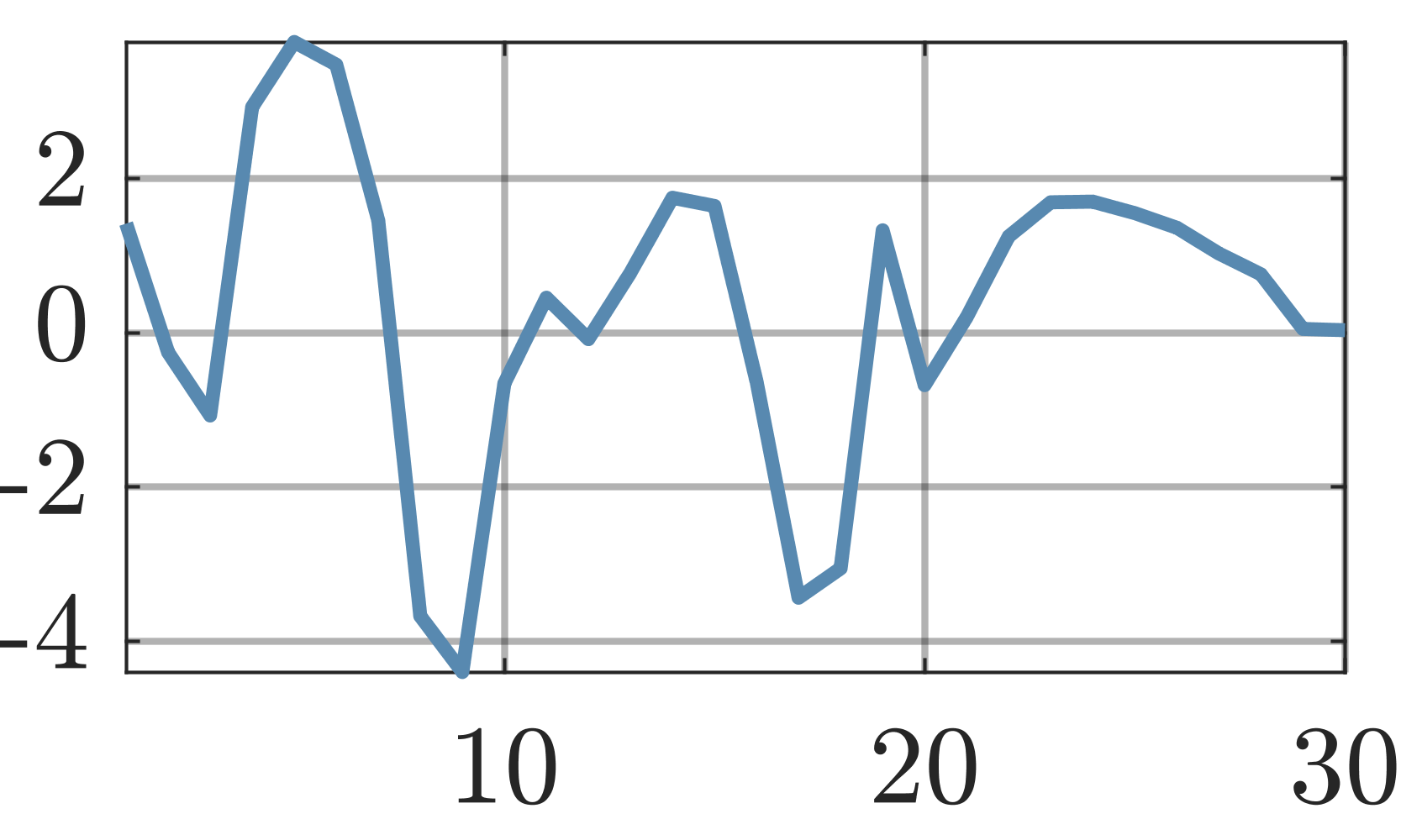}
            \end{minipage}

            \begin{minipage}{0.24\textwidth}
                \centering
                \includegraphics[width=\columnwidth]{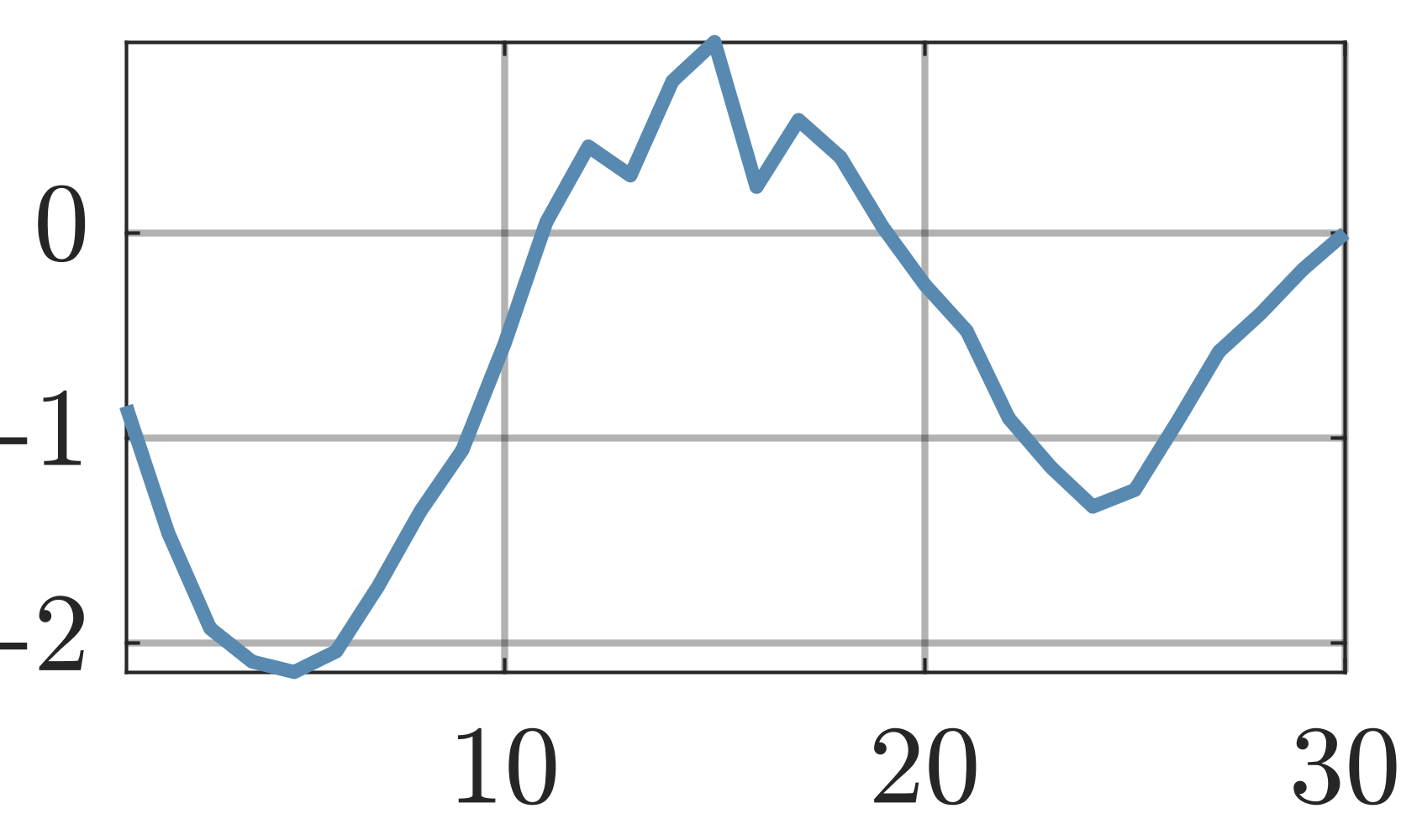}
            \end{minipage}
        \end{tabular}
    \end{minipage}

    \begin{minipage}{\textwidth}
        \centering
        \begin{tabular}{cccc}
            \begin{minipage}{0.24\textwidth}
                \centering
                \includegraphics[width=\columnwidth]{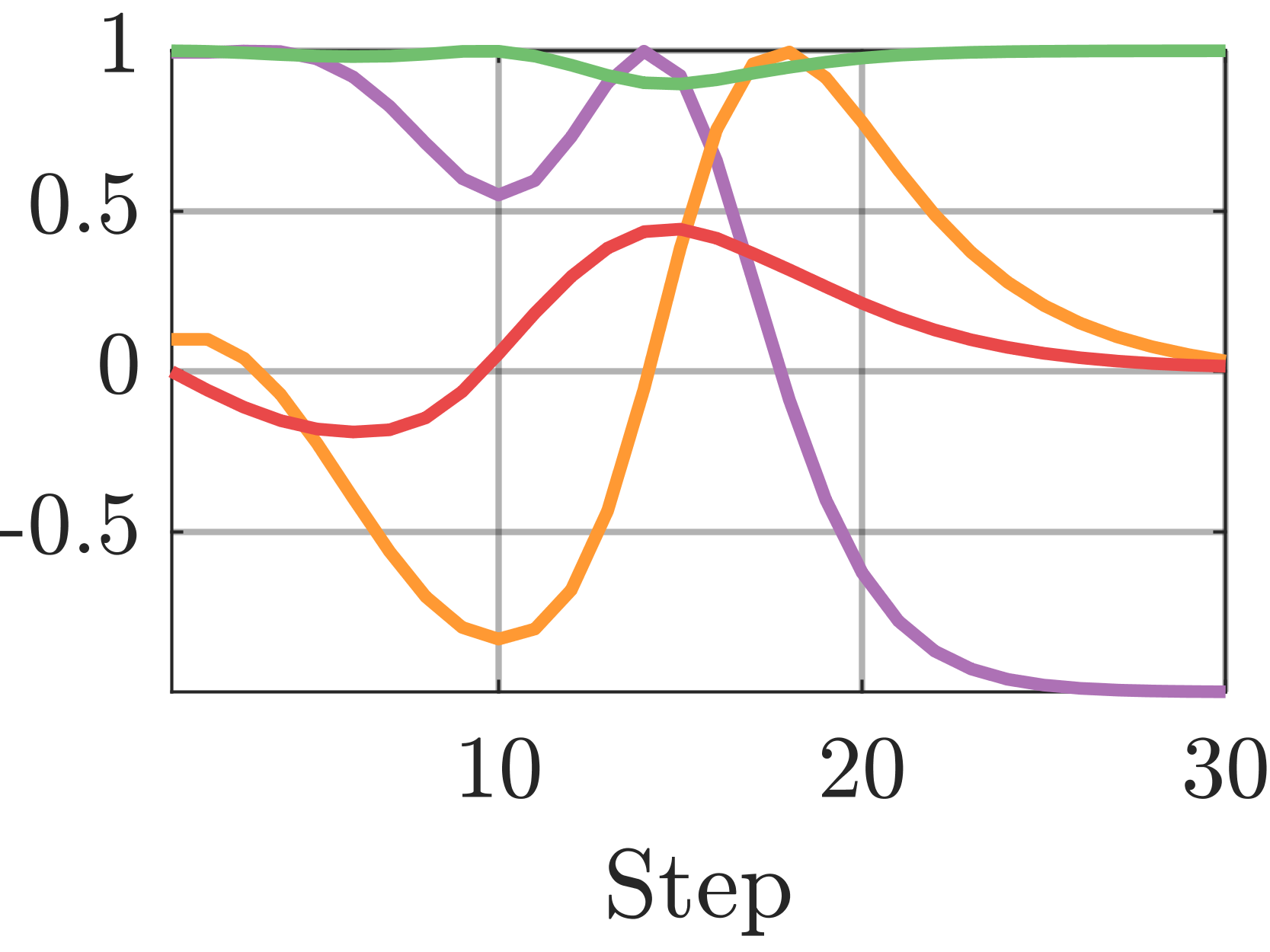}
                Global Optimal: $\xi = 6.6 \times 10^{-4}$
            \end{minipage}

            \begin{minipage}{0.24\textwidth}
                \centering
                \includegraphics[width=\columnwidth]{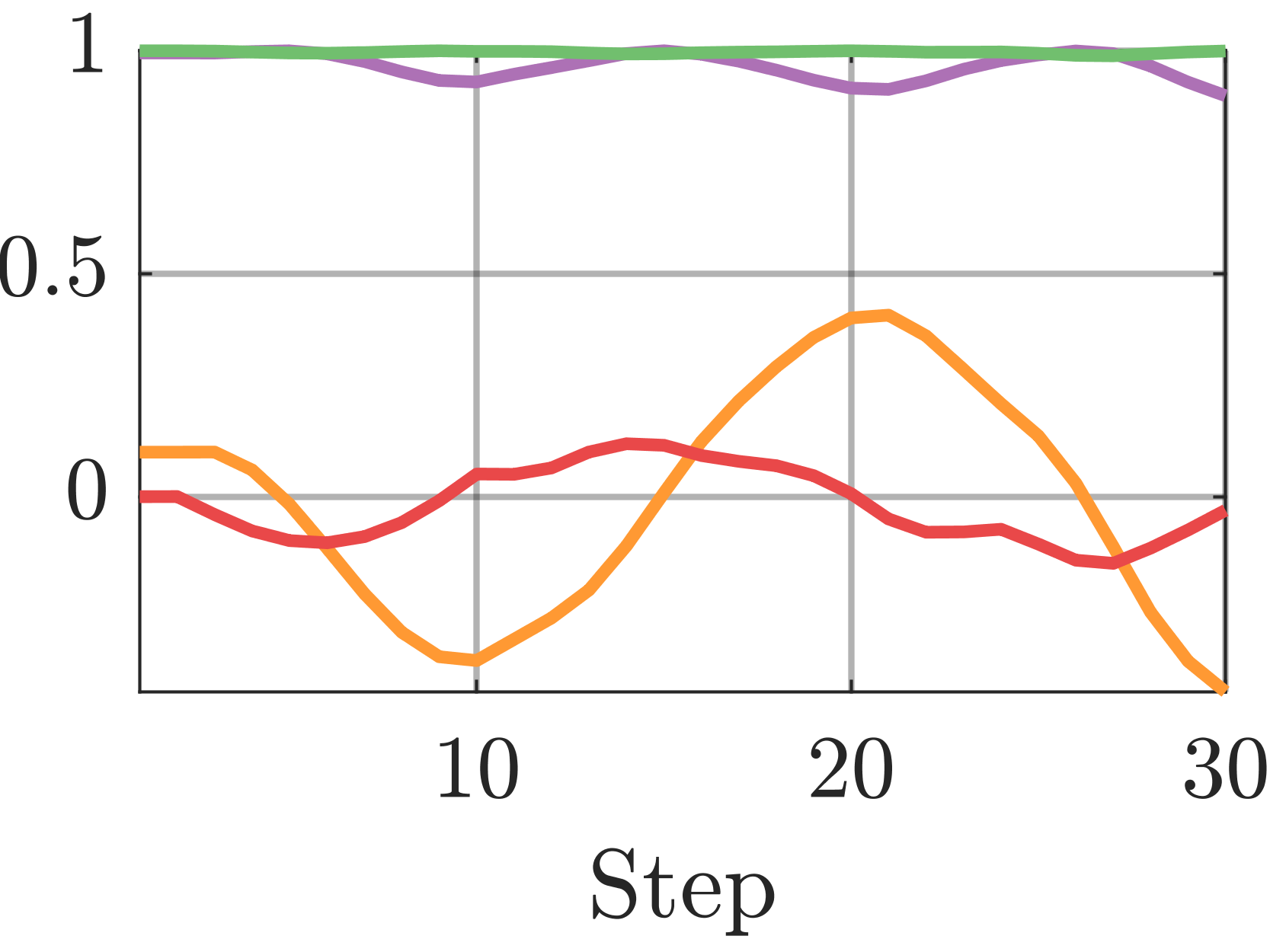}
                Initial Guess 1: $\xi = 2.0 \times 10^{-1}$
            \end{minipage}

            \begin{minipage}{0.24\textwidth}
                \centering
                \includegraphics[width=\columnwidth]{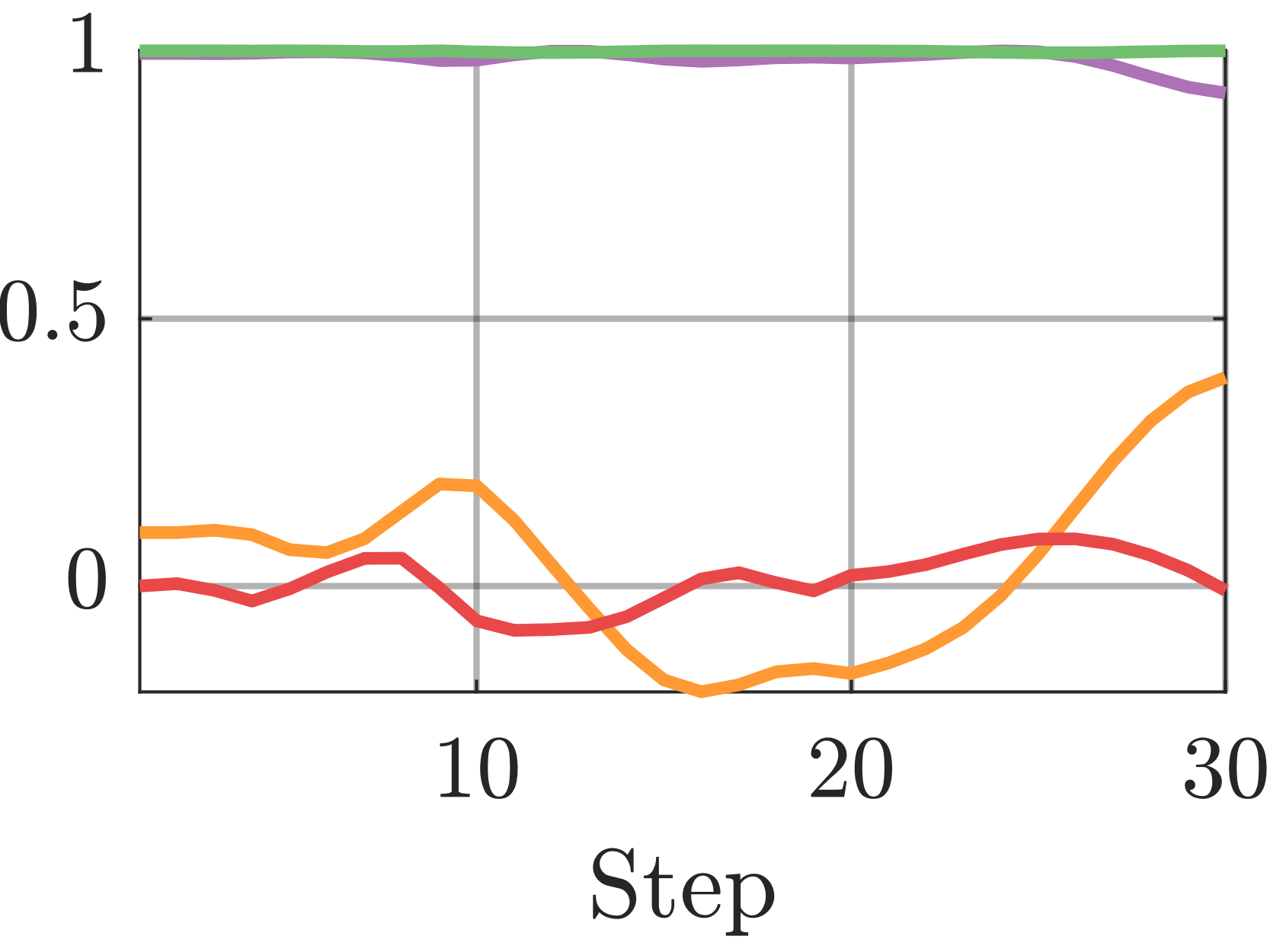}
                Initial Guess 2: $\xi = 2.1 \times 10^{-1}$
            \end{minipage}

            \begin{minipage}{0.24\textwidth}
                \centering
                \includegraphics[width=\columnwidth]{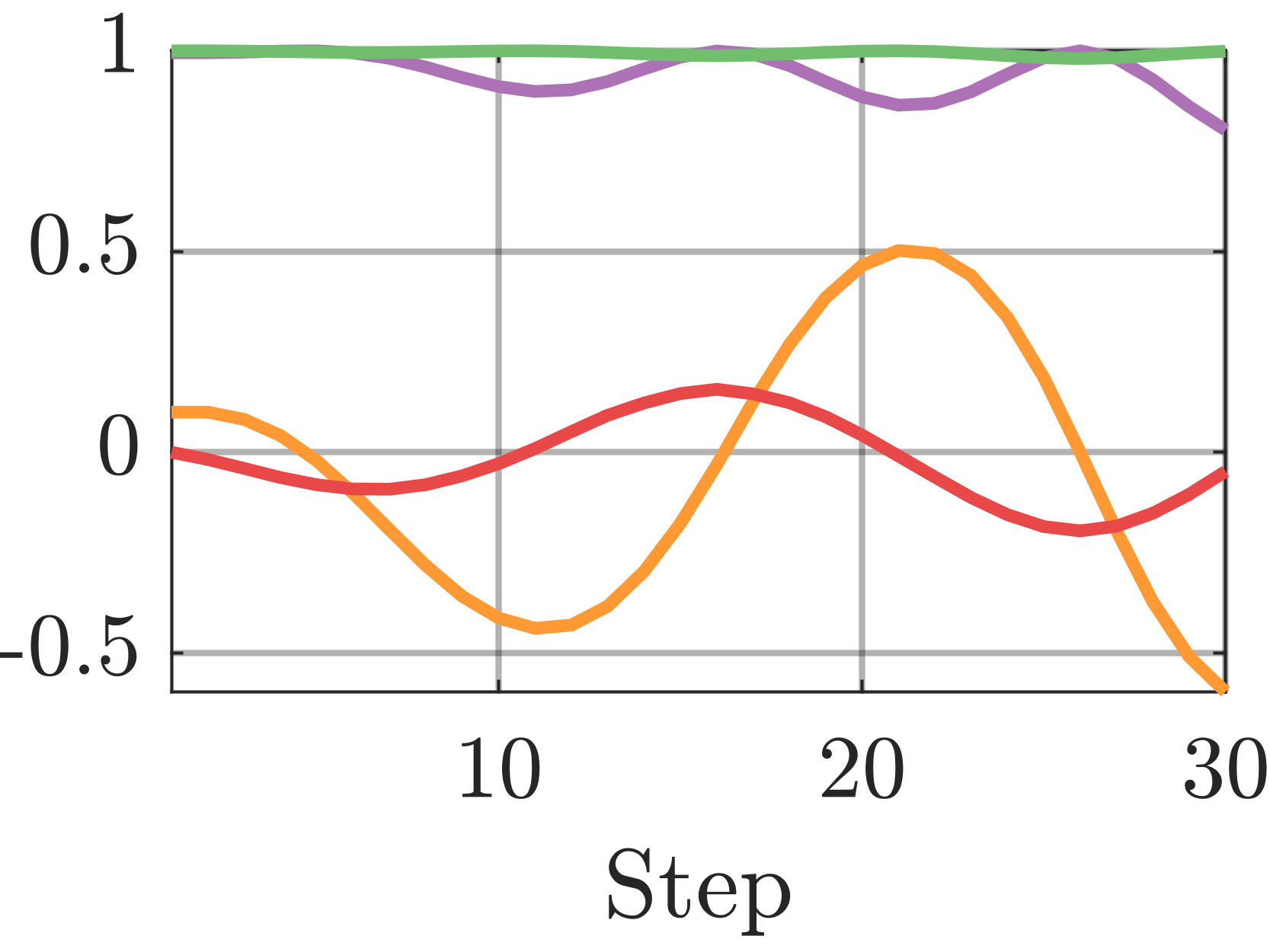}
                Initial Guess 3: $\xi = 1.8 \times 10^{-1}$
            \end{minipage}
        \end{tabular}
    \end{minipage}

    \caption{Random initial guess VS. optimal solution in the pendulum case. \label{fig:exp:p:random}}
    \vspace{3mm}
\end{figure}

\begin{figure}[h]
    \centering
    \begin{minipage}{0.7\textwidth}
        \centering
        \includegraphics[width=\columnwidth]{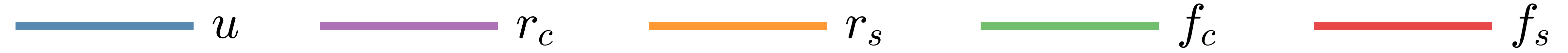}
    \end{minipage}

    \begin{minipage}{\textwidth}
        \centering
        \begin{tabular}{cccc}
            \begin{minipage}{0.24\textwidth}
                \centering
                \includegraphics[width=\columnwidth]{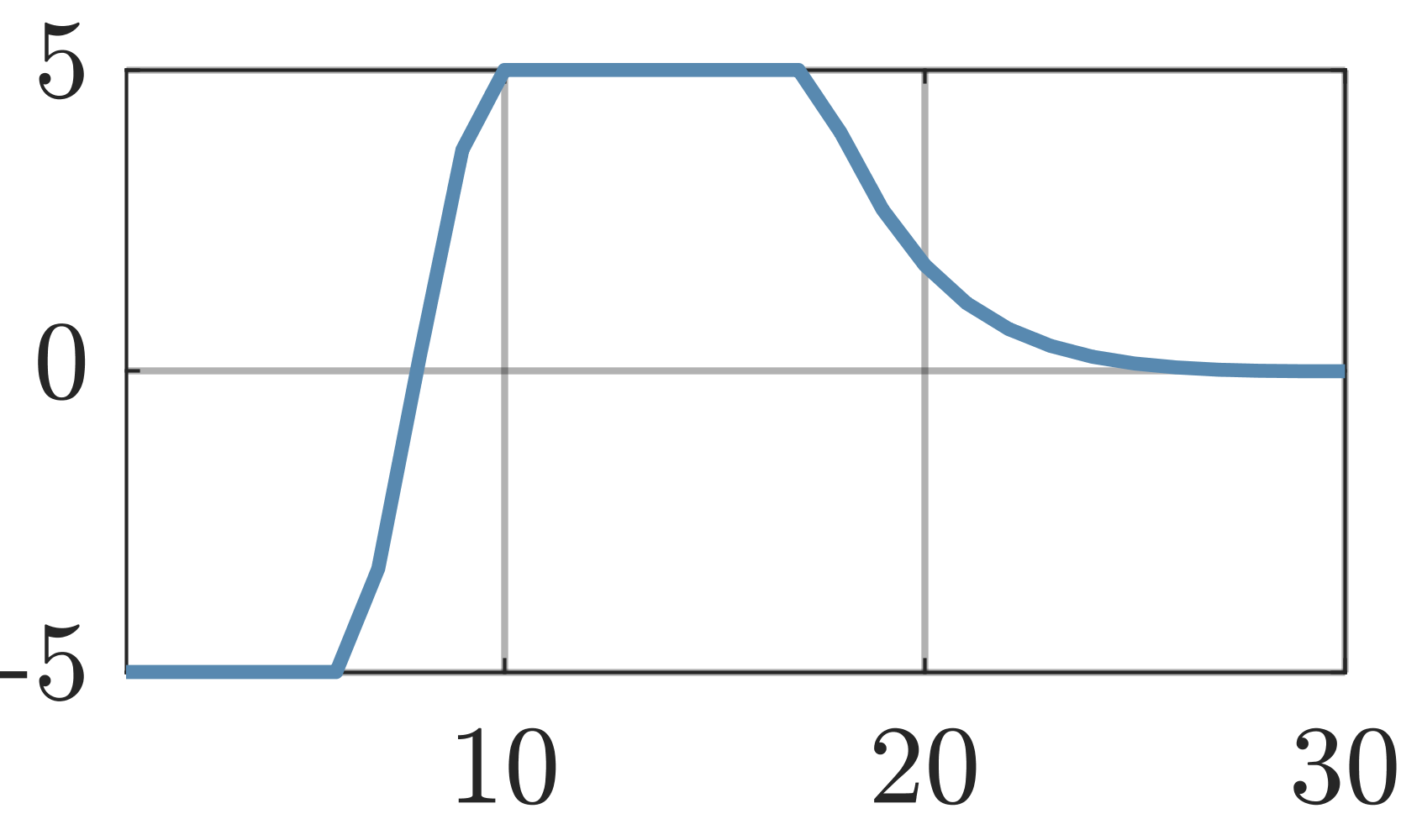}
            \end{minipage}

            \begin{minipage}{0.24\textwidth}
                \centering
                \includegraphics[width=\columnwidth]{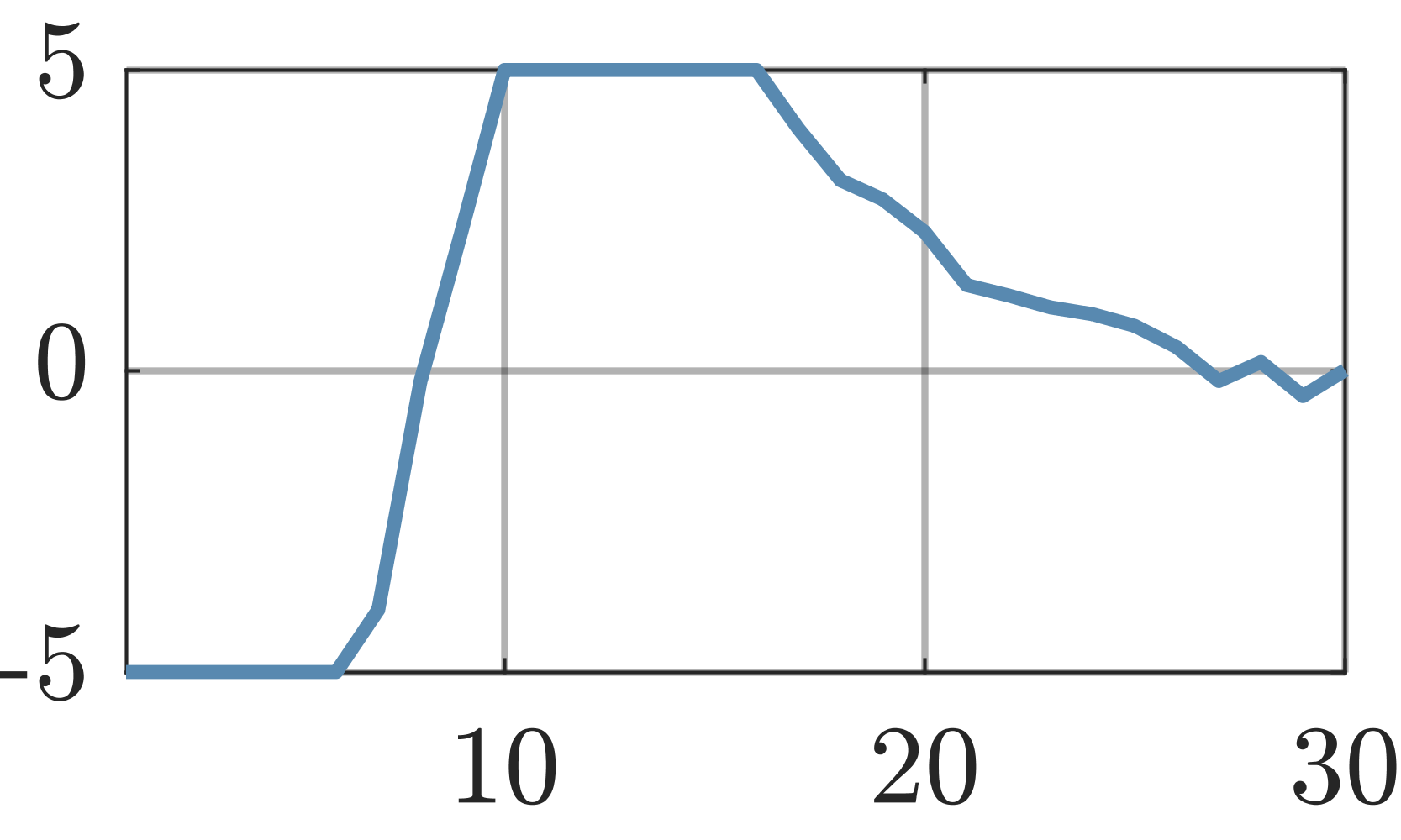}
            \end{minipage}

            \begin{minipage}{0.24\textwidth}
                \centering
                \includegraphics[width=\columnwidth]{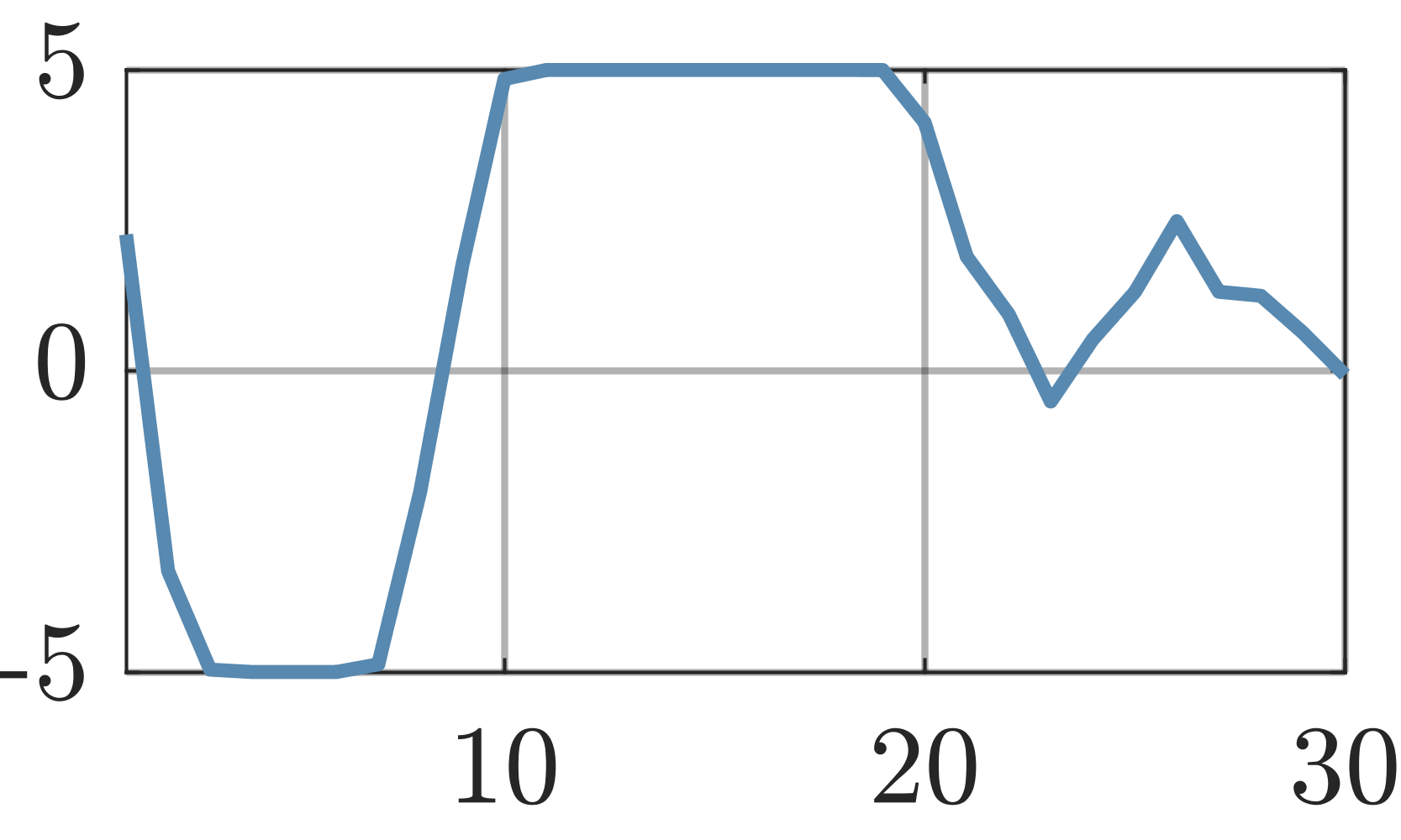}
            \end{minipage}

            \begin{minipage}{0.24\textwidth}
                \centering
                \includegraphics[width=\columnwidth]{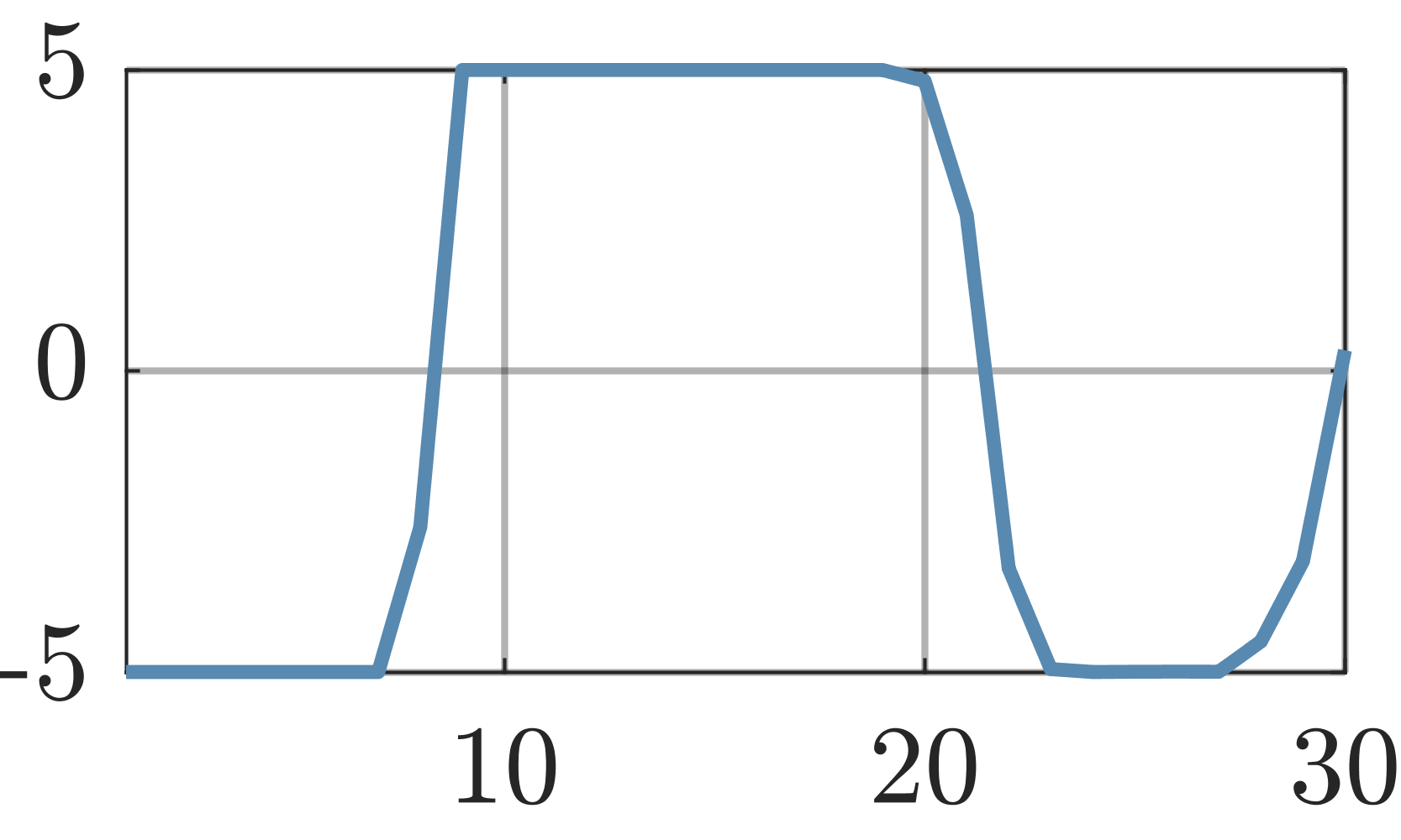}
            \end{minipage}
        \end{tabular}
    \end{minipage}

    \begin{minipage}{\textwidth}
        \centering
        \begin{tabular}{cccc}
            \begin{minipage}{0.24\textwidth}
                \centering
                \includegraphics[width=\columnwidth]{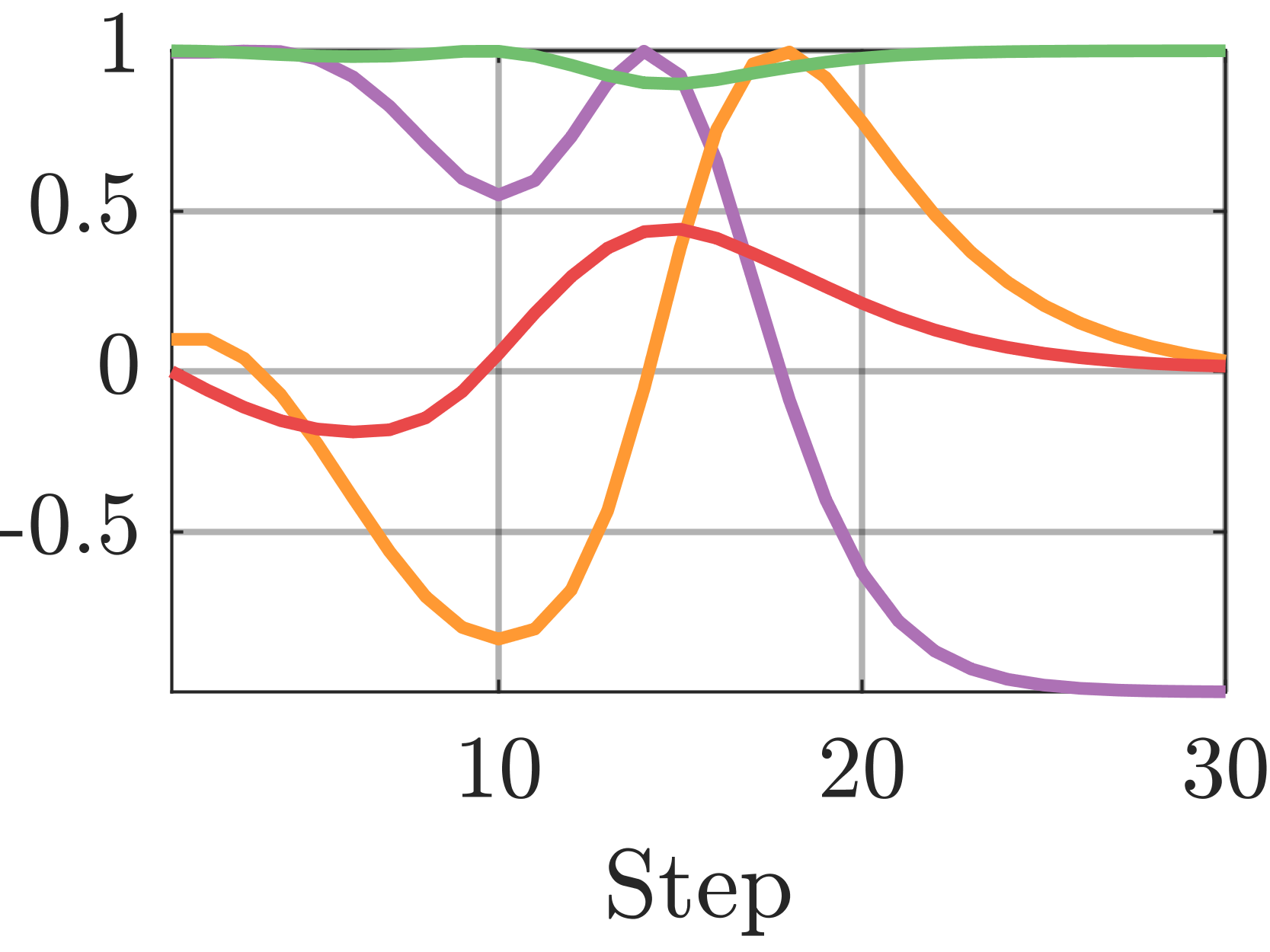}
                Global Optimal: $\xi = 6.6 \times 10^{-4}$
            \end{minipage}

            \begin{minipage}{0.24\textwidth}
                \centering
                \includegraphics[width=\columnwidth]{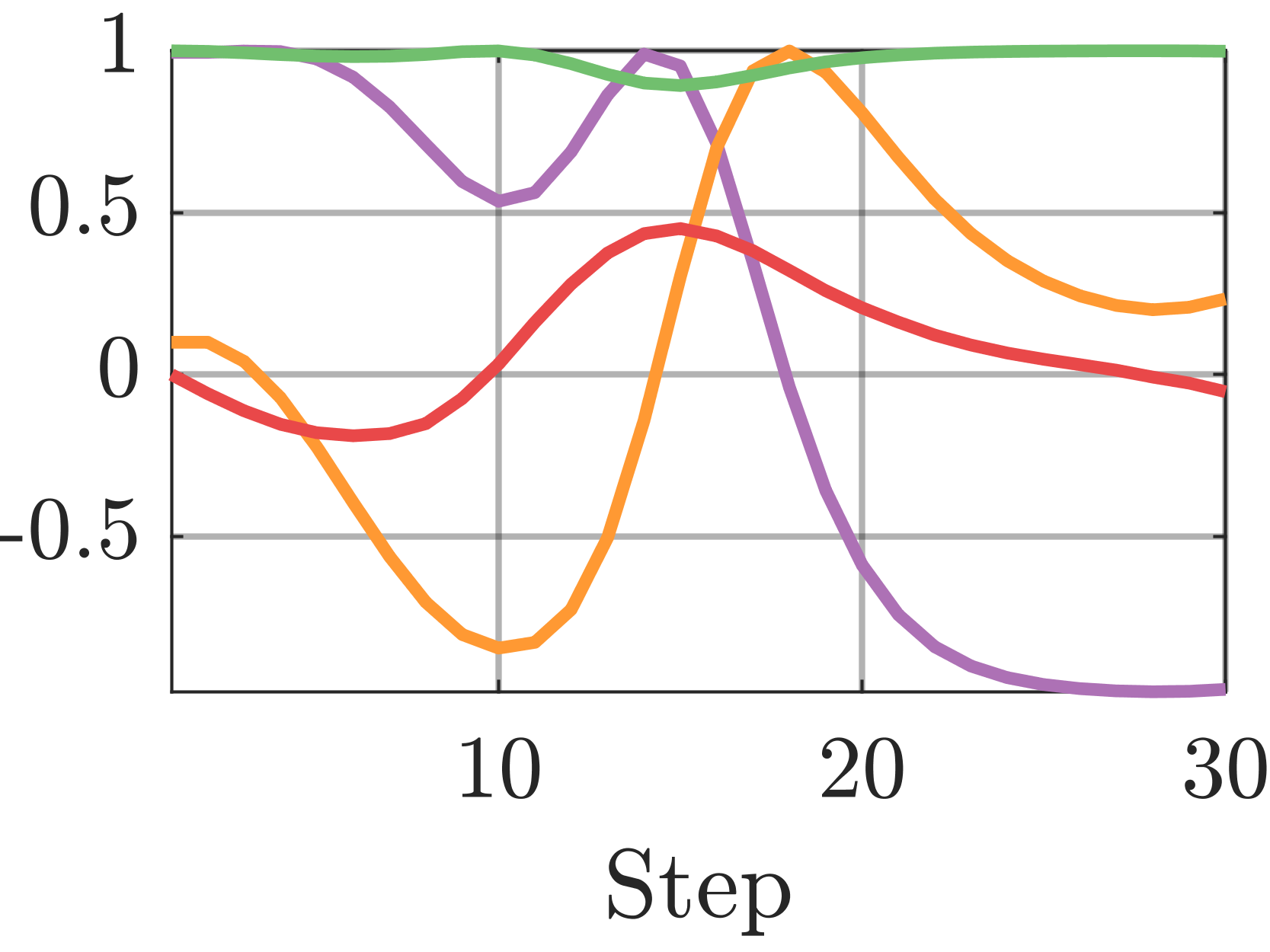}
                $\epsilon = 0.1$: $\xi = 2.8 \times 10^{-3}$
            \end{minipage}

            \begin{minipage}{0.24\textwidth}
                \centering
                \includegraphics[width=\columnwidth]{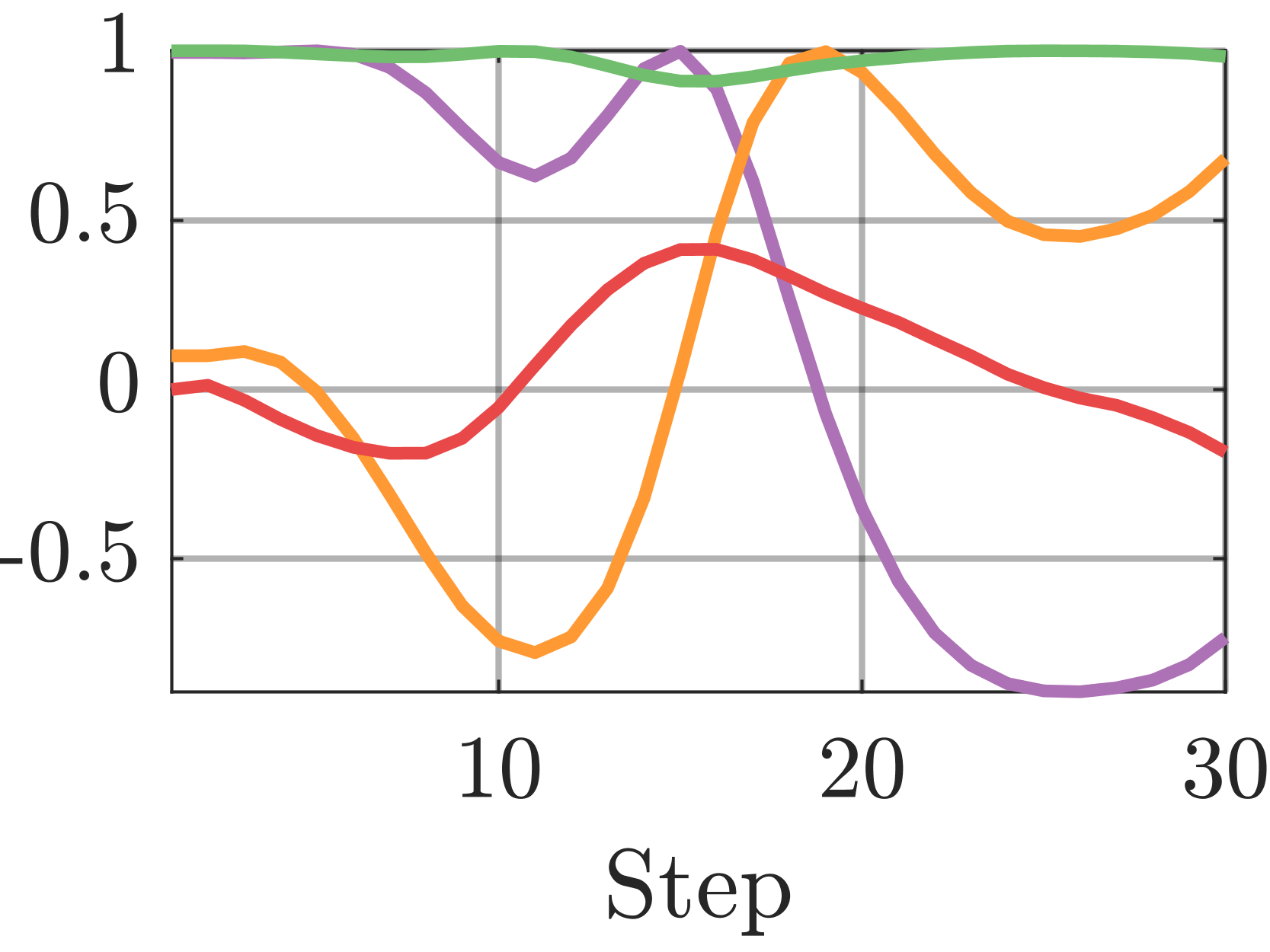}
                $\epsilon = 1.0$: $\xi = 5.1 \times 10^{-2}$
            \end{minipage}

            \begin{minipage}{0.24\textwidth}
                \centering
                \includegraphics[width=\columnwidth]{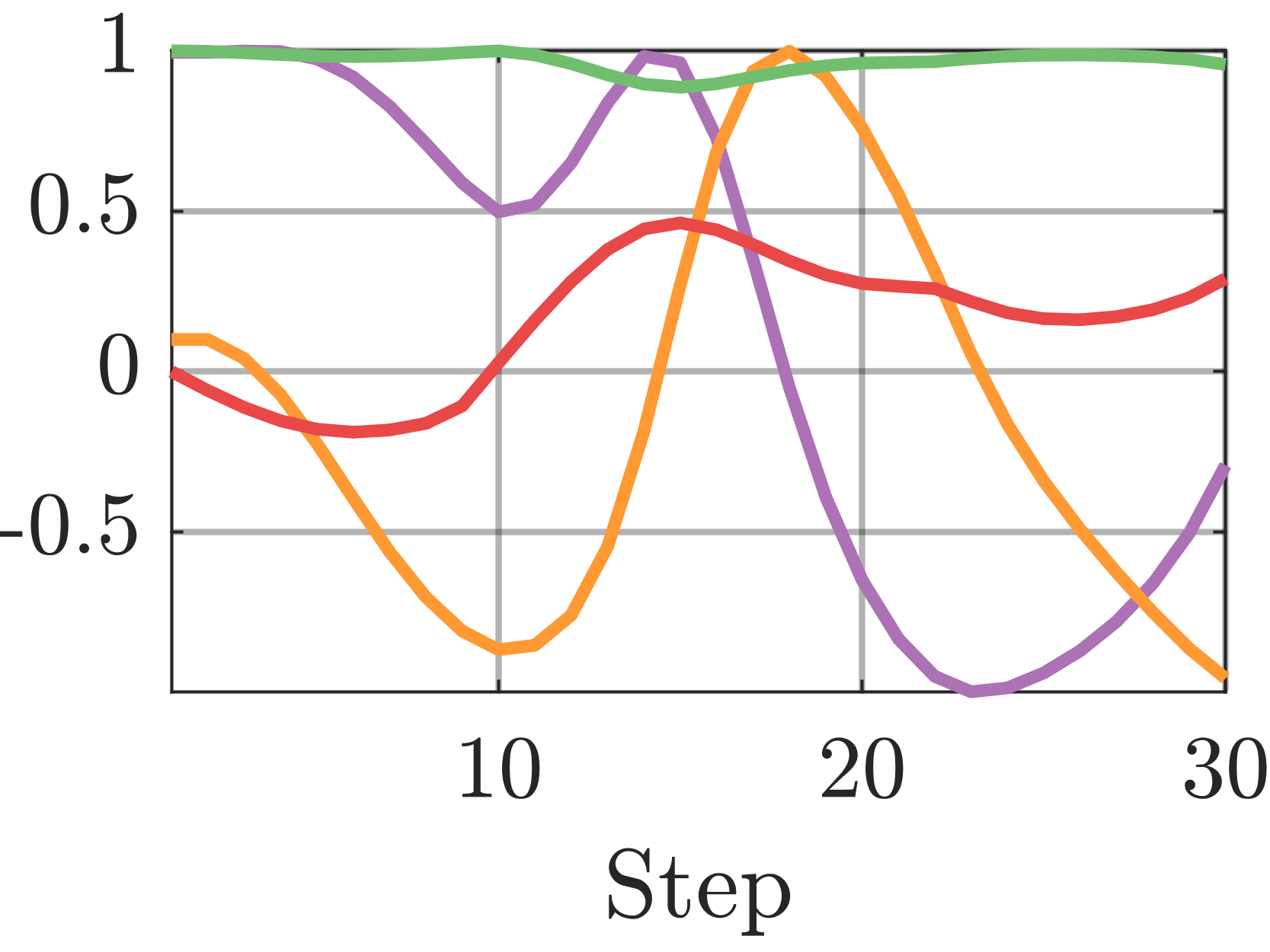}
                $\epsilon = 2.0$: $\xi = 8.6 \times 10^{-2}$
            \end{minipage}
        \end{tabular}
    \end{minipage}

    \caption{Adding noise to the global optimal solution in the pendulum case. \label{fig:exp:p:add-noise}}
    \vspace{3mm}
\end{figure}

\textbf{Spiral-line pattern for hard-to-solve initial states.} 
As shown in Figure~\ref{fig:exp:p}, around $10\%$ initial states are hard to solve for all four SDP solvers. Figure~\ref{fig:exp:p:heatmap} demonstrates the heatmaps of the suboptimality gaps across different solvers. 
We found that the states with relatively high suboptimality gaps roughly form a spiral line, as reported in~\cite{han2023arxiv-nonsmooth}, corresponding to the nonsmooth parts of the pendulum's optimal value function. There are multiple optimal trajectories for these initial states, which likely explains why the Moment-SOS Hierarchy fails to generate a rank-1 solution.


\begin{figure}[t]
    \centering
    \begin{minipage}{\textwidth}
        \centering
        \begin{tabular}{cc}
            \begin{minipage}{0.49\textwidth}
                \centering
                \includegraphics[width=\columnwidth]{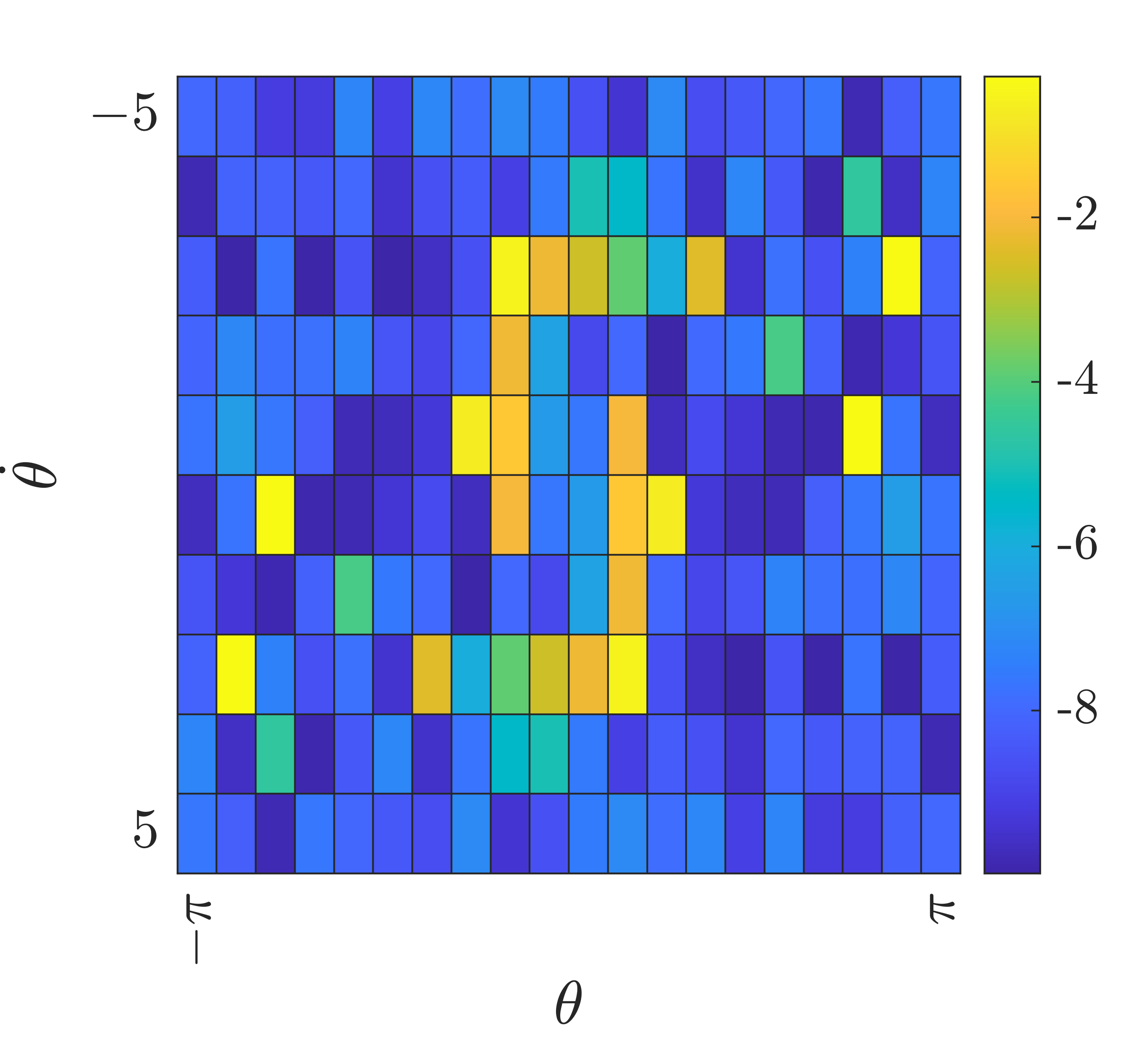}
                \mosek
            \end{minipage}

            \begin{minipage}{0.49\textwidth}
                \centering
                \includegraphics[width=\columnwidth]{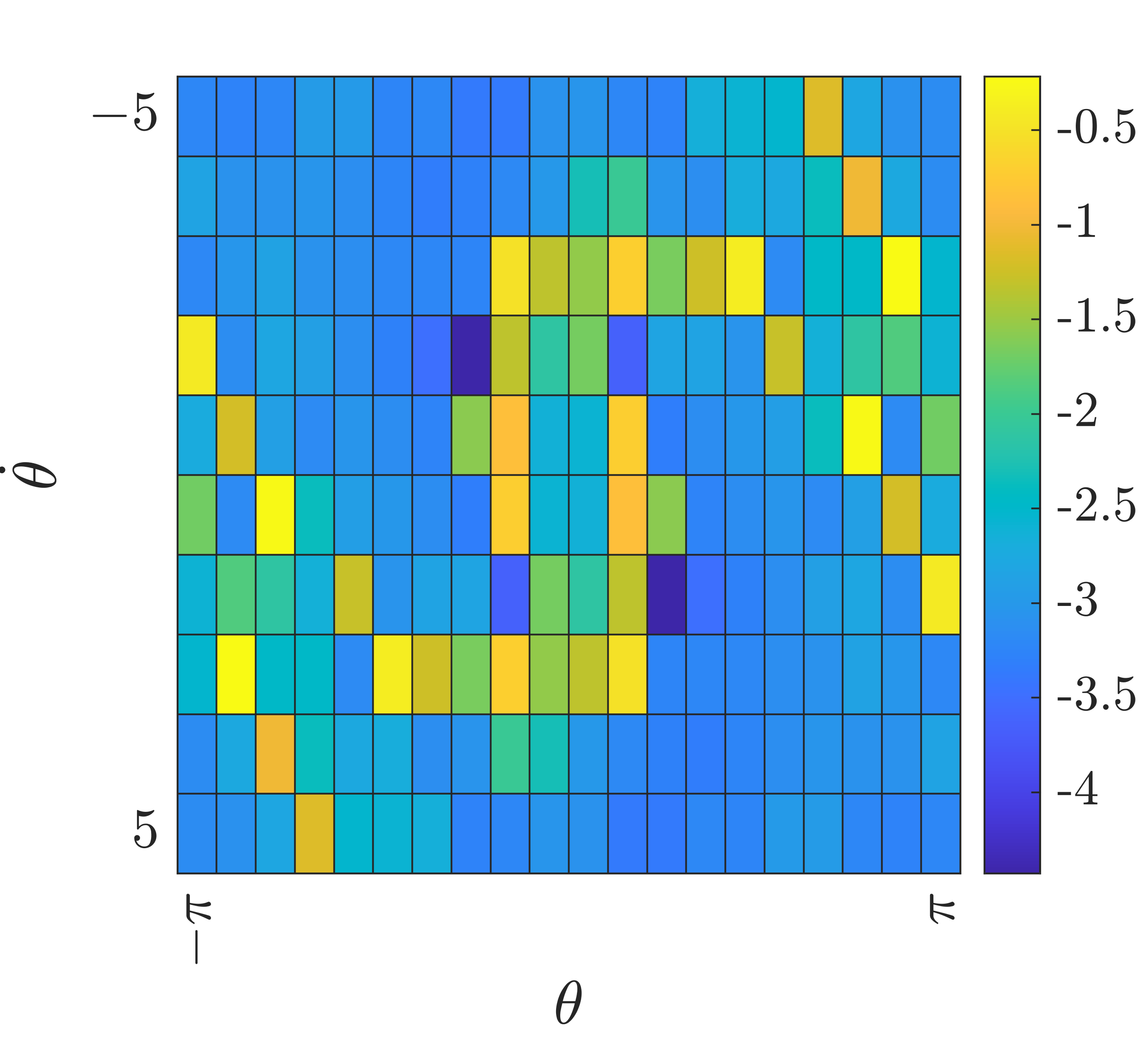}
                \cuadmm
            \end{minipage}
        \end{tabular}
    \end{minipage}

    \begin{minipage}{\textwidth}
        \centering
        \begin{tabular}{cc}
            \begin{minipage}{0.49\textwidth}
                \centering
                \includegraphics[width=\columnwidth]{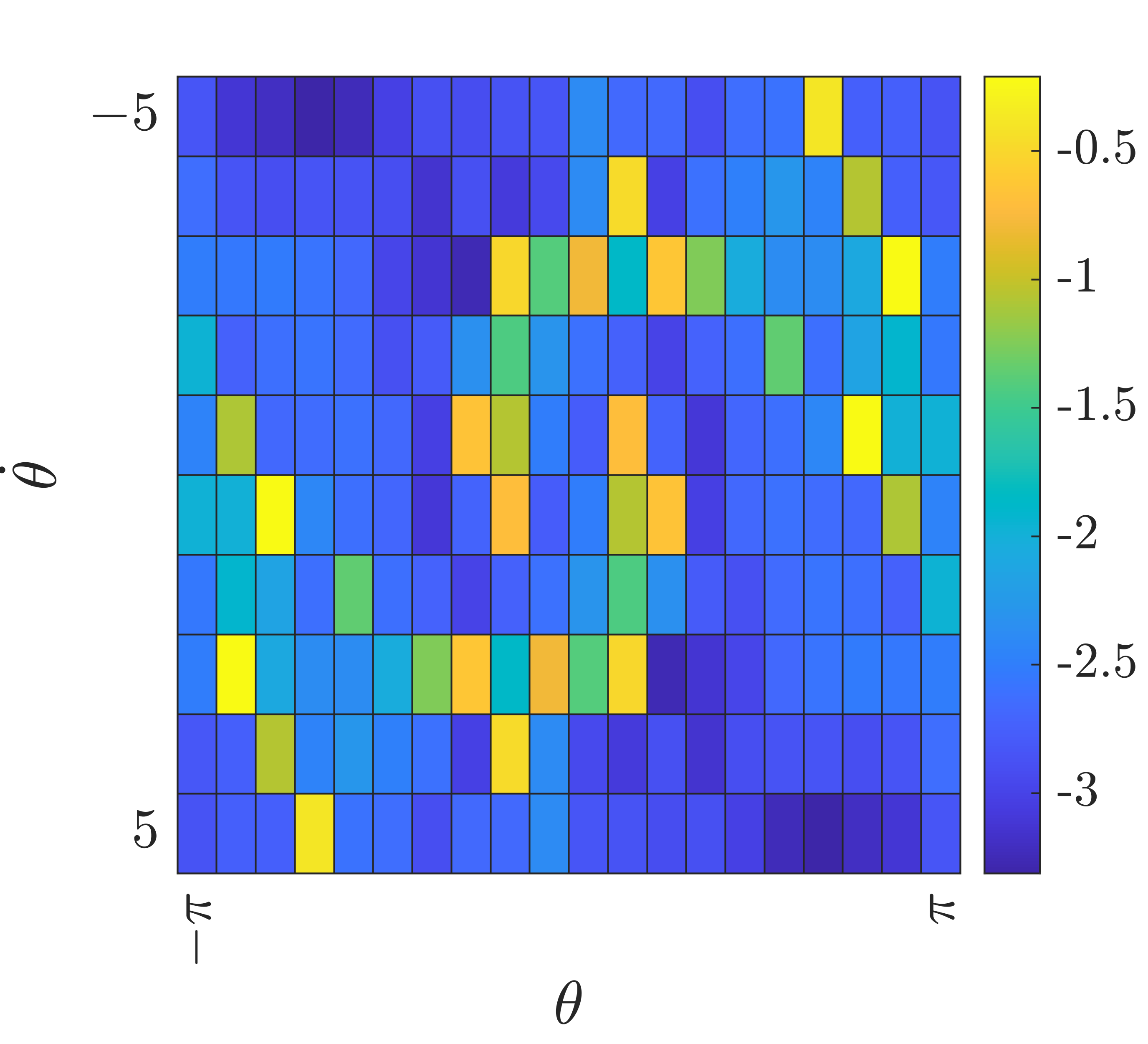}
                \sdpnal
            \end{minipage}

            \begin{minipage}{0.49\textwidth}
                \centering
                \includegraphics[width=\columnwidth]{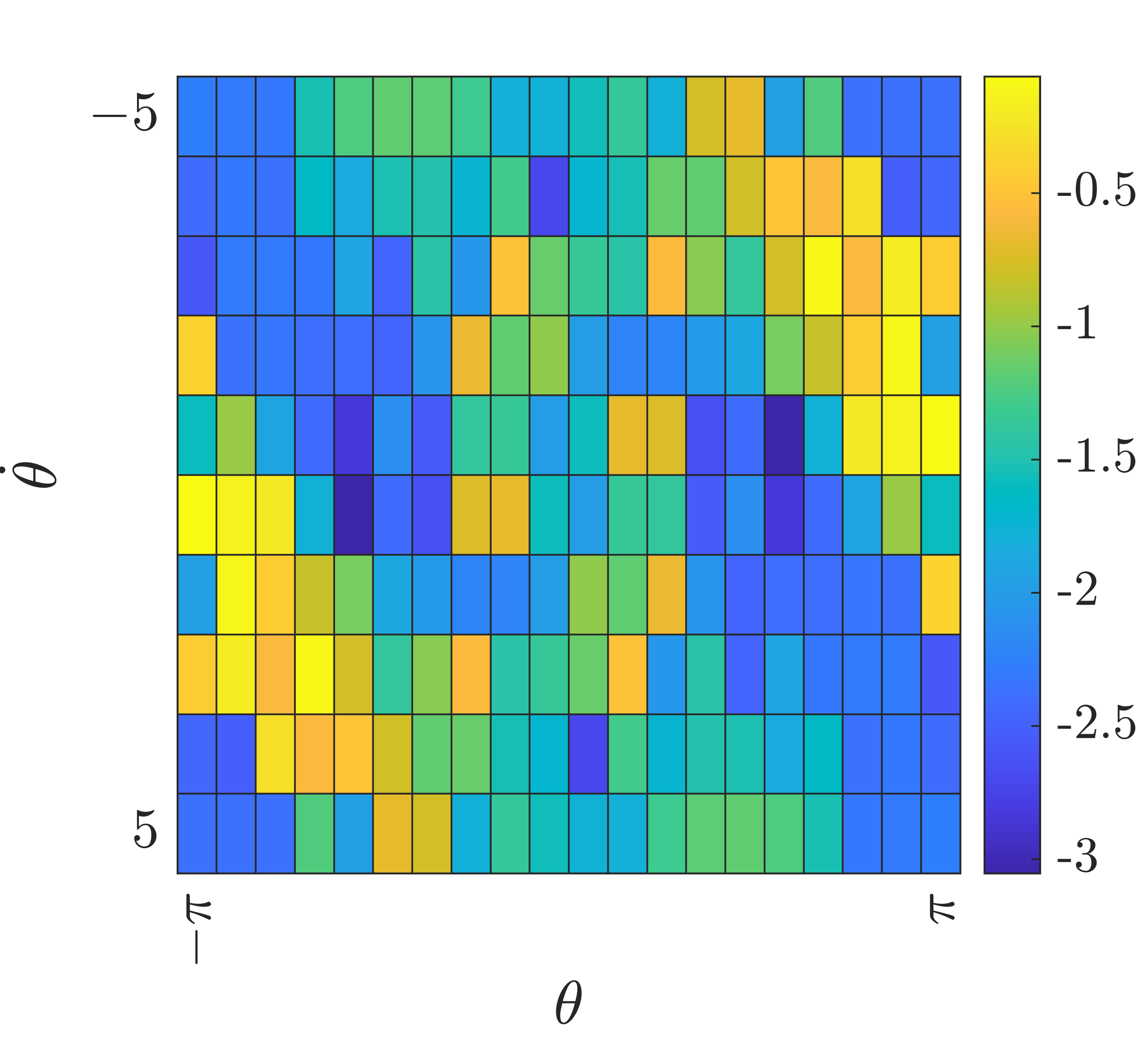}
                \cdcs
            \end{minipage}
        \end{tabular}
    \end{minipage}

    \caption{Heatmaps for suboptimality gaps from different SDP solvers.
    \label{fig:exp:p:heatmap}}
    \vspace{3mm}
\end{figure}

\subsection{Cart-Pole}
\label{app:subsec:exp:ca}
\textbf{Dynamics and constraints.} From~\cite{kelly2017siam-intro-directcollocation}, cart-pole's continuous time dynamics are:
\begin{subequations}
    \begin{align}
        \label{eq:exp:ca:con-dyn}
        \ddot{a} & = \frac{
            l m_2 \sin\theta \cdot \dot{\theta}^2 + u + m_2 g \cos\theta \sin\theta
        }{
            m_1 + m_2 \sin^2\theta
        } \\
        \ddot{\theta} & = -\frac{
            (l m_2 \sin\theta \cdot \dot{\theta}^2 + u) \cos\theta + (m_1 + m_2) g \sin\theta 
        }{l (m_1 + m_2 \sin^2\theta)}
    \end{align}
\end{subequations}
$a$ is the cart's position, and $\theta$ is the pole's angle with the plumb line. $m_1$ and $m_2$ are the mass of cart and pole, respectively. $l$ is the pole's length. With Lie group variational integrator~\cite{lee2008thesis-computationalgeometricmechanics}, the discretized dynamics and constraints are:
\begin{subequations}
    \begin{align}
        & \frac{m_1 + m_2}{\dt} \cdot (a_{k+1} - 2a_k + a_{k-1}) + \frac{m_2 l}{\dt} \cdot (\rs{k+1} - 2\rs{k} + \rs{k-1}) - \dt \cdot u_k = 0 \\
        & \frac{l}{\dt} \cdot (\fs{k} - \fs{k-1}) + \frac{1}{\dt} \cdot (a_{k+1} - 2a_k + a_{k-1}) \rc{k} + g\dt \cdot \rs{k} = 0 \\
        & ~\eqref{eq:exp:p:dis-dyn-constraints-rcupdate},~\eqref{eq:exp:p:dis-dyn-constraints-rsupdate},~\eqref{eq:exp:p:dis-dyn-constraints-so2-r},~\eqref{eq:exp:p:dis-dyn-constraints-so2-f},~\eqref{eq:exp:p:dis-dyn-constraints-fcmin} \\
        & u_{\max}^2 - u_k^2 \ge 0, \ a_{\max}^2 - a_k^2 \ge 0 
    \end{align} 
\end{subequations}
From $a = a_0, \dot{a} = \dot{a}_0, \theta= \theta_0, \dot{\theta} = \dot{\theta}_0$, we want the system to achieve $a = a_f, \dot{a} = 0, \theta = \pi, \dot{\theta} = 0$ in $N$ time steps. Figure~\ref{fig:exp:gen:sys-illustration} (b) illustrates the cart-pole system.

\textbf{Hyepr-parameters.} Denote $x_k$ as $[a_k \; \rc{k} \; \rs{k} \; \fc{k} \; \fs{k}] \in \Real{5}$. In all experiments, $m_1 = 1.0, m_2 = 0.3, l = 0.5$ and $N = 30, \dt = 0.1, f_{c, \min} = 0.5, a_{\max} = 1, u_{\max} = 10$. $P_f$ in~\eqref{eq:exp:gen:lqr-loss} is set to $1$. The number of localizing matrices is $273$. For first-order methods, \texttt{maxiter} is set to $10000$.

\textbf{Experiment results.} Figure~\ref{fig:exp:ca:demos-plots} shows three globally optimal trajectories for cart-pole. Figure~\ref{fig:exp:ca:mpc} shows the results of simulated MPC. Starting from $a_0 = 0, \dot{a}_0 = 0, \theta_0 = 0.1, \dot{\theta}_0 = 0$, we want to achieve $a = 0, \dot{a} = 0, \theta = \pi, \dot{\theta} = 0$. We take the last step's optimal SDP solution as the next step's initial guess. We set the control frequency to $10$Hz. Out of $47$ steps, $4$ extracted solutions from SDP fail to round a feasible solution. However, we can still heuristically apply the first control input from the extracted solution (possibly with some clipping) to the environment. Our certifiable controller is capable of stabilizing the cart-pole.

\begin{figure}[t]
    \begin{minipage}{\textwidth}
        \centering
        \begin{tabular}{ccc}
            \begin{minipage}{0.33\textwidth}
                \centering
                \includegraphics[width=\columnwidth]{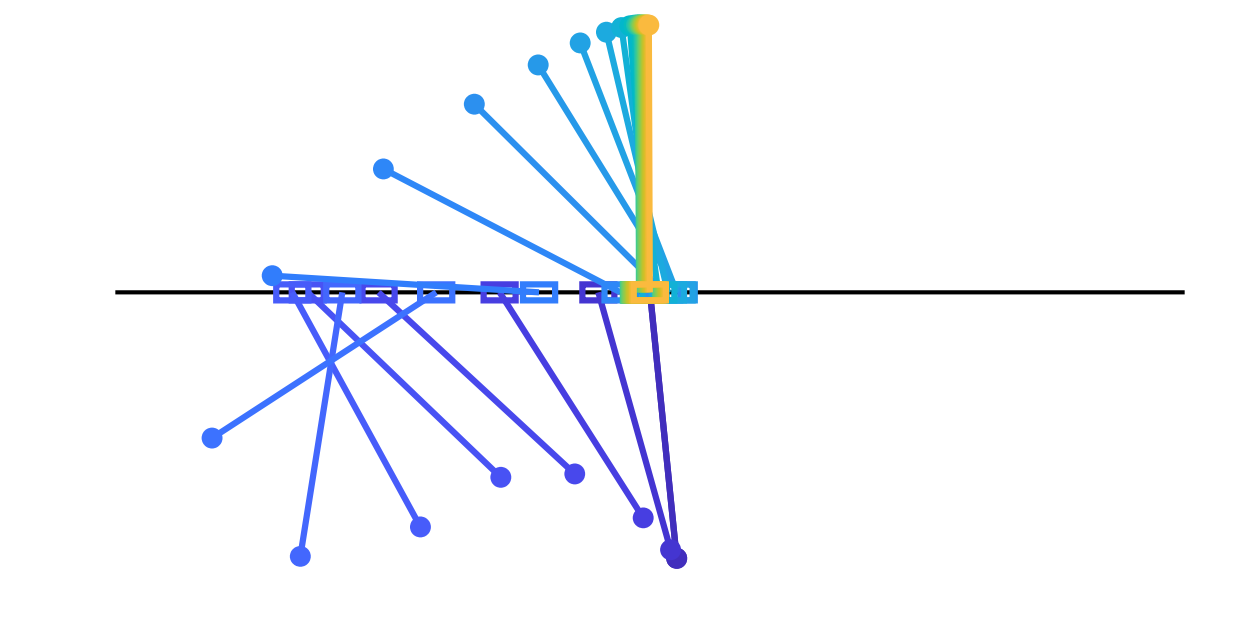}
            \end{minipage}

            \begin{minipage}{0.33\textwidth}
                \centering
                \includegraphics[width=\columnwidth]{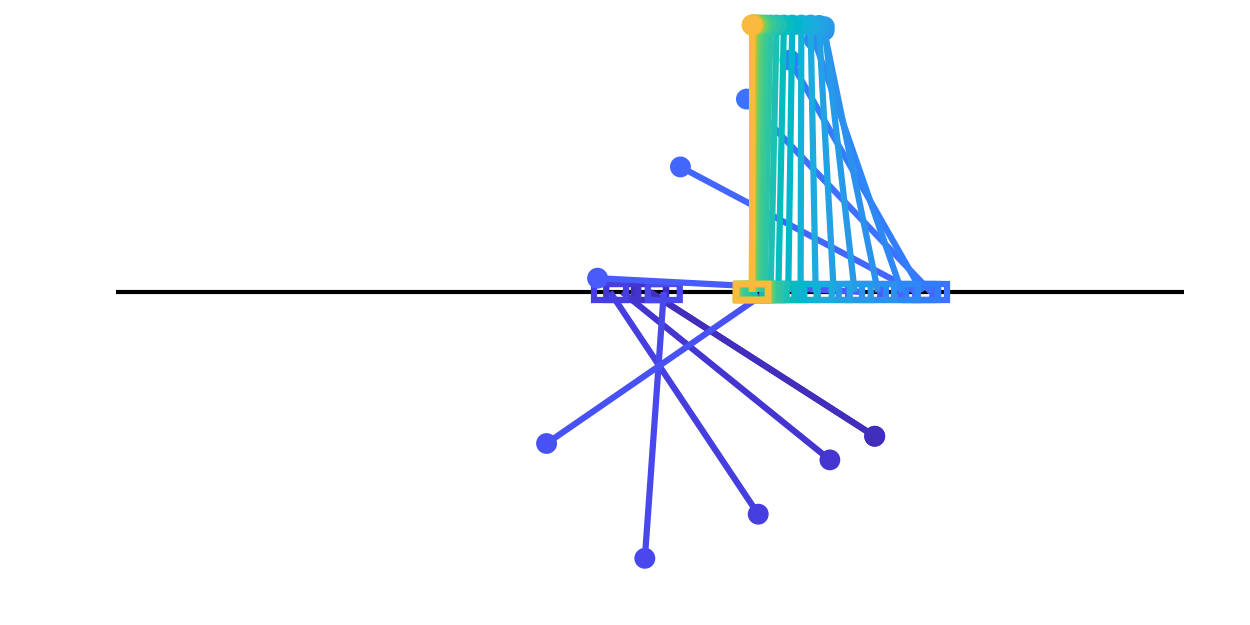}
            \end{minipage}

            \begin{minipage}{0.33\textwidth}
                \centering
                \includegraphics[width=\columnwidth]{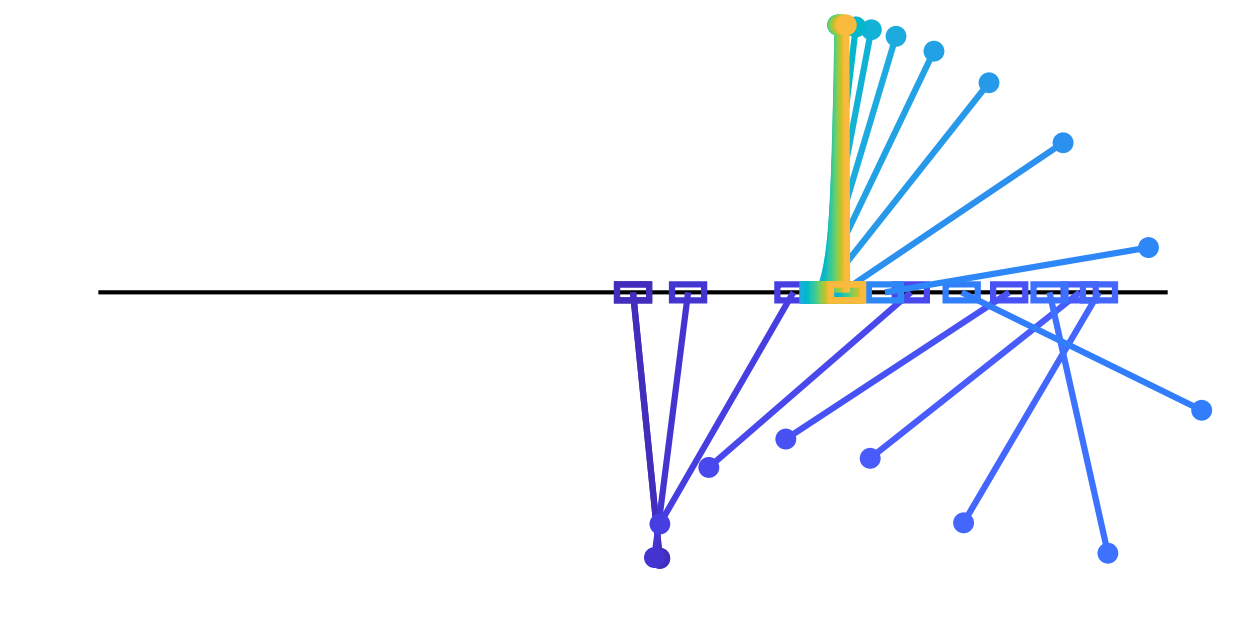}
            \end{minipage}
        \end{tabular}
    \end{minipage}

    \begin{minipage}{\textwidth}
        \centering
        \begin{tabular}{ccc}
            \begin{minipage}{0.33\textwidth}
                \centering
                \includegraphics[width=\columnwidth]{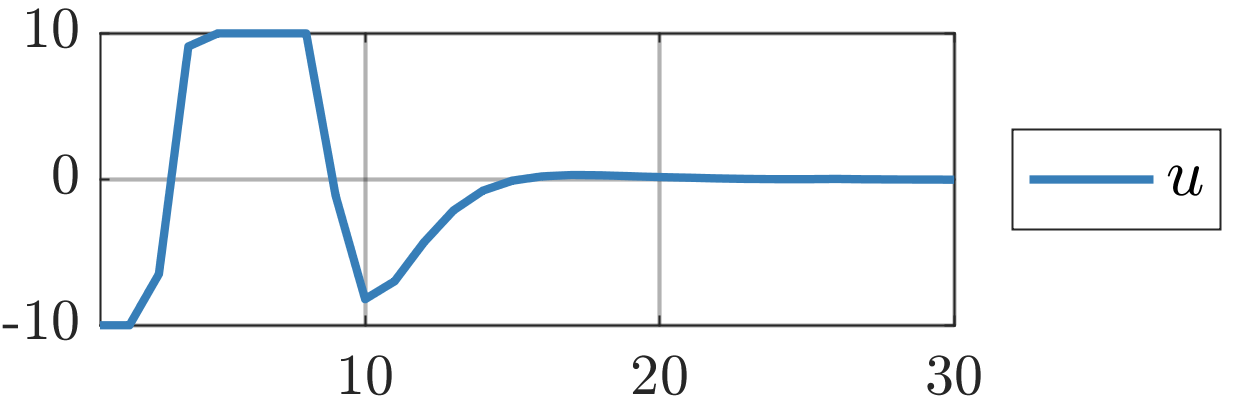}
            \end{minipage}

            \begin{minipage}{0.33\textwidth}
                \centering
                \includegraphics[width=\columnwidth]{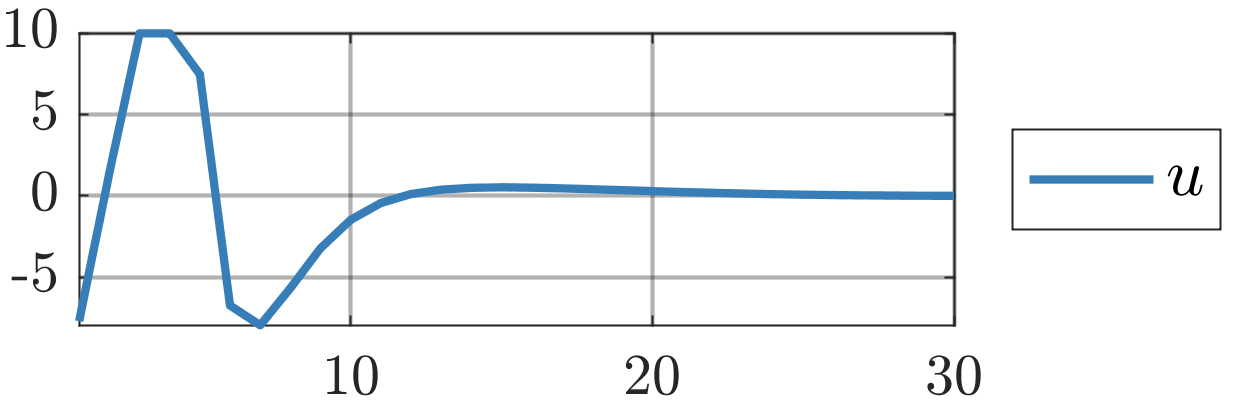}
            \end{minipage}

            \begin{minipage}{0.33\textwidth}
                \centering
                \includegraphics[width=\columnwidth]{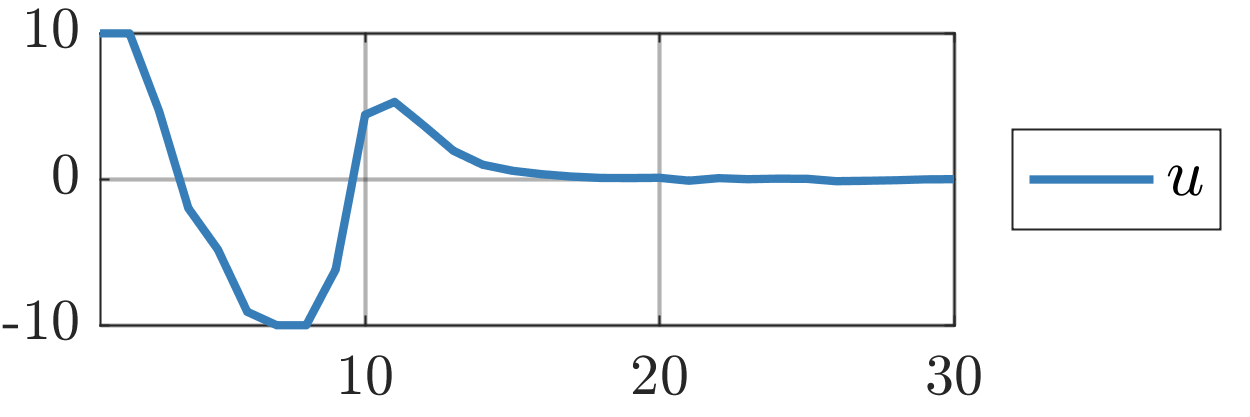}
            \end{minipage}
        \end{tabular}
    \end{minipage}

    \begin{minipage}{\textwidth}
        \centering
        \begin{tabular}{ccc}
            \begin{minipage}{0.33\textwidth}
                \centering
                \includegraphics[width=\columnwidth]{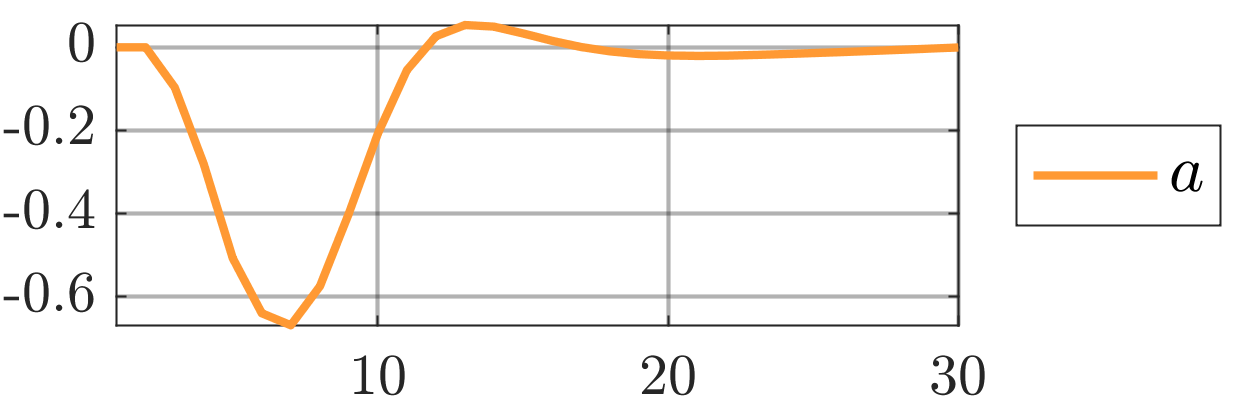}
            \end{minipage}

            \begin{minipage}{0.33\textwidth}
                \centering
                \includegraphics[width=\columnwidth]{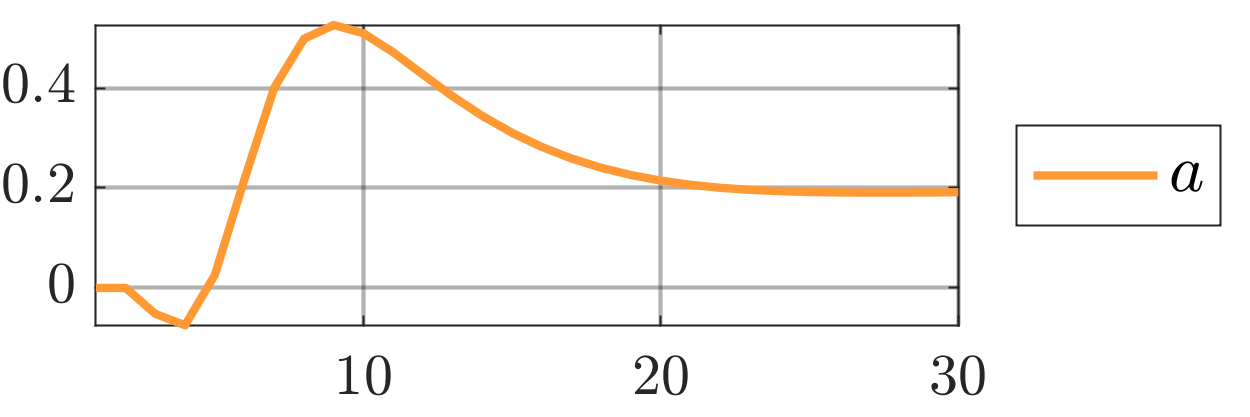}
            \end{minipage}

            \begin{minipage}{0.33\textwidth}
                \centering
                \includegraphics[width=\columnwidth]{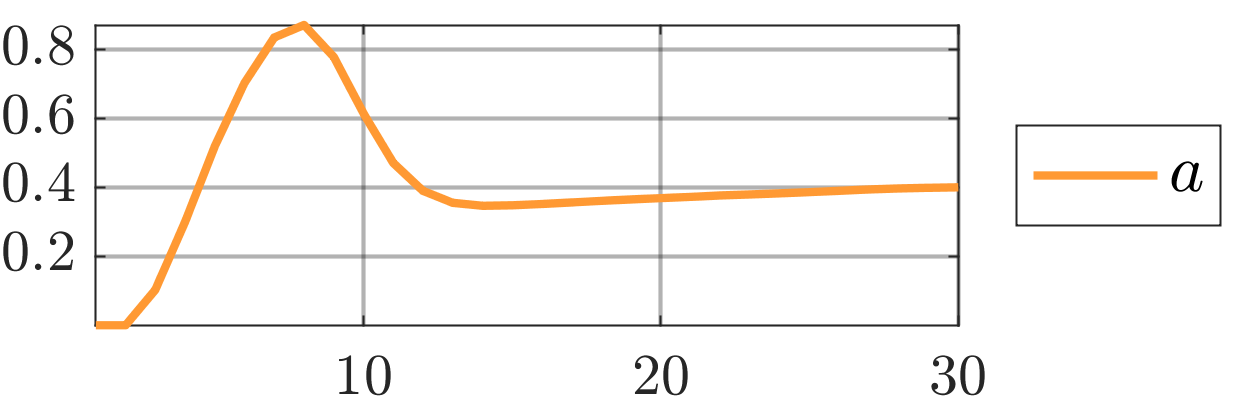}
            \end{minipage}
        \end{tabular}
    \end{minipage}

    \begin{minipage}{\textwidth}
        \centering
        \begin{tabular}{ccc}
            \begin{minipage}{0.33\textwidth}
                \centering
                \includegraphics[width=\columnwidth]{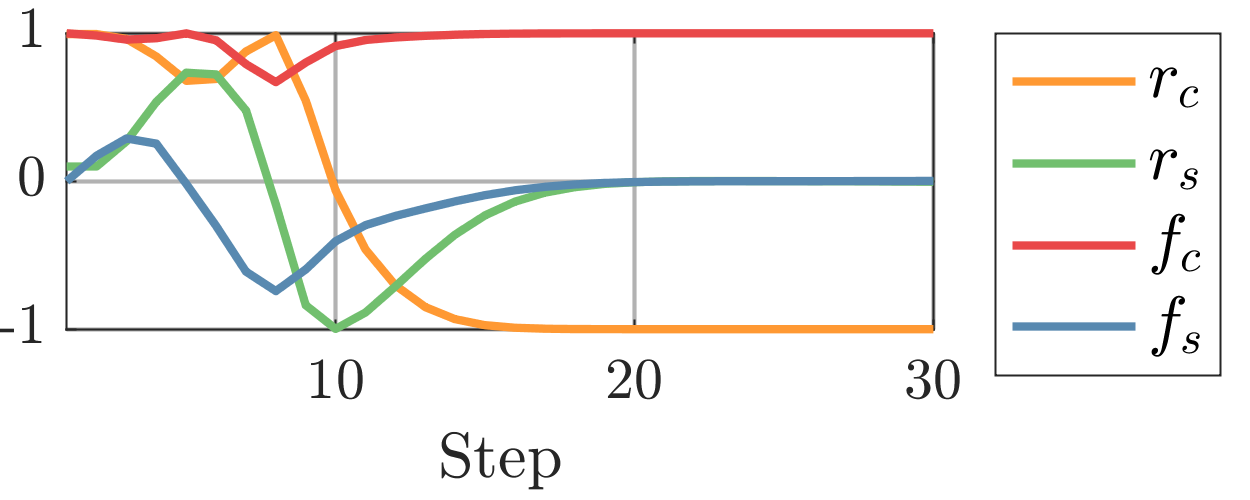}
                (a)
            \end{minipage}

            \begin{minipage}{0.33\textwidth}
                \centering
                \includegraphics[width=\columnwidth]{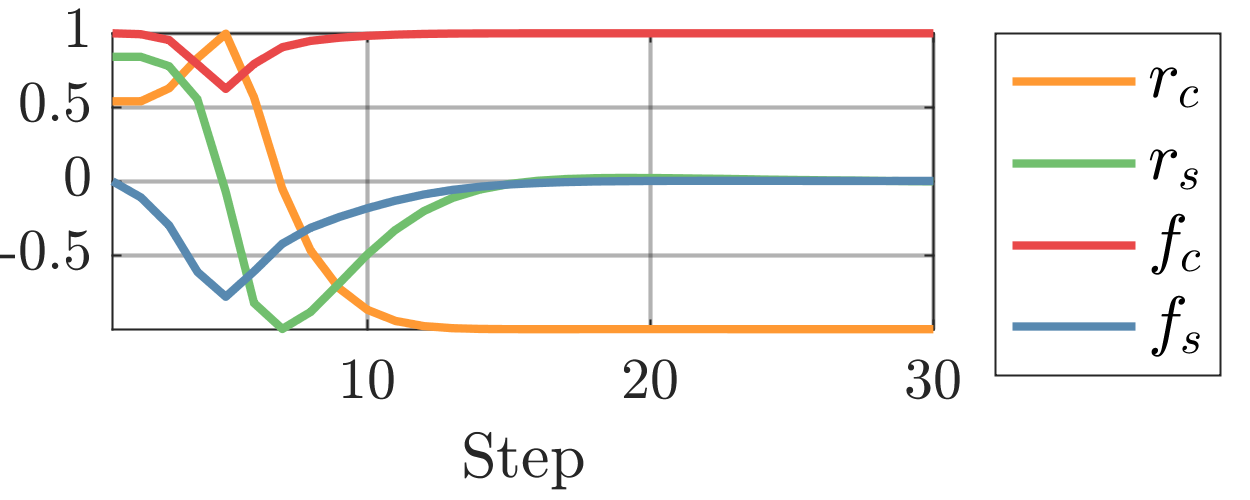}
                (b)
            \end{minipage}

            \begin{minipage}{0.33\textwidth}
                \centering
                \includegraphics[width=\columnwidth]{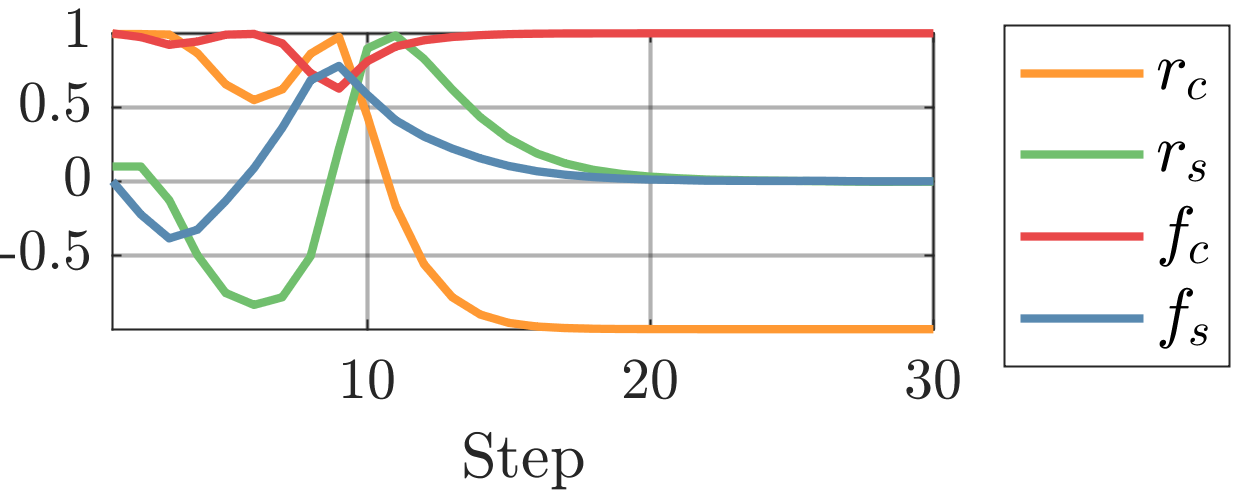}
                (c)
            \end{minipage}
        \end{tabular}
    \end{minipage}

    \caption{Global optimal trajectories for the cart-pole problem. \label{fig:exp:ca:demos-plots}}
    \vspace{3mm}
    
\end{figure}

\begin{figure}[h]
    \centering
    \begin{minipage}{\textwidth}
        \centering
        \begin{tabular}{ccc}
            \begin{minipage}{0.33\textwidth}
                \centering
                \includegraphics[width=\columnwidth]{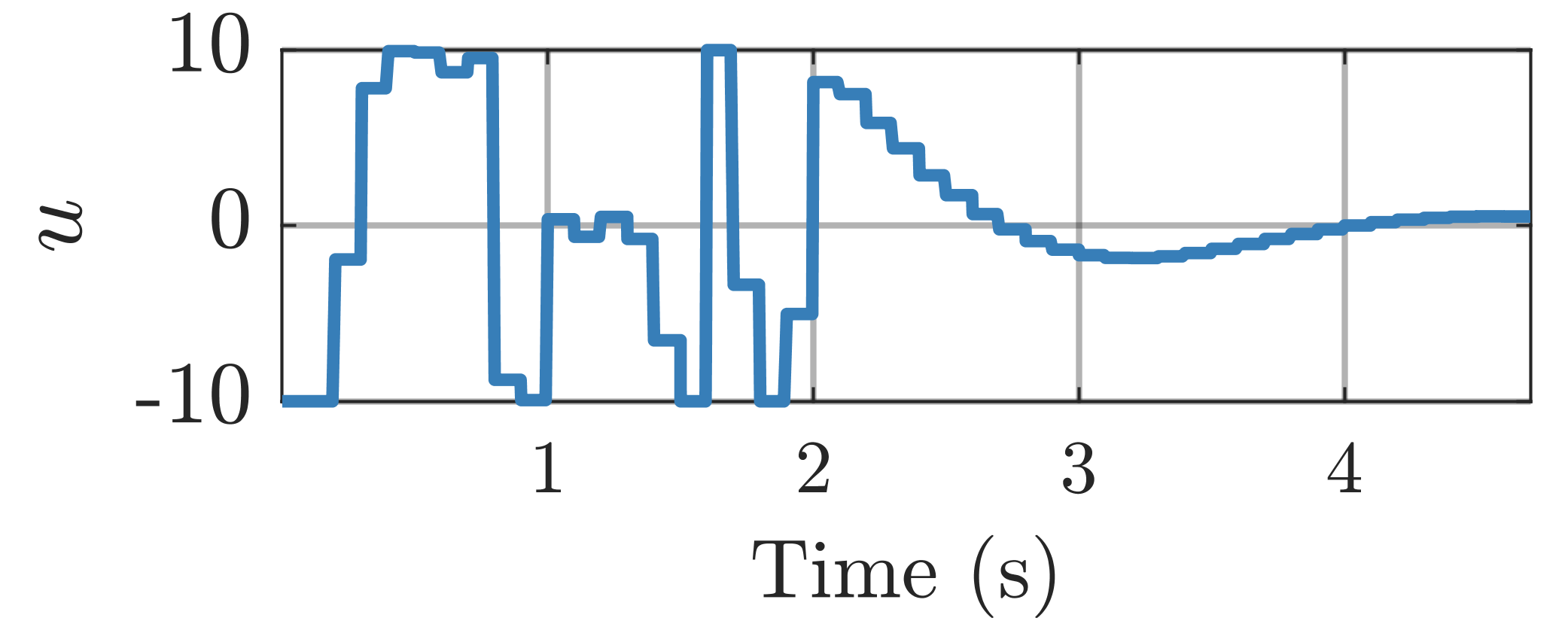}
            \end{minipage}

            \begin{minipage}{0.33\textwidth}
                \centering
                \includegraphics[width=\columnwidth]{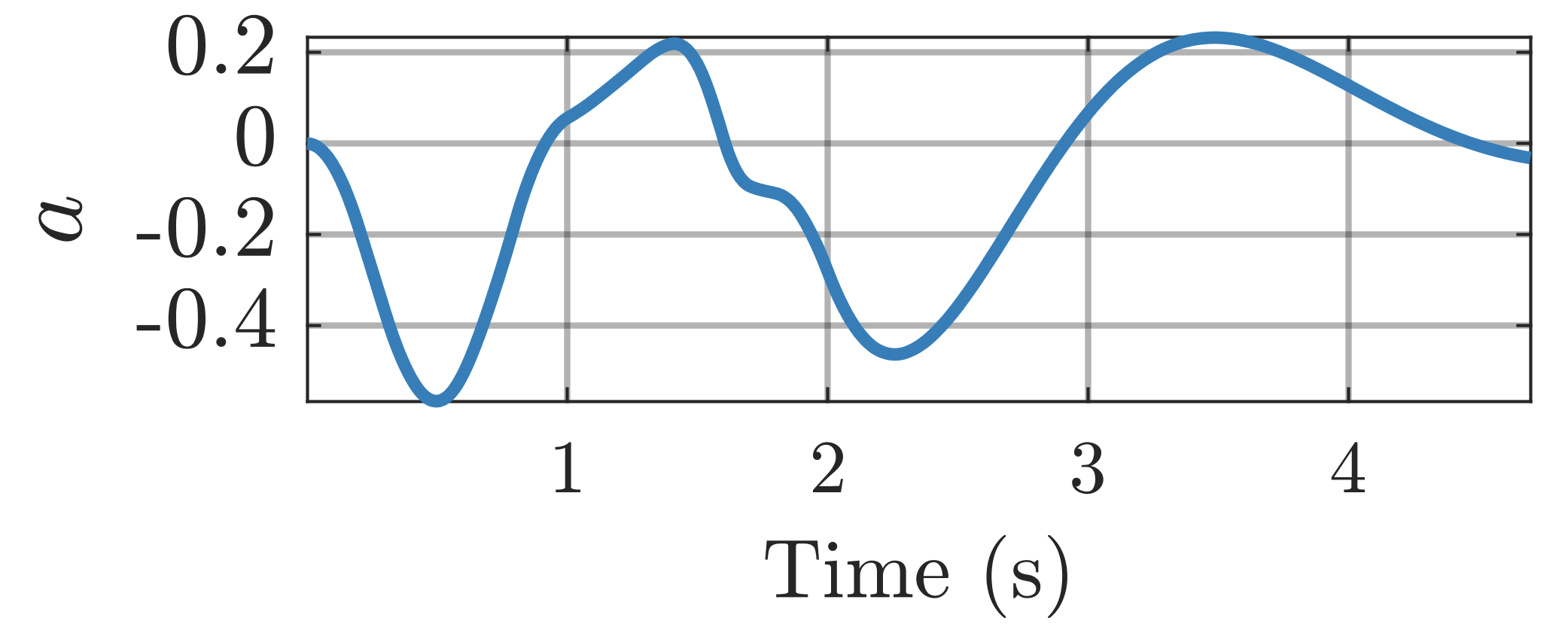}
            \end{minipage}

            \begin{minipage}{0.33\textwidth}
                \centering
                \includegraphics[width=\columnwidth]{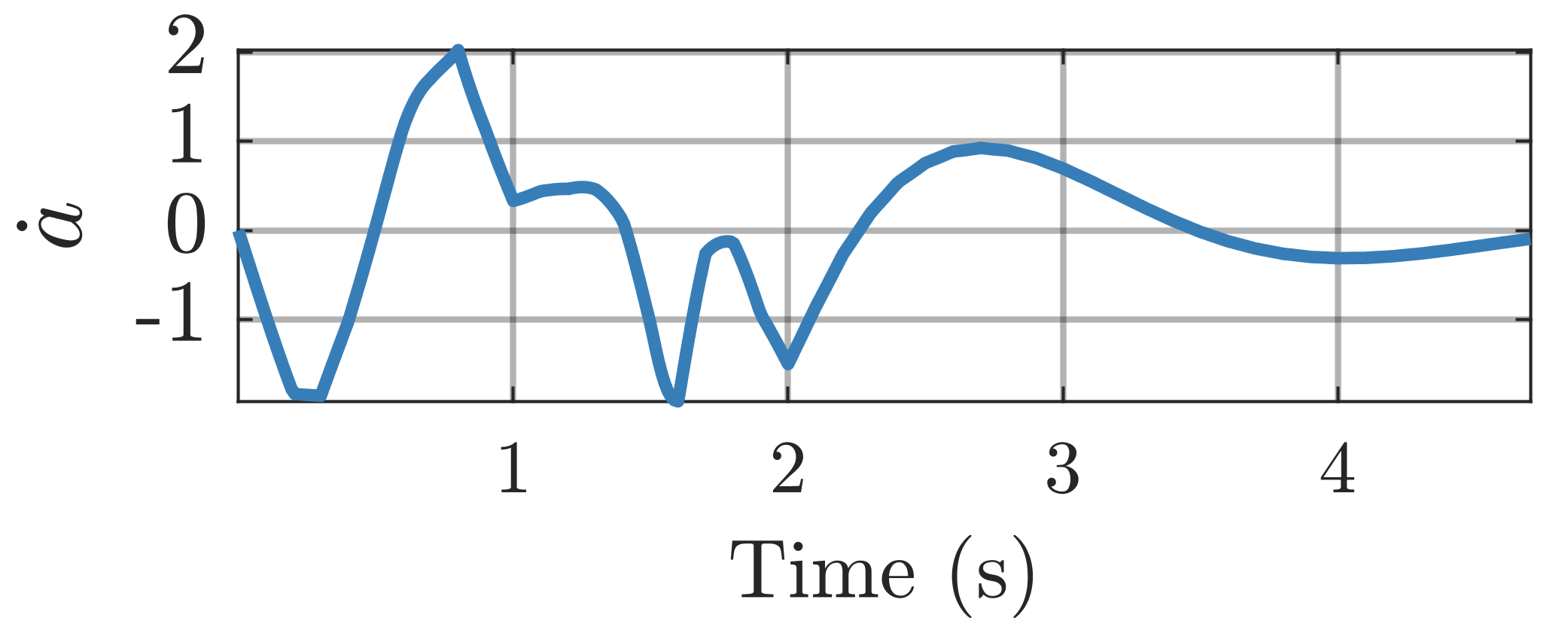}
            \end{minipage}
        \end{tabular}
    \end{minipage}

    \begin{minipage}{\textwidth}
        \centering
        \begin{tabular}{ccc}
            \begin{minipage}{0.32\textwidth}
                \centering
                \includegraphics[width=\columnwidth]{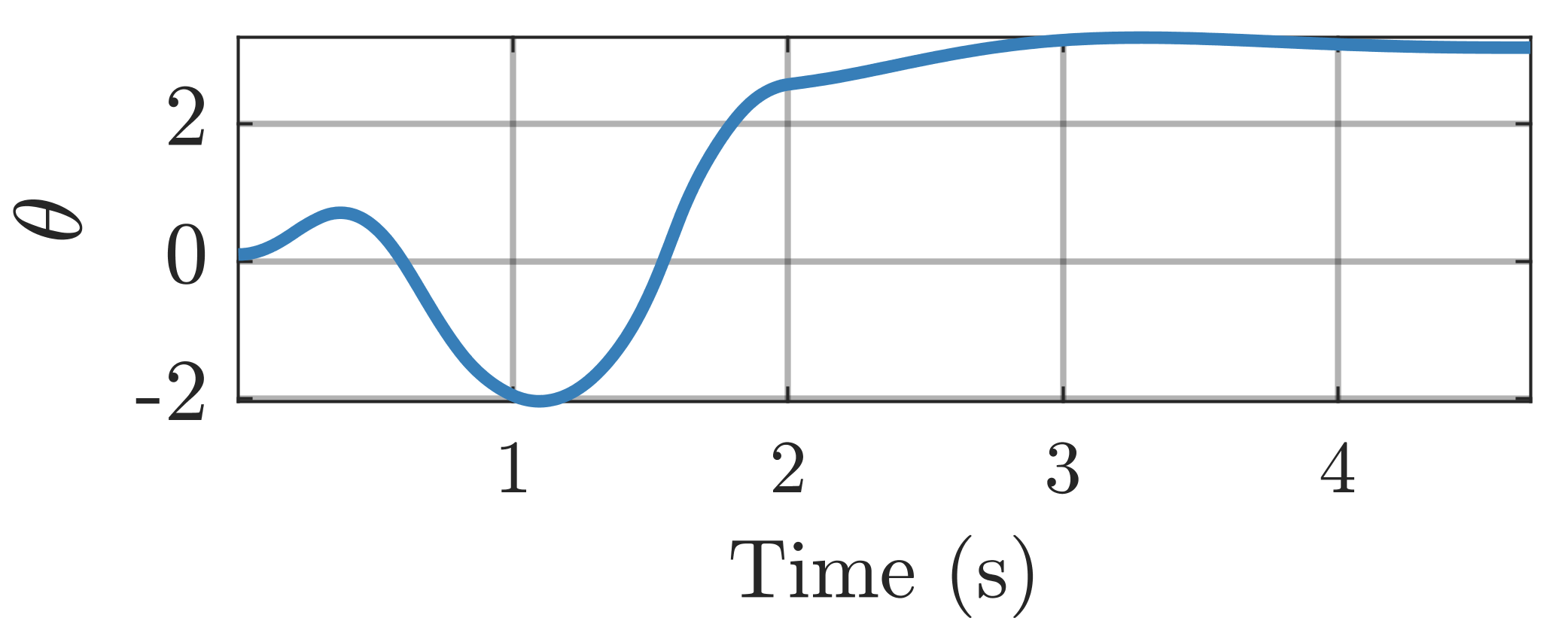}
            \end{minipage}

            \begin{minipage}{0.32\textwidth}
                \centering
                \includegraphics[width=\columnwidth]{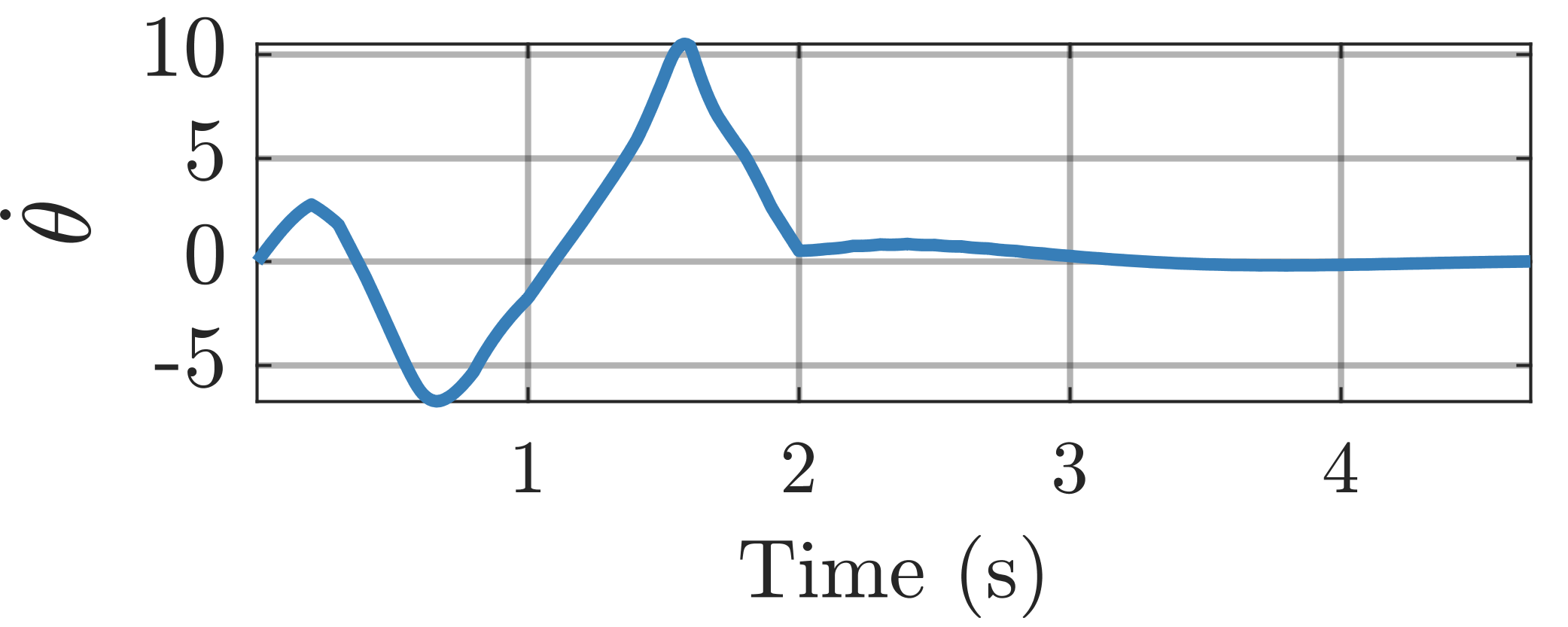}
            \end{minipage}

            \begin{minipage}{0.35\textwidth}
                \centering
                \includegraphics[width=\columnwidth]{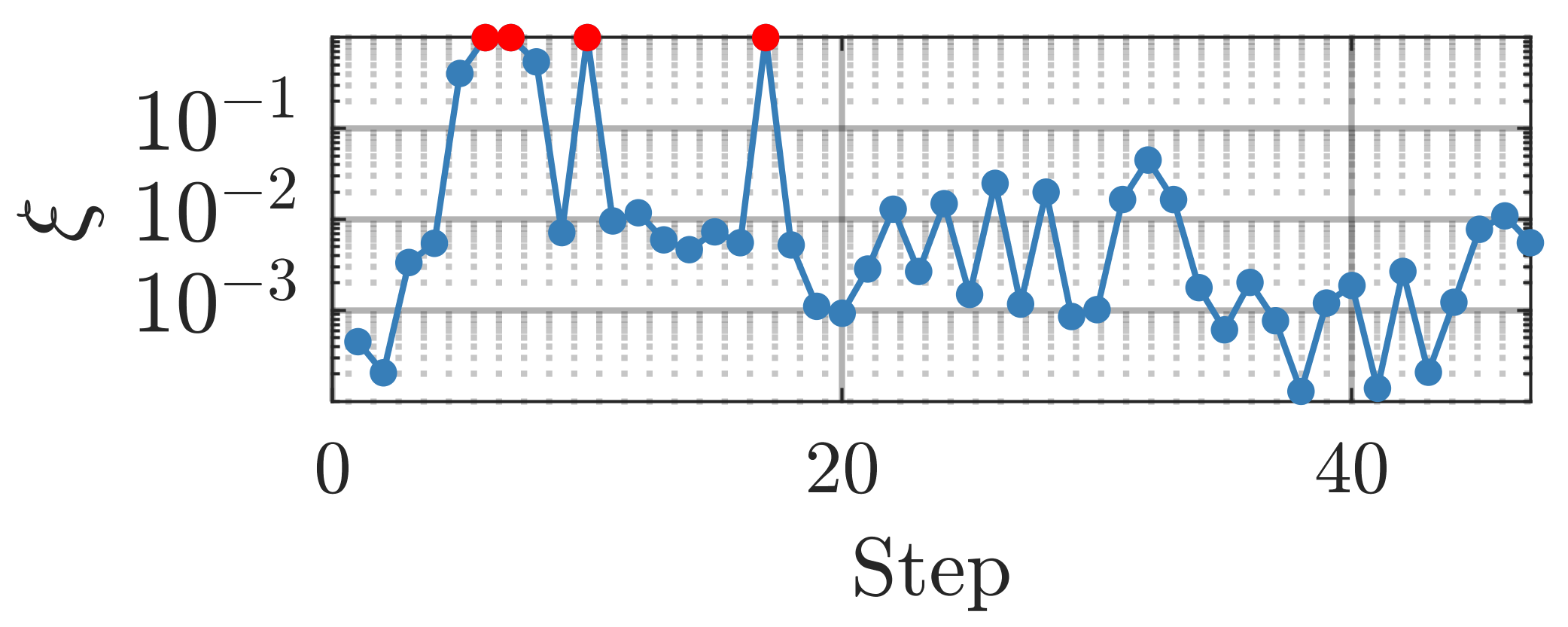}
            \end{minipage}
        \end{tabular}
    \end{minipage}

    \begin{minipage}{\textwidth}
        \centering
        \begin{tabular}{cc}
            \begin{minipage}{0.35\textwidth}
                \centering
                \includegraphics[width=\columnwidth]{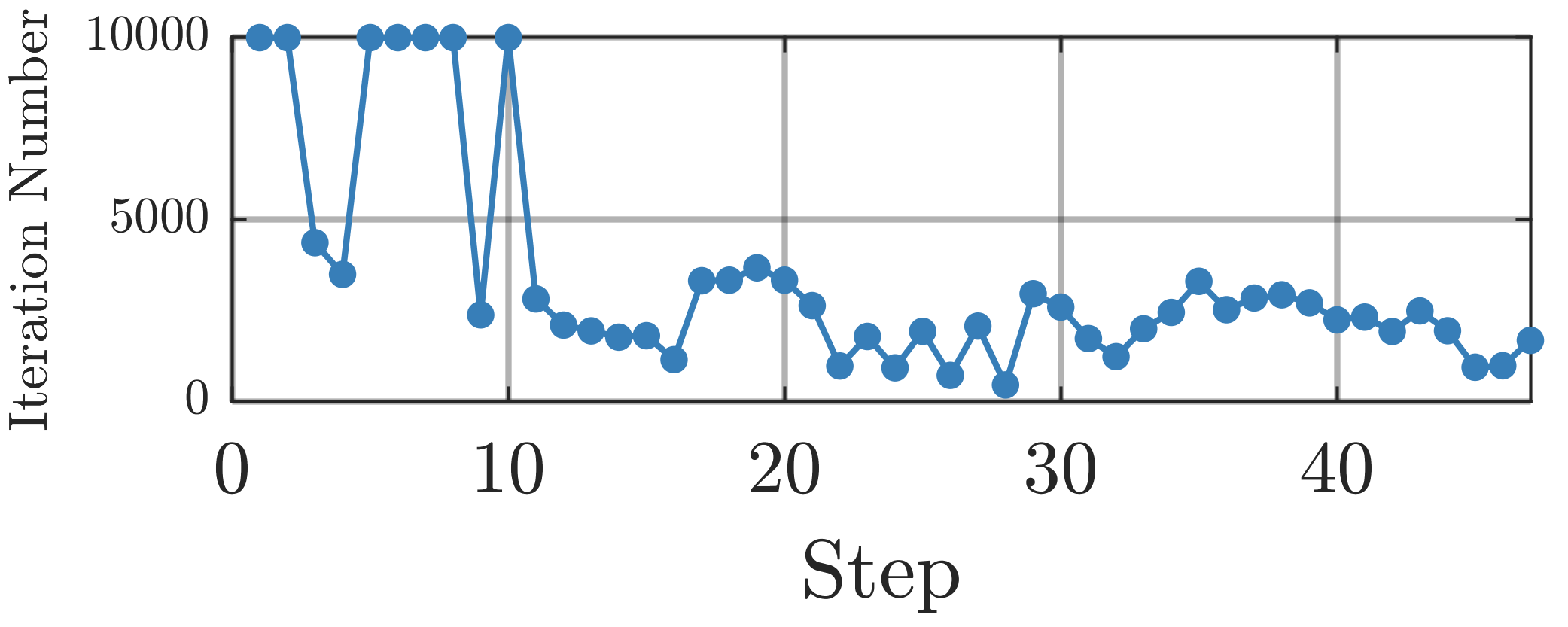}
            \end{minipage}

            \begin{minipage}{0.35\textwidth}
                \centering
                \includegraphics[width=\columnwidth]{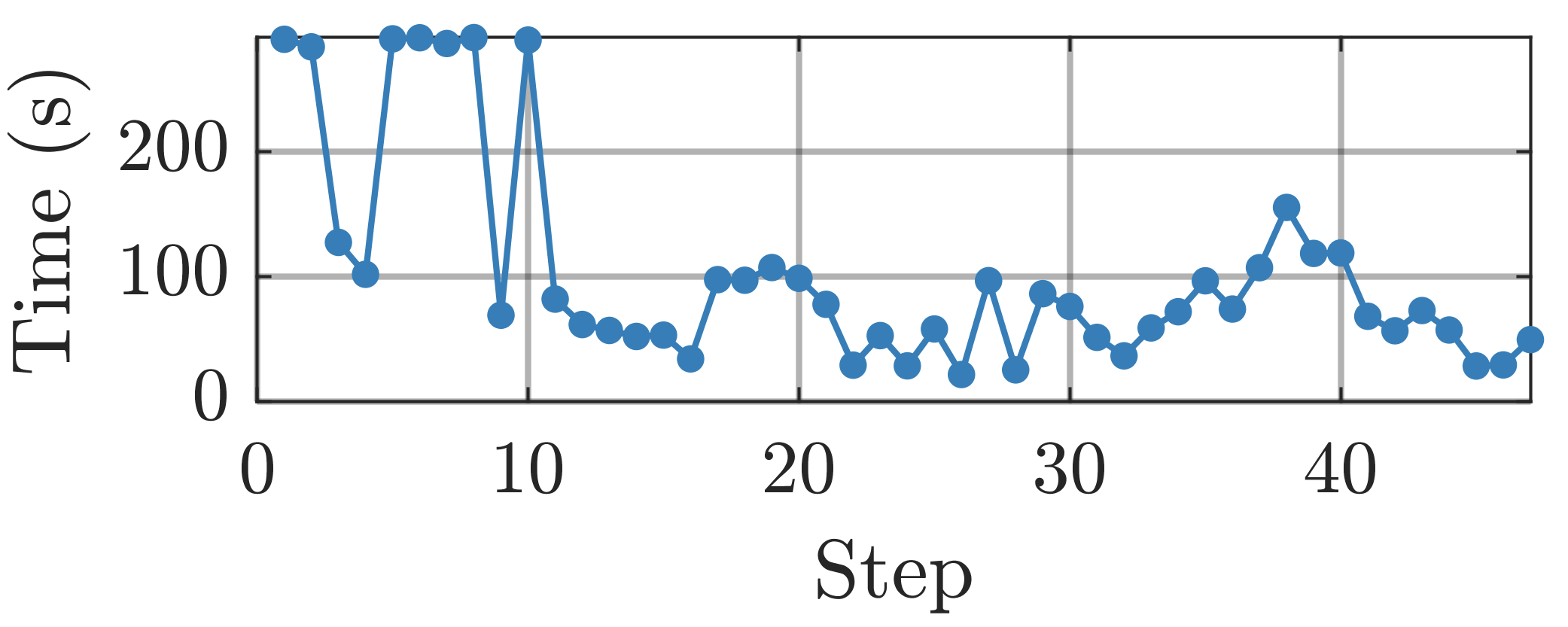}
            \end{minipage}
        \end{tabular}
    \end{minipage}

    \caption{Simulated MPC in the cart-pole case. In the suboptimality gap plot, red points indicate that \fmincon converges to an infeasible state.
    \label{fig:exp:ca:mpc}}
    \vspace{3mm}
\end{figure}

\subsection{Car Back-in}
\label{app:subsec:exp:cr}

\textbf{Dynamics and constraints.} We consider a real-world scenario where we want to put a car (modeled as a rectangle) into the middle of two rectangular obstacles, as shown in Figure~\ref{fig:exp:gen:sys-illustration} (e). According to~\cite{lynch2017book-modern-robotics}, we choose the simplified car model:
\begin{subequations}
    \label{eq:exp:cr:car-model}
    \begin{align}
        & \dot{x} = v \cos\theta \\
        & \dot{y} = v \sin\theta \\
        & \dot{\theta} = w
    \end{align}
\end{subequations}
Denote the length and width of the car as $L$ and $W$, respectively. Four vertices of the car$\left\{ (x_{c,i}, y_{c,i}) \right\}_{i=1}^{4}$ can be represented as:
\begin{subequations}
    \label{eq:exp:cr:car-vertices}
    \begin{align}
        x_{c,1} = x + \frac{L}{2} \cos\theta - \frac{W}{2} \sin\theta, \ y_{c,1} = y + \frac{L}{2} \sin \theta + \frac{W}{2} \cos\theta \\
        x_{c,2} = x - \frac{L}{2} \cos\theta - \frac{W}{2} \sin\theta, \ y_{c,2} = y - \frac{L}{2} \sin \theta + \frac{W}{2} \cos\theta \\
        x_{c,3} = x - \frac{L}{2} \cos\theta + \frac{W}{2} \sin\theta, \ y_{c,3} = y - \frac{L}{2} \sin \theta - \frac{W}{2} \cos\theta \\
        x_{c,4} = x + \frac{L}{2} \cos\theta + \frac{W}{2} \sin\theta, \ y_{c,4} = y + \frac{L}{2} \sin \theta - \frac{W}{2} \cos\theta 
    \end{align}
\end{subequations}
Denote obstacle 1 and 2's vertices as $\left\{ (x_{o1,i}, y_{o1,i}) \right\}_{i=1}^{4}$ and $\left\{ (x_{o2,i}, y_{o2,i}) \right\}_{i=1}^{4}$. To write the non-collision condition as polynomial constraints, we introduce six lifting variables $(A_1, B_1, C_1)$ and $(A_2, B_2, C_2)$, representing two separation lines between the car and two obstacles. Due to the hyperplane separation theorem in convex optimization~\cite{boyd2004book-convex-optimization}, there is no collision between the car and two obstacles if and only if there exist $(A_1, B_1, C_1)$ and $(A_2, B_2, C_2)$, such that
\begin{subequations}
    \label{eq:exp:cr:separation}
    \begin{align}
        A_1 x_{c,i} + B_1 y_{c,i} + C_1 \ge 0, \  A_1 x_{o1,j} + B_1 y_{o1,j} + C_1 \le 0, \ \forall i \in \seqordering{4}, j \in \seqordering{4} \\
        A_2 x_{c,i} + B_2 y_{c,i} + C_2 \ge 0, \  A_2 x_{o2,j} + B_2 y_{o2,j} + C_2 \le 0, \ \forall i \in \seqordering{4}, j \in \seqordering{4}
    \end{align}
\end{subequations}
\eqref{eq:exp:cr:separation} contains $16$ polynomial constraints. Thus, the discretized dynamics and constraints are:
\begin{subequations}
    \label{eq:exp:cr:dis-dyn-constraints}
    \begin{align}
        & \rx{k+1} - \rx{k} = \dt \cdot v_k \rc{k} \\
        & \ry{k+1} - \ry{k} = \dt \cdot v_k \rs{k} \\
        & \fs{k} = \dt \cdot w_k - \frac{1}{6} (\dt \cdot w_k)^3 \label{eq:exp:cr:dis-dyn-constraints:thirdorder-approx} \\
        & ~\eqref{eq:exp:p:dis-dyn-constraints-rcupdate},~\eqref{eq:exp:p:dis-dyn-constraints-rsupdate},~\eqref{eq:exp:p:dis-dyn-constraints-so2-r},~\eqref{eq:exp:p:dis-dyn-constraints-so2-f} \\
        & A_{1,k} x_{c,i} + B_{1,k} y_{c,i} + C_{1,k} \ge 0, \  A_{1,k} x_{o1,j} + B_{1,k} y_{o1,j} + C_{1,k} \le 0, \\
        & A_{2,k} x_{c,i} + B_{2,k} y_{c,i} + C_{2,k} \ge 0, \  A_{2,k} x_{o2,j} + B_{2,k} y_{o2,j} + C_{2,k} \le 0, \\
        & A_{1,k}^2 + B_{1,k}^2 + C_{1,k}^2 = 1, \ A_{2,k}^2 + B_{2,k}^2 + C_{2,k}^2 = 1 \label{eq:exp:cr:dis-dyn-constraints:thirdorder-abc-sphere} \\
        & \forall i \in \seqordering{4}, j \in \seqordering{4} \\
        & \rx{\max}^2 - \rx{k}^2 \ge 0, \ \ry{\max}^2 - \ry{k}^2 \ge 0 \\
        & v_{\max}^2 - v_k^2 \ge 0, \ w_{\max}^2 - w_k^2 \ge 0
    \end{align}
\end{subequations}
\eqref{eq:exp:cr:dis-dyn-constraints:thirdorder-approx} provides a tighter approximation of the angular velocity compared to the first-order approximation $\fs{k} = \dt \cdot w_k$ in Lie group variational integrator, especially when $\dt$ is large ($> 0.2$s). ~\eqref{eq:exp:cr:dis-dyn-constraints:thirdorder-abc-sphere} ensures the Archimedean condition holds. From $x = x_0, y = y_0, \theta = \theta_0$, we want the car to reach $x = 0, y = -3, \theta = \frac{\pi}{2}$ while avoiding the two obstacles.

\textbf{Hyper-parameters.} Denote $x_k$ as $[\rx{k} \; \ry{k} \; \rc{k} \; \rs{k}] \in \Real{4}$ and $u_k$ as $[v_k \; w_k \; \fc{k} \; \fs{k}] \in \Real{4}$. In all experiments, take $L = 6, W = 2.5, N = 30, \dt = 0.25$ and $\rx{\max} = x_0 + 2, \ry{\max} = 8, v_{\max} = 4, w_{\max} = 0.5$. Two obstacles are set as squares, with side length $8$ and centers $(6,-4)$ and $(-6,-4)$. $P_f$ is set to $10$. The number of localizing matrices is $659$. For first-order methods, \texttt{maxiter} is set to $12000$.

\textbf{Extract optimal solution with partial rank-1 property.} 
Even with the sphere constraints~\eqref{eq:exp:cr:dis-dyn-constraints:thirdorder-abc-sphere}, the pairs \((A, B, C)\) will not be unique for a given trajectory. In fact, for a specific non-collision configuration, an infinite number of \((A, B, C)\) pairs can separate two rectangles. Consequently, the moment matrices will not be rank-1. However, in practice, we observe that the submatrix containing elements corresponding to \((x_k, u_k, x_{k+1})\) remains nearly rank-1. Thus, the optimal trajectory can still be extracted from this submatrix.

\subsection{Vehicle Landing}
\label{app:subsec:exp:vl}

\textbf{Dynamics and constraints.} For a spacecraft, its simplified continuous-time dynamics are:
\begin{subequations}
    \label{eq:exp:vl:con-dyn}
    \begin{align}
        & m \ddot{x} = -(u_1 + u_2) \sin\theta \\
        & m \ddot{y} = (u_1 + u_2) \cos\theta - m g \\
        & I \ddot{\theta} = L (u_2 - u_1) 
    \end{align}
\end{subequations}
$I$ is the inertia related to the center of mass. Figure~\ref{fig:exp:gen:sys-illustration} (c) illustrates the parameters of a spacecraft. Similar to the inverted pendulum, we can write the discretized dynamics and constraints as:
\begin{subequations}
    \label{eq:exp:vl:dis-dyn-constraints}
    \begin{align}
        & \rx{k+1} - \rx{k} = \dt \cdot \vx{k}, \ \ry{k+1} - \ry{k} = \dt \cdot \vy{k} \\
        & m \cdot (\vx{k+1} - \vx{k}) = -\dt \cdot (u_{1,k} + u_{2,k}) \rs{k} \\
        & m \cdot (\vy{k+1} - \vy{k}) = \dt \cdot (u_{1,k} + u_{2,k}) \rc{k} - mg \\
        & I \cdot (\fs{k+1} - \fs{k}) = L \dt^2 \cdot (u_{2,k} - u_{1,k}) \\
        & ~\eqref{eq:exp:p:dis-dyn-constraints-rcupdate},~\eqref{eq:exp:p:dis-dyn-constraints-rsupdate},~\eqref{eq:exp:p:dis-dyn-constraints-so2-r},~\eqref{eq:exp:p:dis-dyn-constraints-so2-f},~\eqref{eq:exp:p:dis-dyn-constraints-fcmin} \\
        & u_{i,k} \cdot (u_{i,\max} - u_{i,k}) \ge 0, i = 1, 2 \\
        & \rx{\max}^2 - \rx{k}^2 \ge 0, \ (\ry{k} - \ry{\min}) (\ry{\max} - \ry{k}) \ge 0 \\
        & \vx{\max}^2 - \vx{k}^2 \ge 0, \ \vy{\max}^2 - \vy{k}^2 \ge 0 
    \end{align}
\end{subequations}
Starting from $x = x_0, y = y_0, \theta = \theta_0, \dot{x} = \dot{x}_0, \dot{y} = \dot{y}_0, \dot{\theta} = \dot{\theta}_0$, we want the spacecraft to land at $x = 0, y = 10, \theta = 0$ with linear translational and angular velocities equal to $0$.

\textbf{Hyper-parameters.} Denote $x_k$ as $[\rx{k} \; \ry{k} \; \vx{k} \; \vy{k} \; \rc{k} \; \rs{k} \; \fc{k} \; \fs{k}] \in \Real{8}$ and $u_k$ as $[u_{1, k} \; u_{2,k}] \in \Real{2}$. In all experiments, take $m = 1, I = 50, L = 5, N = 50, \dt = 0.2, \fc{\min} = 0.7$ and $\rx{\max} = 100, \ry{\min} = 10, \ry{\max} = 120, \vx{\max} = \vy{\max} = 20, u_{1,\max} = u_{2,\max} = 8$. $P_f$ in~\eqref{eq:exp:p:con-dyn} is set to $10$. The number of localizing matrices is $499$. For first-order methods, \texttt{maxiter} is set to $10000$. 

Figure~\ref{fig:exp:vl:demos-plots} shows three globally optimal trajectories.


\begin{figure}[t]
    \begin{minipage}{\textwidth}
        \centering
        \begin{tabular}{ccc}
            \begin{minipage}{0.33\textwidth}
                \centering
                \includegraphics[width=\columnwidth]{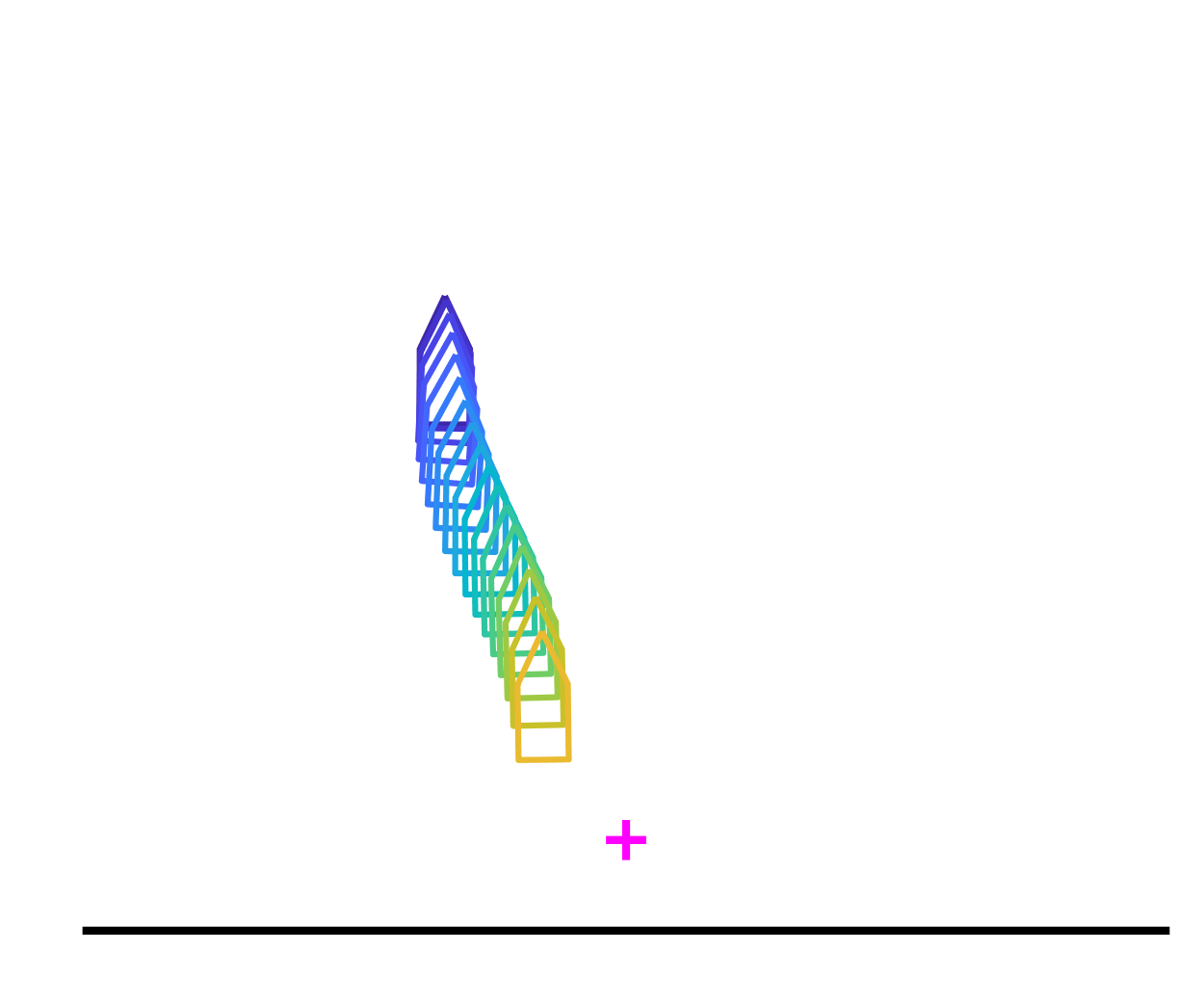}
            \end{minipage}

            \begin{minipage}{0.33\textwidth}
                \centering
                \includegraphics[width=\columnwidth]{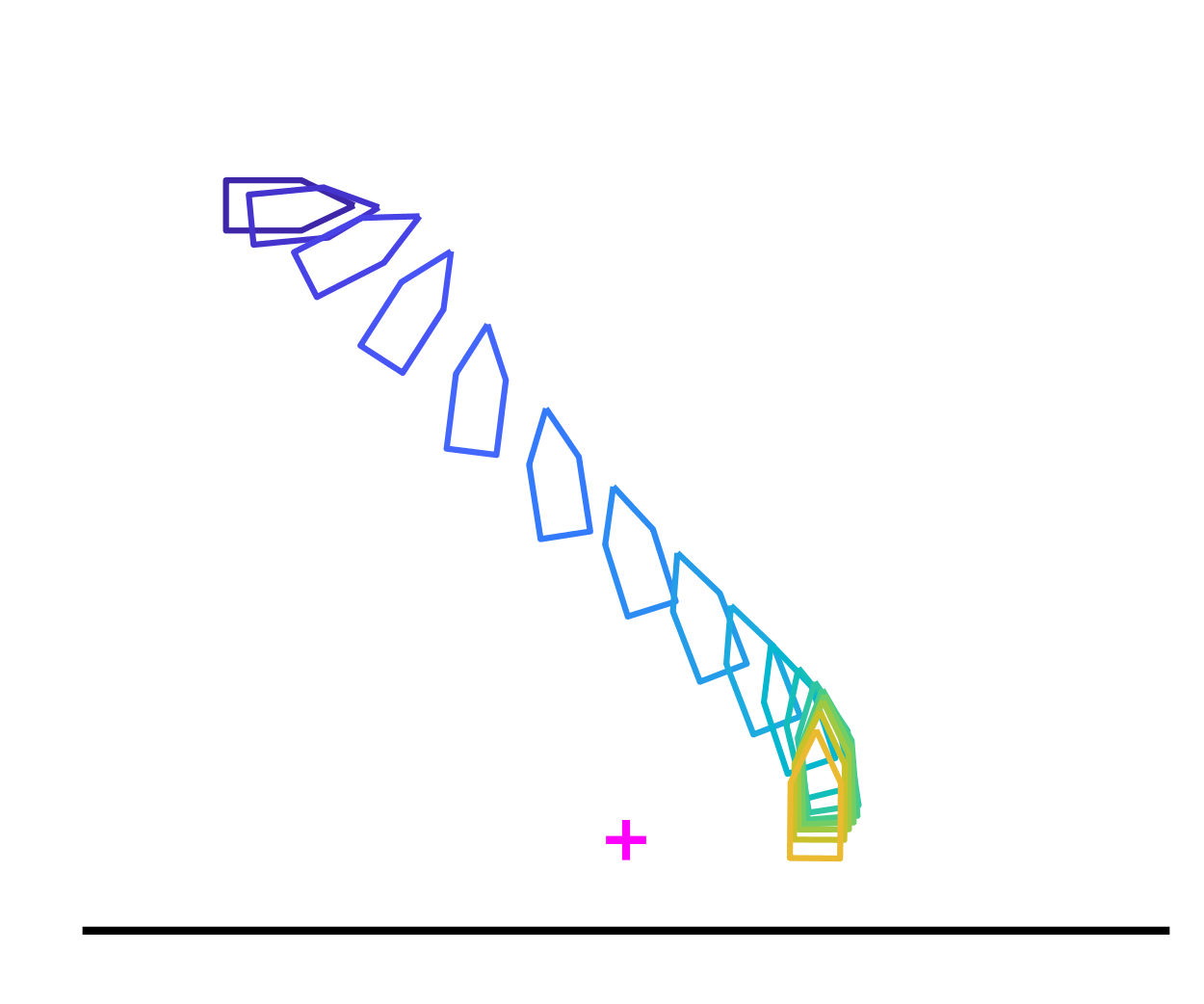}
            \end{minipage}

            \begin{minipage}{0.33\textwidth}
                \centering
                \includegraphics[width=\columnwidth]{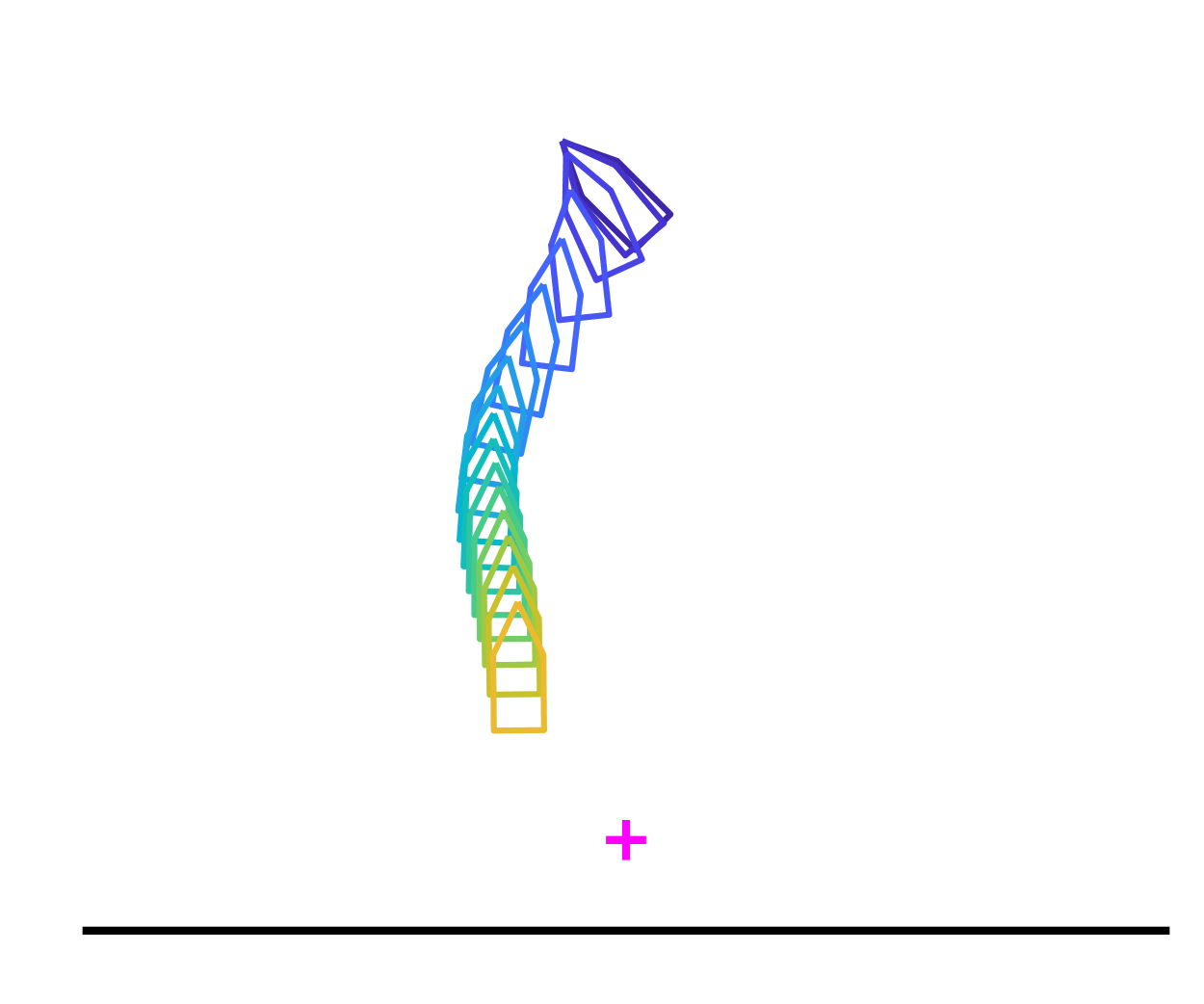}
            \end{minipage}
        \end{tabular}
    \end{minipage}

    \begin{minipage}{\textwidth}
        \centering
        \begin{tabular}{ccc}
            \begin{minipage}{0.33\textwidth}
                \centering
                \includegraphics[width=\columnwidth]{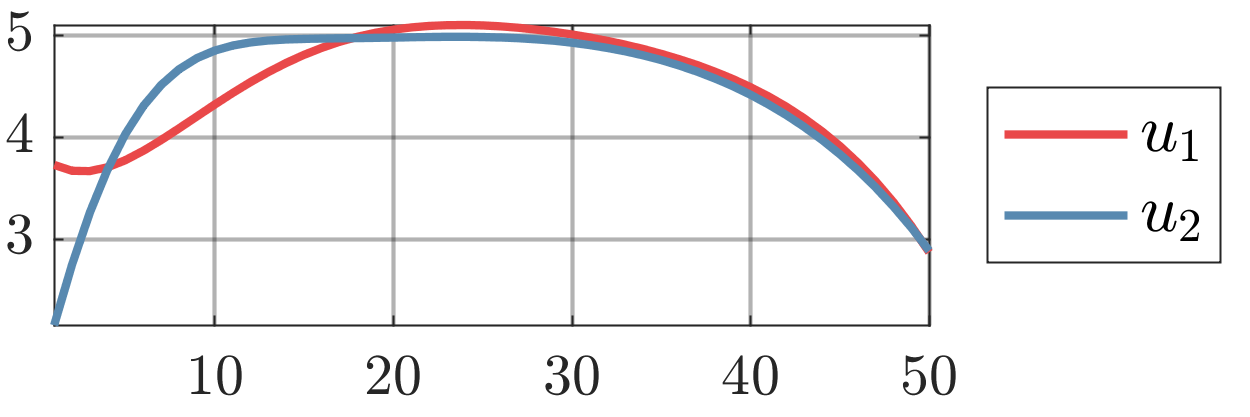}
            \end{minipage}

            \begin{minipage}{0.33\textwidth}
                \centering
                \includegraphics[width=\columnwidth]{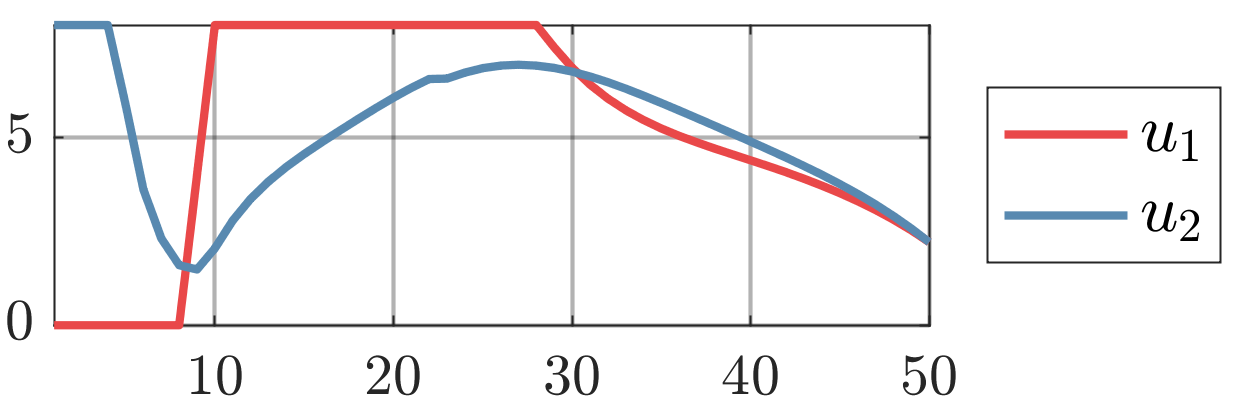}
            \end{minipage}

            \begin{minipage}{0.33\textwidth}
                \centering
                \includegraphics[width=\columnwidth]{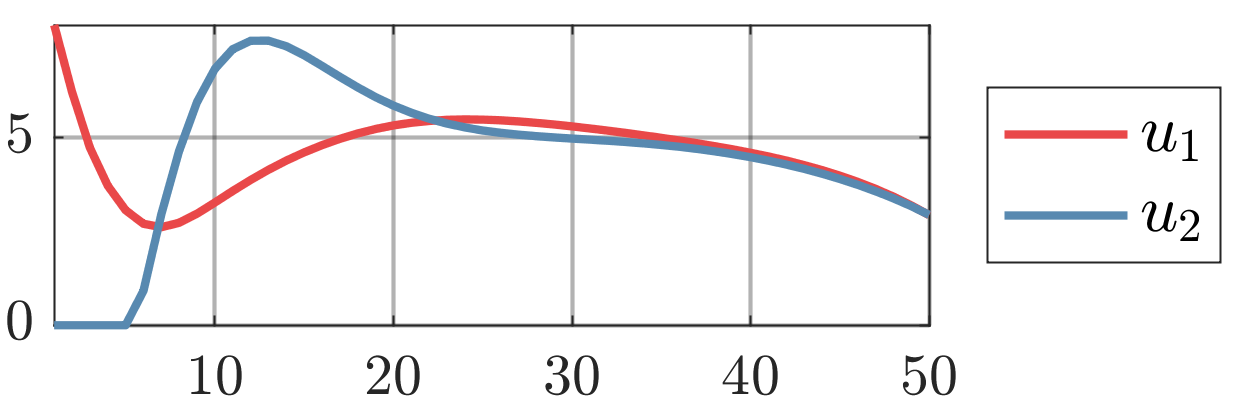}
            \end{minipage}
        \end{tabular}
    \end{minipage}

    \begin{minipage}{\textwidth}
        \centering
        \begin{tabular}{ccc}
            \begin{minipage}{0.33\textwidth}
                \centering
                \includegraphics[width=\columnwidth]{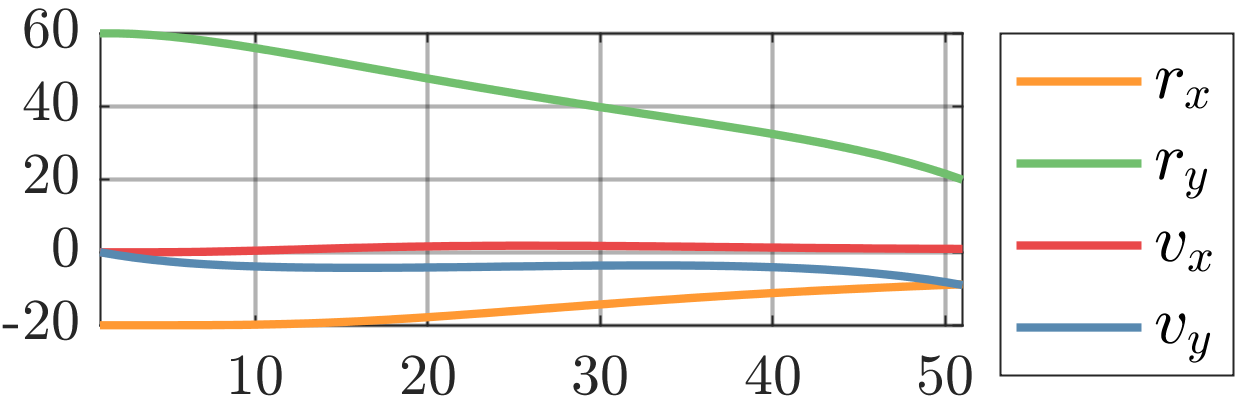}
            \end{minipage}

            \begin{minipage}{0.33\textwidth}
                \centering
                \includegraphics[width=\columnwidth]{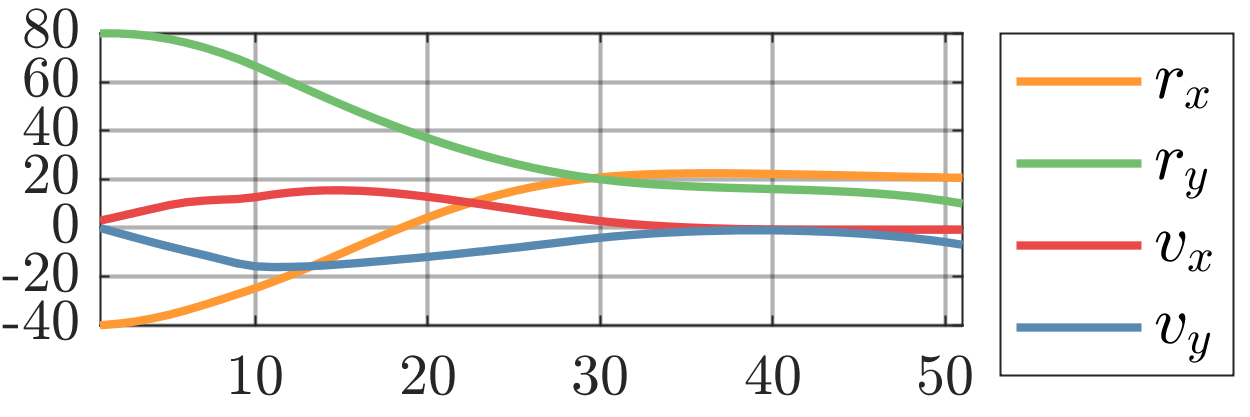}
            \end{minipage}

            \begin{minipage}{0.33\textwidth}
                \centering
                \includegraphics[width=\columnwidth]{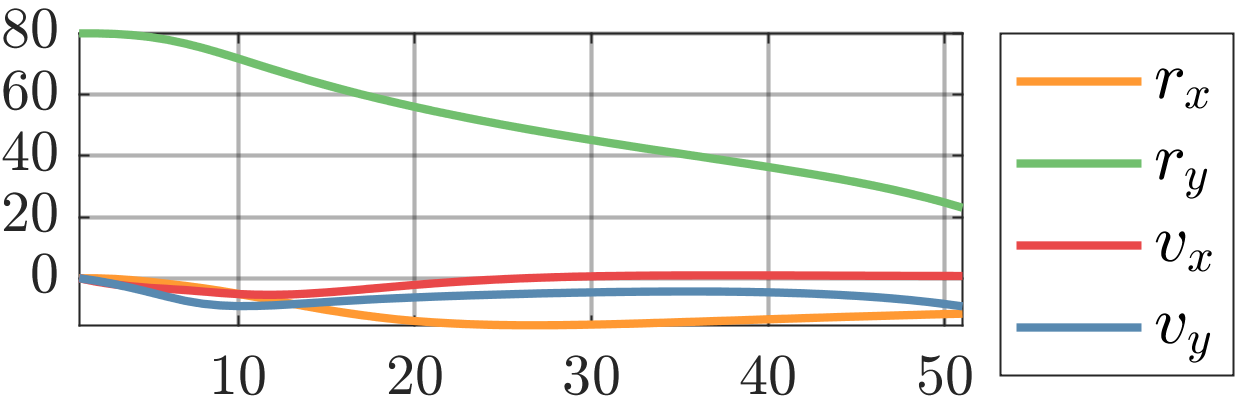}
            \end{minipage}
        \end{tabular}
    \end{minipage}

    \begin{minipage}{\textwidth}
        \centering
        \begin{tabular}{ccc}
            \begin{minipage}{0.33\textwidth}
                \centering
                \includegraphics[width=\columnwidth]{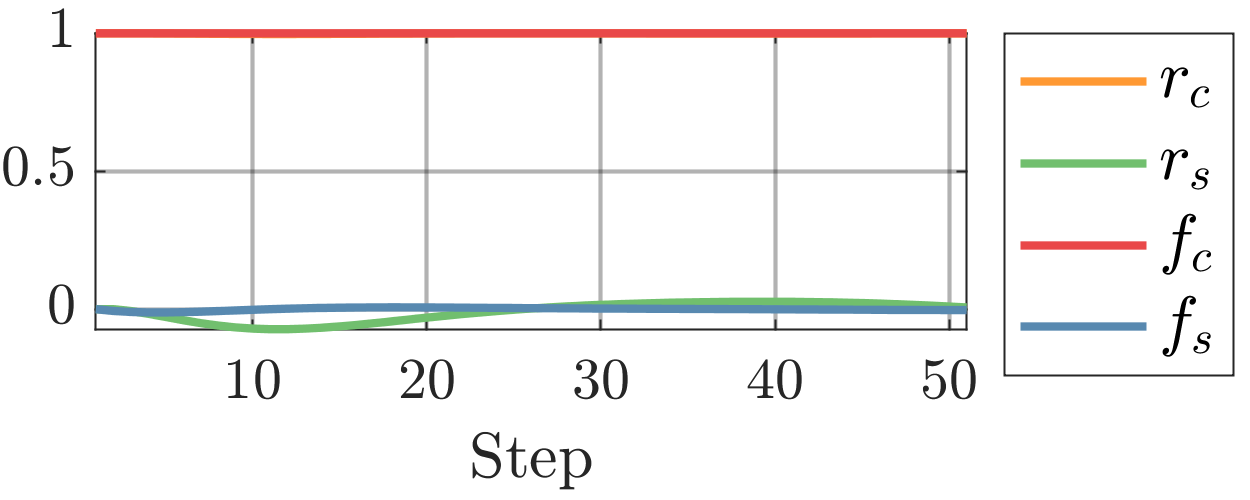}
                (a)
            \end{minipage}

            \begin{minipage}{0.33\textwidth}
                \centering
                \includegraphics[width=\columnwidth]{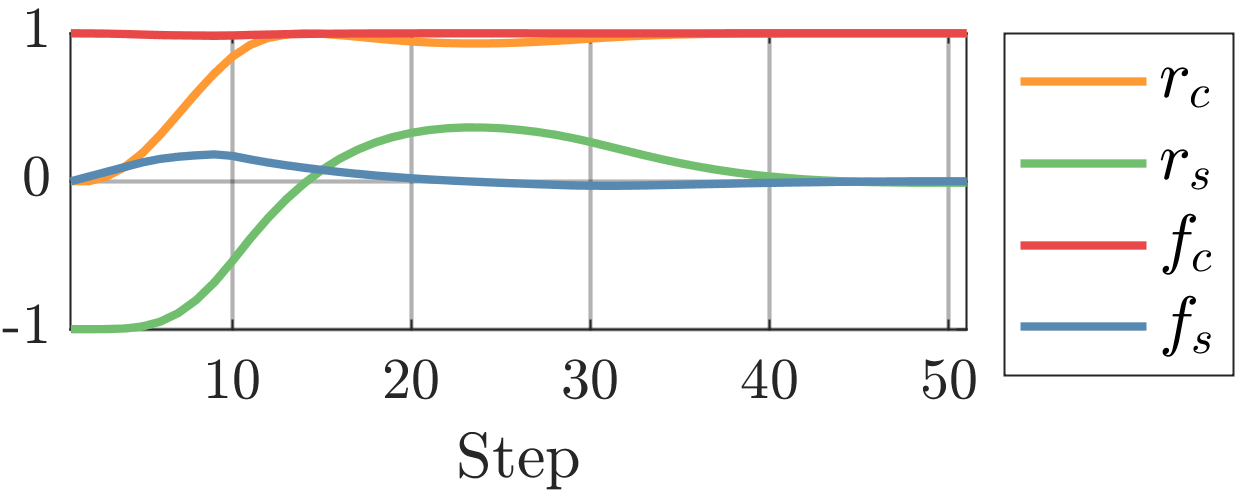}
                (b)
            \end{minipage}

            \begin{minipage}{0.33\textwidth}
                \centering
                \includegraphics[width=\columnwidth]{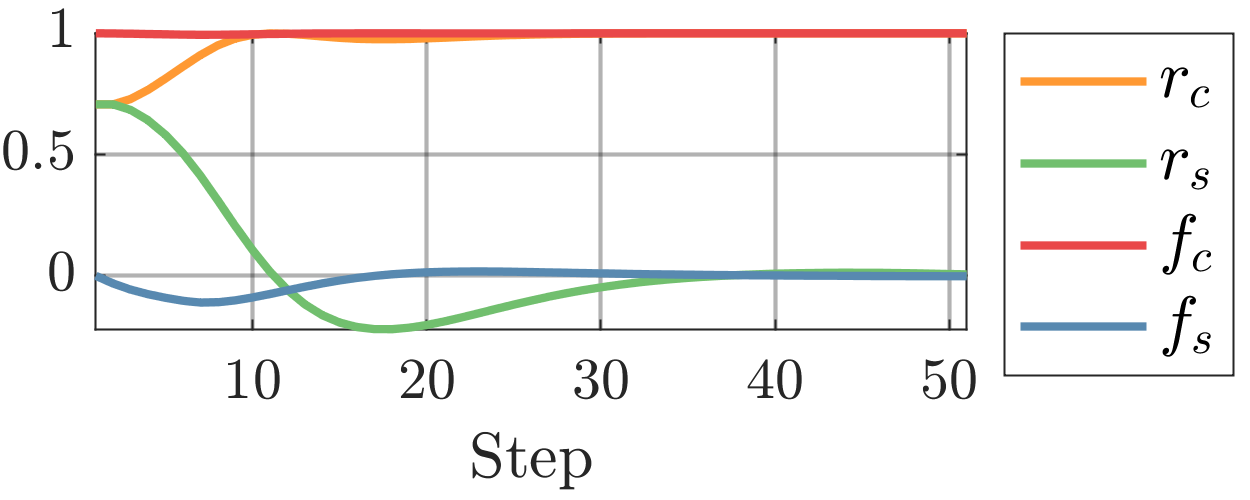}
                (c)
            \end{minipage}
        \end{tabular}
    \end{minipage}

    \caption{Global optimal trajectories for the vehicle landing problem.\label{fig:exp:vl:demos-plots}}
    \vspace{3mm}
    
\end{figure}

\subsection{Flying Robot}
\label{app:subsec:exp:fr}

\textbf{Dynamics and constraints.} We consider a drifting robot with four fumaroles, as shown in Figure~\ref{fig:exp:gen:sys-illustration} (d). The continuous dynamics are:
\begin{subequations}
    \label{eq:exp:fr:con-dyn}
    \begin{align}
        & \ddot{x} = (u_1 - u_2 + u_3 - u_4) \cos\theta \\
        & \ddot{y} = (u_1 - u_2 + u_3 - u_4) \sin\theta \\
        & \ddot{\theta} = \alpha (u_1 - u_2) + \beta (u_3 - u_4) 
    \end{align}
\end{subequations}

As shown in Figure~\ref{fig:exp:gen:sys-illustration} (d). The corresponding discretized nonlinear dynamics and constraints are:
\begin{subequations}
    \label{eq:exp:fr:dis-dyn-constraints}
    \begin{align}
        & \rx{k+1} - \rx{k} = \dt \cdot \vx{k}, \ \ry{k+1} - \ry{k} = \dt \cdot \vy{k} \\
        & (\vx{k+1} - \vx{k}) = \dt \cdot (u_{1,k} - u_{2,k} + u_{3,k} - u_{4,k}) \rc{k} \\
        & (\vy{k+1} - \vy{k}) = \dt \cdot (u_{1,k} - u_{2,k} + u_{3,k} - u_{4,k}) \rs{k} \\
        & (\fs{k+1} - \fs{k}) = \dt^2 \alpha \cdot (u_{2,k} - u_{1,k}) + \dt^2 \beta \cdot (u_{3,k} - u_{4,k}) \\
        & ~\eqref{eq:exp:p:dis-dyn-constraints-rcupdate},~\eqref{eq:exp:p:dis-dyn-constraints-rsupdate},~\eqref{eq:exp:p:dis-dyn-constraints-so2-r},~\eqref{eq:exp:p:dis-dyn-constraints-so2-f},~\eqref{eq:exp:p:dis-dyn-constraints-fcmin} \\
        & u_{i,k} \cdot (u_{i,\max} - u_{i,k}) \ge 0, i = 1, 2, 3, 4 \\
        & \rx{\max}^2 - \rx{k}^2 \ge 0, \ \ry{\max}^2 - \ry{k}^2 \ge 0 \\
        & \vx{\max}^2 - \vx{k}^2 \ge 0, \ \vy{\max}^2 - \vy{k}^2 \ge 0 
    \end{align}
\end{subequations}
From $x = x_0, y = y_0, \theta = \theta_0, \dot{x} = \dot{x}_0, \dot{y} = \dot{y}_0, \dot{\theta} = \dot{\theta}_0$, we want the drifting robot to reach $x = 0, y = 0, \theta = 0$ with translational and angular velocities equal to $0$. 

\textbf{Hyper-parameters.} Denote $x_k$ as $[\rx{k} \; \ry{k} \; \vx{k} \; \vy{k} \; \rc{k} \; \rs{k} \; \fc{k} \; \fs{k}] \in \Real{8}$ and $u_k$ as $[u_{1, k} \; u_{2,k} \; u_{3,k} \; u_{4,k}] \in \Real{4}$. In all experiments, take $\alpha = 0.2, \beta = 0.2, N = 60, \dt = 0.2, \fc{\min} = 0.7$ and $\rx{\max} = 10, \ry{\max} = 10, \vx{\max} = \vy{\max} = 10, u_{i,\max} = 8, i \in \seqordering{4}$. $P_f$ in~\eqref{eq:exp:p:con-dyn} is set to $10$. The number of localizing matrices is $659$. For first-order methods, \texttt{maxiter} is set to $12000$.

\end{document}